\documentclass[11pt]{article} 
\usepackage[utf8]{inputenc} 

\usepackage[margin=1.25in]{geometry} 
\usepackage{graphicx} 
\usepackage{ upgreek }
\usepackage{booktabs} 
\usepackage{array} 
\usepackage{paralist} 
\usepackage{verbatim} 
\usepackage{mathrsfs} 
\usepackage{amssymb}
\usepackage{amsthm}
\usepackage{amsmath,amsfonts,amssymb}
\usepackage{esint}
\usepackage{graphics}
\usepackage{enumerate}
\usepackage{mathtools}
\usepackage{xfrac}
\usepackage{bbm}
\usepackage{subcaption}
\usepackage{stmaryrd} 

\let\oldsquare\square 

\usepackage{mathabx}

\renewcommand{\square}{\oldsquare}

 \usepackage{tocloft}
 
\usepackage[usenames,dvipsnames]{xcolor}
\usepackage[colorlinks=true, pdfstartview=FitV, linkcolor=blue, citecolor=blue, urlcolor=blue]{hyperref}
\usepackage[normalem]{ulem}

\numberwithin{equation}{section}

\newtheorem{theorem}{Theorem}[section]

\newtheorem{corollary}[theorem]{Corollary}
\newtheorem{proposition}[theorem]{Proposition}
\newtheorem{lemma}[theorem]{Lemma}
\theoremstyle{definition}
\newtheorem{definition}[theorem]{Definition}

\newtheorem{remark}[theorem]{Remark}

\let\originalleft\left
\let\originalright\right
\renewcommand{\left}{\mathopen{}\mathclose\bgroup\originalleft}
\renewcommand{\right}{\aftergroup\egroup\originalright}

\newcommand{\vertiii}{\vert\kern-0.3ex\vert\kern-0.25ex\vert}

\newcommand*{\N}{\ensuremath{\mathbb{N}}}

\newcommand*{\R}{\ensuremath{\mathbb{R}}}
\newcommand*{\Zd}{\ensuremath{\mathbb{Z}^d}}
\newcommand*{\Rd}{\ensuremath{\mathbb{R}^d}}
\newcommand{\eps}{\varepsilon}

\renewcommand*{\tilde}{\widetilde}
\renewcommand*{\hat}{\widehat}

\newcommand{\f}{\mathbf{f}}
\newcommand{\g}{\mathbf{g}}

\newcommand{\ep}{\eps}

\DeclareSymbolFont{boldoperators}{OT1}{cmr}{bx}{n}
\SetSymbolFont{boldoperators}{bold}{OT1}{cmr}{bx}{n}
\renewcommand{\a}{\mathbf{a}}

\newcommand{\ahom}{\bar{\a}}
\newcommand{\uhom}{\bar{u}}

\newcommand{\cu}{\square}
\newcommand{\F}{\mathcal{F}}
\renewcommand{\P}{\mathbb{P}}
\newcommand{\E}{\mathbb{E}}
\newcommand{\X}{\mathcal{X}}

\renewcommand{\O}{\mathcal{O}}

\newcommand{\Y}{\ensuremath{\mathcal{Y}}}

\newcommand{\indc}{1}

\DeclareMathOperator{\dist}{dist}

\DeclareMathOperator{\var}{var}
\DeclareMathOperator{\cov}{cov}

\DeclareMathOperator{\size}{sl}

\makeatletter
\newcommand\bigDiamond{\mathop{\mathpalette\bigDi@mond\relax}}
\newcommand\bigDi@mond[2]{%
	\vcenter{\hbox{\m@th
			\scalebox{\ifx#1\displaystyle 2\else1.2\fi}{$#1\Diamond$}%
	}}%
}
\newcommand\bigLozenge{\mathop{\mathpalette\bigL@zenge\relax}}
\newcommand\bigL@zenge[2]{%
	\vcenter{\hbox{\m@th
			\scalebox{\ifx#1\displaystyle 2\else1.2\fi}{$#1\blacklozenge$}%
	}}%
}
\makeatother

\newcommand{\avsum}{\mathop{\mathpalette\avsuminner\relax}\displaylimits}

\makeatletter
\newcommand\avsuminner[2]{%
  {\sbox0{$\m@th#1\sum$}%
   \vphantom{\usebox0}%
   \ooalign{%
     \hidewidth
     \smash{\,\rule[.23em]{8.8pt}{1.1pt} \relax}%
     \hidewidth\cr
     $\m@th#1\sum$\cr
   }%
  }%
}
\makeatother

\makeatletter
\newcommand\avsuminnerr[2]{%
  {\sbox0{$\m@th#1\sum$}%
   \vphantom{\usebox0}%
   \ooalign{%
     \hidewidth
     \smash{\,\rule[.23em]{6pt}{0.7pt} \relax}%
     \hidewidth\cr
     $\m@th#1\sum$\cr
   }%
  }%
}
\makeatother

\def\Xint#1{\mathchoice
{\XXint\displaystyle\textstyle{#1}}%
{\XXint\textstyle\scriptstyle{#1}}%
{\XXint\scriptstyle\scriptscriptstyle{#1}}%
{\XXint\scriptscriptstyle\scriptscriptstyle{#1}}%
\!\int}
\def\XXint#1#2#3{{\setbox0=\hbox{$#1{#2#3}{\int}$}
\vcenter{\hbox{$#2#3$}}\kern-.5\wd0}}

\def\fint{\Xint-}

\makeatletter 
\newcommand{\negphantom}{\v@true\h@true\negph@nt} 
\newcommand{\neghphantom}{\v@false\h@true\negph@nt} 
\newcommand{\negph@nt}{\ifmmode\expandafter\mathpalette 
  \expandafter\mathnegph@nt\else\expandafter\makenegph@nt\fi} 
\newcommand{\makenegph@nt}[1]{%
  \setbox\z@\hbox{\color@begingroup#1\color@endgroup}\finnegph@nt} 
\newcommand{\finnegph@nt}{%
  \setbox\tw@\null 
  \ifv@ \ht\tw@\ht\z@\dp\tw@\dp\z@\fi \ifh@\wd\tw@-\wd\z@\fi\box\tw@} 
\newcommand{\mathnegph@nt}[2]{%
  \setbox\z@\hbox{$\m@th #1{#2}$}\finnegph@nt} 
\makeatother

\newcommand{\rv}[1]{\textcolor{black}{#1}}

\usepackage{titlesec}

\newcommand{\NN}{\mathbb{N}}
\newcommand{\ZZ}{\mathbb{Z}}
\newcommand{\RR}{\mathbb{R}}
\newcommand{\al}{\alpha}

\newcommand{\ga}{\gamma}

\newcommand{\la}{\lambda}

\newcommand{\addperiod}[1]{#1.}
\titleformat{\section}
   {\normalfont\Large}{\thesection.}{0.5em}{}
\titleformat*{\subsection}{\normalfont\large}
\titleformat{\subsubsection}[runin]
  {\bfseries}
  {\thesubsubsection.}
  {0.5em}
  {\addperiod}
\titleformat*{\subsubsection}{\bfseries}
\titleformat*{\paragraph}{\bfseries}
\titleformat*{\subparagraph}{\large\bfseries}

\title{\bf \Large Quantitative homogenization and large-scale regularity of {P}oisson point clouds}

\author{Scott Armstrong
\thanks{Courant Institute of Mathematical Sciences, New York University.
{\footnotesize \href{mailto:scotta@cims.nyu.edu}{scotta@cims.nyu.edu}.}
}
\and 
Raghavendra Venkatraman
\thanks{Courant Institute of Mathematical Sciences, New York University.
{\footnotesize \href{mailto:raghav@cims.nyu.edu}{raghav@cims.nyu.edu}.}
}
}
\date{\today}

\usepackage[nottoc,notlot,notlof]{tocbibind}

\begin{document}

\maketitle

\begin{abstract}
We prove quantitative homogenization results for harmonic functions on supercritical continuum percolation clusters---that is, Poisson point clouds with edges connecting points which are closer than some fixed distance. We show that, on large scales, harmonic functions resemble harmonic functions in Euclidean space with sharp quantitative bounds on their difference. In particular, for every point cloud which is supercritical (meaning that the intensity of the Poisson process is larger than the critical parameter which guarantees the existence of an infinite connected component), we obtain optimal corrector bounds, homogenization error estimates and large-scale regularity results.
\end{abstract}

\setcounter{tocdepth}{2}  

%
%
\section{Introduction}
\subsection{Motivation and overview}
In this paper, we use quantitative homogenization methods to study the large-scale behavior of harmonic functions on random geometric graphs. We consider a model of a random geometric graph that is obtained by taking the Poisson process on~$\Rd$ with intensity~$\lambda>0$ and connecting points which are separated by a distance of at most one. The intensity~$\lambda$ is assumed to be larger then the percolation threshold which ensures an infinite connected component. This model is a standard example of a \emph{continuum percolation} model~\cite{MR} and goes by the names \emph{the Poisson blob model}, \emph{the Poisson Boolean model} and, in two dimensions, \emph{the Gilbert disk model}. It was introduced by Gilbert~\cite{Gil} as a mathematical model of wireless networks. 

\smallskip

Harmonic functions on random graphs have been of recent interest to the graph-based machine learning community, see for instance~\cite{CL,GGHS,CGL,CG22} and the references therein.
These papers typically consider a weighted random geometric graph built from an~i.i.d.~sample of points on a closed manifold (the so-called ``ground truth'') by joining points within a certain threshold distance. They seek to obtain quantitative large data limit theorems on the spectrum of the graph laplacian, towards the spectrum of the (weighted) Laplace-Beltrami operator on the manifold. Such theorems identify settings under which one can consistently make statistical inference (like classification or regression) with high probability theoretical guarantees.

These results consider the ``very supercritical'' regimes in which the density of points grows as the scale increases. This kind of assumption makes the geometric structure of the graph very simple, as one is very far from the percolation threshold and every vertex has a large number of edges. In this setting, one does not see \emph{homogenization}, but rather \emph{averaging}, and  quantitative estimates can be obtained by Taylor expansion.  The decorrelation and mixing properties of the random graph play less of a role. The goal of the present work is to demonstrate that quantitative homogenization methods are able to obtain sharper results which are applicable to much sparser graphs. 

\smallskip

Throughout, we fix~$d\in\N$ with~$d\geq 2$, a parameter~$\lambda\in (0,\infty)$ and we let~$\P$ be the probability measure corresponding to the Poisson point process (\rv{also referred to, in the sequel, as Poisson point cloud}) on~$\Rd$ with intensity~$\lambda$. One way of characterizing this is to consider a family~$\{ N_z \,:\, z\in\Zd\}$ of independent Poisson random variables with mean~$\lambda$ and then choosing~$N_z$ points in the cube~$z+ (\sfrac12,\sfrac12)^d$, sampled according to the uniform measure. Then~$\P$ is the law of this random point cloud on~$\Rd$. 
For each realization of the Poisson point cloud, we create a graph by connecting any pair of points separated by a Euclidean distance of less than one. 
Almost surely, this random graph will have an infinite number of bounded connected components, and at most one infinite connected component. Our aim is to study the behavior of harmonic functions on ``large'' connected components of this graph, which we refer to as the \emph{Poisson blob model}.
This model has large-scale connectivity properties qualitatively similar to bond percolation on~$\Zd$ in dimension~$d\geq 2$, with the parameter~$\lambda$ playing the role of the probability~$p$ of an open edge in the latter. Indeed, it is well-known (see~\cite{MR}) that there exists a critical value~$\lambda_c(d)>0$ such that, with probability one, the graph has a unique unbounded connected component if~$\lambda > \lambda_c$ (the supercritical regime) and no unbounded connected component if~$\lambda< \lambda_c$ (the subcritical regime). 

\smallskip

Our primary interest is in the behavior of harmonic functions on the infinite cluster in the supercritical case when~$\lambda>\lambda_c$.
This random graph may of course exhibit very irregular, non-Euclidean behavior on small scales, particularly if~$\lambda$ is close to~$\lambda_c$, but on large scales it behaves rather like the ambient Euclidean space~$\Rd$ in which it is embedded. We should therefore expect that harmonic functions on the cluster behave, on large scales, like harmonic functions on~$\Rd$.

\smallskip

This picture fits squarely within the realm of homogenization theory, with the  irregular local geometry of the graph, rather than an oscillating coefficient field, creating the small-scale heterogeneities which must be homogenized. Unfortunately, the lack of uniformity of the geometry of the random graph on small scales means that the equation we are homogenizing is degenerate in a very strong sense, and thus the classical methods of qualitative stochastic homogenization are not immediately applicable. The lack of uniform ellipticity also limits the applicability of certain quantitative approaches to homogenization (such as~\cite{GO,GNO}).

\smallskip

The aim of this paper is to demonstrate that the coarse-graining methods developed in~\cite{AS,AKM19} in the context of quantitative stochastic homogenization for elliptic equations in the continuum are well-suited to the study of harmonic functions on random graphs, and in particular to continuum supercritical percolation clusters. These methods are flexible because, similar to renormalization group arguments, they do not look too closely at the small scales where the disorder is strongest. In fact, they are indifferent to the precise nature of the microscopic medium (e.g., whether it is a finite difference equation on a discrete graph or a PDE in the continuum). One just needs that the environment, in the case of a discrete random graph, \rv{resembles} Euclidean space and \rv{is} ``elliptic'' in a suitable coarse-grained sense. In the case of supercritical continuum percolation, we can get these estimates relatively easily from well-known results in the literature: in particular, we use ``renormalization'' estimates developed in the setting of bond percolation in~\cite{GM,AP} and extended to continuum percolation in~\cite{Penrose}.

\smallskip

This point has, in fact, already been made in the work~\cite{AD16}, which adapted the ideas of~\cite{AS,AKM19} to study harmonic functions on supercritical bond percolation clusters. That paper used the renormalization estimates of~\cite{AP} to find a good scale on which the supercritical cluster was well-behaved, and then ran the arguments of~\cite{AKM19}. Our paper adapts the strategy of~\cite{AD16} to the case of continuum percolation, which is conceptually quite similar. The arguments here are however simpler than those of~\cite{AD16} and clarify that the methods used are not restricted to special geometries (such graphs which are subsets of the lattice~$\Zd$).

\smallskip

The main results, which are stated precisely in the next subsection, are roughly summarized as follows:

\begin{itemize}
	
	\item In Theorem~\ref{t.large-scale.C01.intro}, we present a large-scale regularity result which implies that the finite difference of a harmonic function across a typical edge in the graph is bounded by a multiple of the (suitably normalized)~$L^2$ oscillation of the function on a macroscopic scale. This is analogous to the classical~$C^{0,1}$ estimate for harmonic \rv{functions}, and such an estimate was first proved in the context of stochastic homogenization in~\cite{AS}. 
	
\item In Theorem~\ref{t.corrector.bounds} we give optimal estimates on the first-order correctors, which depend on dimension and match the scalings found in the uniformly elliptic setting.

	\item In Theorem~\ref{t.homogenization.intro}, we give an explicit, quantitative estimate on the (random) minimal length scale at which homogenization begins to be seen; on scales above this random scale, we show that the difference of harmonic functions on the \rv{supercritical} continuum percolation cluster and harmonic functions on~$\Rd$ becomes small at the optimal rate.

\end{itemize}

The proof of Theorem~\ref{t.large-scale.C01.intro} is based on the strategy of~\cite{AS,AD16}. We define suitable coarse-grained coefficients and perform an iteration up the scales to demonstrate convergence with an algebraic rate. This sub-optimal rate of homogenization allows us to approximate arbitrary solutions by harmonic functions sufficiently well that we obtain an improvement of regularity on large scales, leading to the statement of Theorem~\ref{t.large-scale.C01.intro}. 

\smallskip

The large-scale ~$C^{0,1}$ estimate can be combined with concentration inequalities to obtain optimal bounds on the first-order correctors, leading to the statement of Theorem~\ref{t.corrector.bounds}. Finally, Theorem \ref{t.homogenization.intro} is obtained by using these corrector estimates and a classical two-scale expansion-type argument. 

\smallskip

There have several other recent works studying harmonic functions on Poisson point clouds. 
Braides and Carroccia~\cite{BC} obtained qualitative homogenization results in two dimensions and Faggionato~\cite{Faggio} proved qualitative homogenization results in all dimensions. Both of these papers use arguments which do not give quantitative information and thus cannot access finer properties (such as large-scale regularity and error estimates).

\smallskip 
	
The arguments in the present paper can be generalized far beyond the particular setting of Poisson point clouds. One may consider more general random graphs, with longer correlations and with long-range edges and (possibly random) conductivities which are neither bounded uniformly from above nor from below. Such features will increase the complexity of the proof, and may require first adapting the works of~\cite{AP,Pisztora} to the particular random graph under consideration, but will  not alter the overall strategy of the homogenization/renormalization approach. As an application of our results, in a subsequent  paper~\cite{AV3} we obtain optimal convergence rates for the spectrum of the graph Laplacian for the Dirichlet problem in Euclidean domains. This improves on the results of~\cite{BN,Singer,BIK,shi2015convergence,GGHS,Dobson,WR,CG22,CGL,Lu}, allowing sparse graphs down to the percolation threshold, which is made possible  by the homogenization estimates in this paper.

\subsection{Problem formulation and statements of the main results}
\label{ss.theorems}

We begin with some notation. 
Throughout, we work in dimension~$d \geq 2$.
A \emph{cube} is a set of the form 
\begin{equation*}
	\Zd \cap \bigl(z + [ 0,N)^d \bigr)\,,  \quad z \in \Zd, N \in \N\,. 
\end{equation*}
The \emph{center} of the cube $\cu$ in the preceding display is $z$, and its \emph{size} \rv{or side length}, denoted $\size(\cu) := N.$ 
A \emph{triadic cube} is a cube of the form 
\begin{equation*}
	\cu_m(z) := \Zd \cap \Bigl( z + \Bigl[-\frac12 3^m, \frac12 3^m\Bigr)^d \Bigr)\,, \quad z \in 3^m \Zd, \ m \in \N\,. 
\end{equation*}
When $z = 0,$ we will simply write $\cu_m.$ \rv{Sometimes we will consider~$\cu_m$ to be a subset of Euclidean space; that is, we will also use~$\cu_m$ to denote~$[-\frac12 3^m, \frac12 3^m)^d$ when no confusion arises. }

\rv{The notation~$B_R(x),$ as usual, denotes the open Euclidean ball centered at~$x,$ of radius~$R.$ When~$x=0$, we simply write~$B_R.$ }

It is convenient to use the $\O_s$ notation introduced in \cite{AKM19} to measure the sizes and stochastic integrability of random variables. Given $s,\theta > 0$ and a nonnegative random variable $\X,$ we write 
\begin{equation*}
	\X = \O_s(\theta) \quad \iff \quad \E \Bigl[ \exp \Bigl( \frac{\X}{\theta}\Bigr)^s\Bigr] \leqslant 2\,. 
\end{equation*}
By Markov's inequality, 
\begin{equation*}
\X = \O_s(\theta) \quad \implies \quad	\P[ \X \geqslant \theta t] \leqslant 2 \exp (-t^s)\,, \ \forall t  > 0.
\end{equation*}

As mentioned above, we work with a Poisson point process~$\eta$ on~$\Rd$ with intensity~$\lambda > 0.$ It is well known that there is a critical~$\lambda_c(d) > 0$ depending only on dimension such that if~$\lambda > \lambda_c(d),$ then the graph built by connecting points within unit distance, has a unique unbounded connected component (and no such unbounded connected component if~$\lambda < \lambda_c$). 

\rv{Our definition of the graph Laplacian requires identifying a well-connected component of the percolation cluster that localizes and has well-controlled Euclidean geometry on large-scales. Definition~\ref{d.wellconnected} identifies \emph{well-connected cubes} $\cu,$ identified by the property of having a unique connected component of~$\eta \cap \cu$ with diameter at least $\tfrac{1}{100}\size(\cu):$ this component is denoted by~$\eta_*(\cu).$ In addition, in a well-connected cube, the number of points in~$\eta_*(\cu)$ is \rv{comparable to} the volume of the cube, and finally, each point in the cube $\cu$ has a neighbor in~$\eta_*(\cu)$ no further than distance $\tfrac{1}{100}\size(\cu)$ away. }

\smallskip

\rv{
	\begin{definition} \label{d.wcdomain}
		Given a Borel set~$U \subseteq \Rd,$ we let~$\eta_*(U)$ denote the largest connected component of~$\eta_*(\cu) \cap U$,
		where $\cu$ is the smallest triadic cube containing~$U$ such that~$\cu$ is well-connected in the sense of Definition~\ref{d.wellconnected}. 
	\end{definition}
The cluster~$\eta_*(U)$ is well-defined, thanks to the notion of well-connected cubes in Definition~\ref{d.wellconnected}. It also agrees with~$\eta_*(\cu)$ when~$U = \cu$ is a triadic cube. Finally, we set~$\eta_*:= \eta_*(\Rd).$} Results on the geometry of continuum percolation clusters studied in~Penrose~\cite{Penrose} and Penrose and Pisztora~\cite{PP} imply that large cubes are well-connected except on an event with probability which is exponentially small in the size of the cube---see Proposition~\ref{p.Penrose} below for the precise statement.
These works were, in turn, based on previous  works of Pisztora~\cite{Pisztora} and Antal and Pisztora~\cite{AP} which provided analogous estimates for \rv{supercritical} percolation clusters on~$\Zd$.  Proposition~\ref{p.Penrose} tames the large-scale geometry of the \rv{supercritical} percolation cluster. It tells us essentially that, on large scales, the geometry is not very different from Euclidean space. By combining Proposition~\ref{p.Penrose} with very simple cube decompositions, we will obtain Poincar\'e and Sobolev-type functional inequalities on all sufficiently large subsets of the random graph; these imply compactness properties and are obviously quite useful for studying harmonic functions on the graph. 

\smallskip

If we zoom out to large scales on which the geometry of the random graph is tame, then we can implement the quantitative homogenization methods of~\cite{AKM19,AK22}. This is the same strategy as the one previously used in~\cite{AD16} to obtain quantitative results for supercritical bond percolation on the lattice~$\Zd$. 

\smallskip

If $U$ is a Borel set and $f : \eta_*(U) \to \R$ is a function, then we write 
\begin{equation*}
	(f)_U \equiv \avsum_{x \in \eta_*(U)} f(x) :=  \frac{1}{|U|} \sum_{x \in \eta_*(U)} f(x)\,.
\end{equation*}
 By a convenient abuse of notation we also write $|U|$ for the Lebesgue measure of the set $U$. 
We will also need to work with coarse-grained quantities which are functions on a lattice. In such situations, we will conflate the notation for average and write 
\begin{equation*}
	\avsum_{x\in \Zd\cap \cu} f(x) = \frac{1}{|\cu \cap \Zd|} \sum_{x\in \Zd\cap \cu} f(x)\,,
\end{equation*}
with~$|\cu \cap \Zd|$ denoting cardinality. The meaning of~$|\cdot|$ (Lebesgue measure versus counting measure) will be clear from the context.
For a function in the continuum, $f: U \to \R,$ we write, as usual, 
\begin{equation*}
	\fint_U f(x)\,dx := \frac{1}{|U|} \int_Uf(x)\,dx\,. 
\end{equation*}
Given $f: \eta_*(U) \to \R,$ we write 
\begin{equation*}
	\|f\|_{\underline{L}^2 (\eta_*(U))} :=  \Biggl( \avsum_{x \in \eta_*(U)} |f(x)|^2 \Biggr)^{\sfrac12}\,. 
\end{equation*}
Given $x,y \in \eta_*$ we write $x \sim y$ if $|x-y| \leqslant 1$ (as measured in Euclidean distance). We write 
\begin{equation*}
	\|f\|_{\underline{H}^1(\eta_*(U))} :=  \Biggl( \frac1{|U|}\sum_{x,y\in \eta_*(U),\,x\sim y}  (f(x) - f(y))^2  \Biggr)^{\sfrac12}\,. 
\end{equation*}
Obviously, the $\underline{H}^1(\eta_*(U))$ seminorm comes with a natural inner product, namely for two functions $f,g: \eta_*(U) \to \R,$ we set 
\begin{equation*}
	\langle f, g \rangle _{\underline{H}^1(\eta_*(U))} :=  \frac1{|U|}\sum_{x,y\in \eta_*(U),\,x\sim y}  (f(x) - f(y))(g(x) - g(y))  \,. 
\end{equation*}

We next introduce the technical framework in which to precisely formulate our problem. We let~$\Omega$ denote the collection of locally finite, pure point Radon measures. These are the measures~$\eta$ having the form
\begin{equation*}
	\eta = \sum_{i\in\N} \delta_{x_i}\,,
\end{equation*}
where~$\{ x_i \}_{i\in\N}\subseteq\Rd$ is a disjoint sequence with no accumulation points and~$\delta_x$ denotes the Dirac measure at a point~$x\in\Rd$. 
For each Borel subset~$\cu_m \subseteq \Rd$, we define 
the~$\sigma$-algebra~$\mathcal{F}(\cu_m)$ to be the one generated by the family of random variables
\begin{equation*}
	\eta \mapsto \int_{\Rd} \eta(x) \varphi(x) \,dx\,, \quad 
	\varphi\in C^\infty_c(\cu_m)\,.
\end{equation*}
The largest of these is~$\mathcal{F}:= \mathcal{F}(\Rd)$.
We introduce the translation group~$\{T_y \}_{y\in\Rd}$ which acts on~$\Omega$ by
\begin{equation*}
	T_y \eta := \eta(\cdot+y)\,.
\end{equation*}
We say that a probability measure~$\P$ on~$(\Omega,\F)$ is~$\Zd$-stationary if
\begin{equation*}
	\P \circ T_z  =\P \,, \quad \forall z\in\Zd\,.
\end{equation*}
Throughout the paper,~$\P$ denotes the law of the Poisson process on~$\Rd$ with intensity $\lambda > \lambda_c(d)$. This is a probability measure on~$(\Omega,\F)$ which is characterized by the following properties: 
	\begin{itemize}
		\item The probability that a Borel set $B\subseteq \RR^d$ has $k \in \NN$ points, i.e., 
		\begin{equation*}
			\P\bigl [\eta(B) = k \bigr ] = \frac{(\ga |B|)^k}{k!}e^{-\ga |B|}\,;
		\end{equation*}
		\item For any finite collection~$\{B_j\}_{j=1}^N$ of pairwise disjoint Borel subsets, the random variables $\{\eta(B_j)\}_{j=1}^N$ are $\P$--independent.  
	\end{itemize}
	
Given a typical realization~$\eta$ of the Poisson point process, an (unweighted) graph is obtained by joining points~$x,y \in \eta$ that are within distance~$1$ of each other, with an edge with unit weight. The graph laplacian associated to this graph, applied to a function~$u: \eta_*(U) \to \R$ is given by 
\begin{equation} \label{e.Lapdef}
	\mathcal{L}u(x) := \sum_{y \in \eta_*(U),\,y\sim x }  (u(y) - u(x))\,.
\end{equation}

Our theorem statements below can be easily extended to significantly more general edge weights, such as those given by a compactly supported kernel, or Gaussian kernels. Our choice of the~$0$ or~$1$ edge-weights is merely for simplicity.

In the sequel we will refer to functions $u$ as being \emph{graph harmonic} in a domain $U \subseteq \Rd$ if they satisfy $\mathcal{L} u (x) = 0$ for each $x \in \eta_*(U).$ Our goal in this paper is to understand the quantitative \emph{large-scale} behavior of graph harmonic functions.

\smallskip 

We next give the precise statements of the main results. We prove three main theorems concerning the large-scale behavior of solutions to the equation $\mathcal{L}u = f$. 

The first result concerns the regularity of solutions on large scales. It can be compared to the classical pointwise (or~$C^{0,1}$) estimate for harmonic functions. The latter states that there exists~$C(d)<\infty$ such that, for any harmonic function~$u$ on~$B_R$, we have
\begin{equation*}
| \nabla u(0) | \leq \frac{C}{R} \| u - (u)_{B_R} \|_{\underline{L}^1(B_R)}\,.
\end{equation*}
This is essentially equivalent \rv{(by the Lebesgue differentiation theorem and the translation invariance of the Laplacian)} to the statement that 
\begin{equation}
\label{e.harmonicC01}
\sup_{t \in (0,R/2)} 
\| \nabla u \|_{\underline{L}^2(B_t)}
\leq
\frac{C}{R} \| u - (u)_{B_R} \|_{\underline{L}^1(B_R)}\,.
\end{equation}
We will prove an estimate like~\eqref{e.harmonicC01} for solutions on our random graph, when~$R$ is large. Since the small-scale geometry of our graph is wild, we cannot expect to have an estimate~\eqref{e.harmonicC01} unless we make a restriction to remove the small scales; this means that the supremum on the left side should have a restriction that~$t$ be larger than some \emph{random scale}, which is the scale at which the random graph is ``well-behaved'' in the sense of being close to homogenized. 
The precise statement is the following.

\begin{theorem}[Large-scale $C^{0,1}$ regularity] \label{t.large-scale.C01.intro} 
There exist constants~\rv{$C(d,\lambda), s(d) \in (0,\infty)$} and a nonnegative random variable~$\X$ satisfying 
\begin{equation}
\label{e.X.integrability}
\X = \O_{s}(C) 
\end{equation}
such that, for every $R \geqslant \X,$  every $f: \eta_*(B_R) \to \R$, and solution~$u$ of~$\mathcal{L}u=f$ in~$\eta_*(B_R)$,  we have the estimate
	\begin{equation} \label{e.gradbound.lsr.intro}
		\sup_{t \in [r, \frac12 R]} \rv{\biggl( \avsum_{x,y\in \eta_*(B_t),\,x\sim y} |u(x)-u(y)|^2 \biggr)^{\sfrac12} }\leqslant \frac{C}{R}\|u - (u)_{B_R}\|_{\underline{L}^2(\eta_*(B_R))} + C \int_{\rv{\X}}^R \|f \|_{\underline{H}^{-1}(\eta_*(B_t))} \frac{\,dt}{t}\,.
	\end{equation}
	
\end{theorem}

The statement of Theorem~\ref{t.large-scale.C01.intro} can be generalized to higher regularity, beyond~$C^{0,1}$, to all orders of regularity. In particular, we can obtain Liouville theorems which classify all global, polynomially growing solutions. The precise statements of these results are essentially the same as the ones found in~\cite[Theorem 3.8]{AKM19} or~\cite[Theorem 2]{AD16}, and the proof also follows along the same lines.

An important consequence of the above large-scale~$C^{0,1}$ estimate are estimates on the growth of the first-order correctors. As we will see, there exists a family of random fields~$\{ \phi_e: e \in \Rd\},$ defined on the unique infinite percolation cluster~$\eta_*(\Rd)$, which are stationary modulo additive constants and satisfy
\begin{equation*}
\mathcal{L} ( \ell_e + \phi_e ) = 0 \quad \mbox{on} \ \eta_*(\Rd)\,.
\end{equation*}
In the next theorem, we give optimal quantitative estimates on the correctors, presented as estimates in negative Sobolev norms of their gradients. 
Here the quantity~$[w]$ denotes the ``coarsened version'' of a function~$w$ defined on the graph. The former is defined on~$\Zd$ by essentially taking its values from those of~$w$ at nearby points on the graph (see~\eqref{e.Qcoarse}). Note that the difference between~$w$ and~$[w]$ is well-controlled by Poincar\'e inequalities on the cluster~$\eta_*$, and so bounds on~$[w]$ imply corresponding bounds for~$w$, but the former are easier to work with (see~\eqref{e.correctorsize.opt.nograd} below).

\begin{theorem}[Optimal corrector estimates]
	\label{t.corrector.bounds}
	There exists constants~\rv{$c(d,\lambda ),s(d) \in (0,\infty)$} such that,
	for every~$r\in [1,\infty)$,
	\begin{equation*}
	\biggl |
		\int_{\Rd}
		\nabla[\phi_e] (x) 
		\Psi_r (x)
		\biggr |
		=
		\O_{s} ( Cr ^{-\frac d2} )\,,
	\end{equation*}
\rv{where~$\Psi_r := r^{-d} \Psi(\cdot/r),$ and~$\Psi$ is any nonnegative, compactly supported smooth function with~$\int_{\Rd} \Psi = 1.$  }
	Consequently, for any $\mathfrak{s}>0,$ and any $\varepsilon \in (0,\frac12],$ we have 
	\begin{equation} \label{e.correctorsize.opt}
		\Bigl\|\nabla[\phi_e]\Bigl(\frac{\cdot}{\varepsilon}\Bigr)\Bigr\|_{H^{-\mathfrak{s}}(B_1)} \leqslant 
		\begin{cases}
			\O_{s} (C\varepsilon^\mathfrak{s})  & \mbox{if} \ \mathfrak{s} <\frac{d}2\,,
			\\
			\O_s\bigl( C\ep^{\frac{d}2} |\log \ep|^{\frac12}\bigr) &  \mbox{if} \ \mathfrak{s} = \frac{d}2\,,\\
			\O_s\bigl( C \ep^{\frac{d}2} \bigr) &\mbox{if} \  \mathfrak{s} > \frac{d}2\,,
		\end{cases}
	\end{equation}
	for a constant~$C(d,\lambda)<\infty$. 
\end{theorem}

\begin{remark}
\rv{We note that in~\cite{Dario_AAP}, the authors derive similar optimal corrector estimates on the supercritical bond percolation cluster using a similar strategy as we use in the proof of the preceding theorem.  }
\end{remark}

With the corrector bounds in hand, we can draw a number of corollaries with sharp convergence rates, following the corresponding results in the continuum homogenization theory. Here we present an optimal homogenization error estimate for the Dirichlet problem. For its statement, we need some notation. We let $U_0 \subseteq \Rd$ denote a bounded $C^{1,1}$ or convex domain, and we let 
\begin{equation*}
	U_m := 3^m U_0\,. 
\end{equation*}
By rescaling suitably, we may assume that~$U_0 \subseteq \cu_0$ so that~$U_m \subseteq \cu_m$. \rv{Recall that for any function~$u : \eta_*(U_m) \to \RR$ we define 
\begin{equation*}
	\mathcal{L} u (x) := \sum_{y \in \eta_*(U_m),\,y\sim x}  (u(y) - u(x)) \,. 
\end{equation*}
}

\begin{theorem}[Homogenization error estimate]
	\label{t.homogenization.intro} 
	Suppose~$\lambda > \lambda_c$.
	There \rv{exist} a scalar matrix~$\ahom$ depending only on $\lambda$, constants~$C(d,\lambda),s(d) > 0$ 
	and a nonnegative random variable~$\X$ satisfying 
	\begin{equation*}
	\X = \O_s(C) \,,
	\end{equation*}
	such that, for every~$m \in \N$ with~$3^m > \X$ and every function~\rv{$\uhom \in C^2(U_m) \cap H^1_0(U_m)$}, the Dirichlet problem 
	\begin{equation*}
		\left\{
		\begin{aligned}
			& \mathcal{L} u = \nabla \cdot \ahom \nabla \uhom & \mbox{on} & \ \rv{\eta_*( U_m) \setminus \{ x \in U_m : \dist(x,\rv{\partial U_m}) <2 \}} \,,
			\\ &
			u = 0 & \mbox{on} & \ \rv{\eta_* (U_m) \cap \{ x \in U_m : \dist(x,\rv{\partial U_m}) <2 \}} \,,
		\end{aligned}
		\right.
	\end{equation*}
	for the graph Laplacian has a unique solution~\rv{$u:\eta_*(U_m) \to \R$} which satisfies, for~$C(d,\lambda,U_0)<\infty$ which depends also on~$U_0$,
	\begin{equation}
\| u - \uhom \|_{{L}^2(\eta_*(U_m))} 
		\leq C\| \uhom \|_{{L}^2(U_m)} \cdot 
		\begin{cases}
3^{-m} 
 & \mbox{ if } d \geqslant 3\,, \\
m^{\sfrac12} 3^{-m}  & \mbox{ if } d = 2 \,.
		\end{cases}	\end{equation}
\end{theorem}
\rv{
	\begin{remark}
	As the function~$\uhom \in C^2(U_m)$, the function~$\nabla \cdot \ahom \nabla \uhom$ is continuous, and therefore its restriction to~$\eta_*(U_m)$ is well-defined.
\end{remark}
\begin{remark}
	We note from \cite{AKM19} that these rates of converge match with optimal rates in quantitative continuum stochastic homogenization. 
\end{remark}
}
The dependence in the length scale~$3^m$ of our estimates are optimal, as they agree with the bounds obtained in continuum homogenization theory. 
The parameter~$s(d)>0$ is a small constant which we do not estimate. Identifying the optimal exponent~$s$, which characterizes the stochastic integrability of our estimates, would require finer versions of the percolation estimates of~\cite{Penrose,PP}; on the homogenization side, we would also need to revisit the proof of Theorem~\ref{t.homogenization.intro}, and replace the argument based on nonlinear concentration inequalities, which lose stochastic integrability, with the renormalization approach of~\cite{AK22}. The latter method uses only linear concentration inequalities and therefore does not lose such quantitative information. This level of precision is beyond the scope of the present paper. 

\smallskip

A version of Theorem~\ref{t.homogenization.intro} with macroscopic dependence in the coefficients, or for point clouds on a manifold, can be obtained by slightly varying the proof. This kind of generalization does not require sophisticated arguments. In fact, we can use essentially the same arguments as are used in periodic homogenization. We just need the corrector bounds with the macroscopic parameter frozen (which is already given by Theorem~\ref{t.corrector.bounds}), and suitable assumption of continuity in the macroscopic variable. 
Other kinds of homogenization estimates, such as those for the homogenization of the parabolic (or elliptic) Green functions, can also be obtained by combining the corrector estimates of Theorem~\ref{t.corrector.bounds} with standard arguments from classical homogenization theory, see for instance~\cite[Chapter 8]{AKM19}.

\section*{Acknowledgements}
S.A. acknowledges support from the NSF grants DMS-1954357 and DMS-2000200. R.V. acknowledges support from the Simons foundation through award number~733694, and an AMS-Simons award.

\section{Large-scale geometry of random graph} 
\label{s.large-scale}

In this section, we perform the first renormalization step of taming the geometric structure of our random graph. 
This essentially amounts to quoting some estimates in the percolation literature and putting them in a form suitable for our use in this paper. A similar step was performed in the case of bond percolation in~\cite{AD16}, and what is presented here is close to Section~3 of that paper.

\subsection{Renormalization estimates in supercritical percolation}
The result we need states roughly that, in a sufficiently large cube, the random graph has a unique large cluster. The minimal size of the cube is random, but its stochastic moments are strongly controlled. In the case of supercritical bond percolation, these ``cluster estimates'' can be found in the work of Pisztora~\cite{Pisztora} and Antal and Pisztora~\cite{AP}, and so the paper~\cite{AD16} relied on these works. 
There are completely analogous results for continuum percolation, which were proved around the same time by Penrose~\cite{Penrose} and Penrose and Pisztora~\cite{PP} using very similar arguments.

\smallskip

In dimensions three and higher, the arguments in all of these papers (for both bond and continuum percolation) are based on the earlier work of Grimmett and Marstrand~\cite{GM} on percolation estimates in slabs. Their results were formalized in the setting of bond percolation, and later extended to continuum percolation in Tanemura~\cite{Tanemura}. 
In two dimensions the slab estimates are not valid, but in this case the supercritical ``large cluster'' estimates can be proved in a different and much simpler way by a reduction to subcritical estimates for the dual graph. 
The proof of the slab estimate has a qualitative step in the argument which is not easy to quantify (and thus dependence on the percolation parameters is lost). However, a quantitative version (which is still far from sharp) was recently proved in~\cite{DKT}. We also mention the very recent work~\cite{CMT} in which the cluster estimates are proved using a new argument which avoids the slab estimates, does not distinguish between~$d=2$ and~$d>2$, and leads to quantitative estimates of the same order as in~\cite{DKT}. 

\smallskip

The precise result we will quote here was proved in Penrose~\cite[Proposition 3, page 105]{Penrose}.
Before giving the statement, we make a definition of ``well-connected cube,'' which is roughly a cube in which there exists a single large connected component of the graph and the number of points in the graph is close to the volume times the density of the infinite cluster.

\begin{definition}[Well-connected cube]
\label{d.wellconnected}
Let~$r \in \N$ and~$\cu = z+r\cu_0$ be a cube with side length~$r$. We say that~$\cu$ is a \emph{well-connected cube} and write~$\cu \in \mathrm{WC}$ if:
\begin{enumerate}

\item There is a unique connected component of~$\eta \cap \cu$ with diameter larger than~$\frac{1}{100} r$. \emph{\rv{We denote this by~$\eta_*(\cu)$.}}

\item The number of points in the cluster~$\eta_*(\cu)$ satisfies
\begin{equation*}
\frac{99}{100} \theta (\lambda)
\leq \frac{| \eta_*(\cu) |}{|\cu|}
\leq \frac{101}{100} \theta (\lambda) \,,
\end{equation*}
with~$\theta(\la)$ denotes the density of the infinite cluster.

\item $\sup_{x\in \cu} \dist(x, \eta_*(\cu)) \leq \frac{1}{100} r$.

\end{enumerate}
\end{definition}

The following proposition bounds the probability of the event that a large triadic cube is not well-connected.

\begin{proposition}[{\cite[Proposition 3]{Penrose}}]
\label{p.Penrose}
There exists~$c(d,\lambda) > 0$ such that
\begin{equation}
\label{e.Penrose}
\P\bigl [ \cu_n \not\in \mathrm{WC} \bigr ]
\leq 
\exp \bigl ( -c 3^n \bigr )\,, \quad \forall n\in \N\,.
\end{equation}
\end{proposition}

The statement of~\cite[Proposition 3]{Penrose} is slightly different from the one given above. First, a superficial difference is that we have scaled the estimate differently. \rv{ Second, the estimate in~\cite[Proposition 3]{Penrose} is for a slightly different point process (obtained by essentially specifying the number of points in a region larger than the cube~$\cu_n$). It is however easy to recover the statement of Proposition~\ref{p.Penrose} from~\cite[Proposition 3]{Penrose}, because the two point processes have the same distribution if we condition on the number of points in the cube (which follows a Poisson distribution in our case and a binomial distribution in the setting of~\cite{Penrose}). }

\smallskip

The estimate~\eqref{e.Penrose} is sharp in the sense that the probability can also be lower bounded by~$\exp \bigl ( -C 3^n \bigr )$ for some other constant~$C(d,\lambda)<\infty$. This is the probability of seeing a long ``finger'' proportional to the size of the cube, that is, a connected blob of length longer than~$r/50$ which has diameter of order one that is surrounded by an empty region (and thus not connected to the main cluster). 
If one slightly modifies the first point in the definition of well-connected cube to require uniqueness only of connected components with at least~$\frac{1}{100}r^d$ many points---rather than with diameters larger than~$\frac{1}{100}r$---then the right side of the estimate in~\eqref{e.Penrose} can be improved to~$\exp \bigl ( -c 3^{n(d-1)} \bigr )$. This \emph{surface-order} type large deviations estimate is the main result of~\cite{PP}, and it is also sharp.

\smallskip

It would presumably be possible to use the latter estimate, rather than the one stated in Proposition~\ref{p.Penrose}. This would, we expect, yield versions of our main results with sharp stochastic integrability. However, it is more convenient to use the notion in Definition~\ref{d.wellconnected} because dealing with the presence of long fingers is technical, requiring a much less simple presentation. 

\smallskip

\subsection{Minimal scales for the Poincar\'e inequality}

Proposition~\ref{p.Penrose} provides a random length scale above which our random graph behaves more or less like Euclidean space~$\Rd$. Below this scale, the geometry of the graph can be very complicated, so we will avoid looking too closely at these small scales.
Our homogenization procedure will begin on scales strictly larger.

\smallskip

In this subsection, we will zoom out a bit more and find another random scale above  which~(i)~the Poincar\'e inequality is valid, and~\rv{(ii)}~affine functions have bounded energy.
The idea in the case of the Poincar\'e inequality is to follow the usual proof in the Euclidean setting, but ignoring all scales below the scale at which the cluster is well-connected. This is the same idea as the one in the proof of~\cite[Lemma 3.2]{AD16}, which considered the bond percolation case. It is formalized by introducing the notion of a \emph{coarsened function}, defined below, and uses a partitioning scheme developed in~\cite[Proposition 2.1]{AD16}.

\smallskip

\begin{lemma}
\label{l.Poincare}
There exists \rv{$C(\lambda,d) > 0$} and a random scale $\mathcal{M} > 0$ satisfying the bound 
\begin{equation} \label{e.minscalesize}
	\mathcal{M} = \O_{\frac{d}{2(d+1)}} (C),
\end{equation}
such that the following holds for every~$3^m > \mathcal{M}$:
\begin{itemize}

\item The cube~$\cu_m$ belongs to~$\mathrm{WC}$ as defined in Definition~\ref{d.wellconnected}.

\item There exists a constant~$C(d,\lambda) > 0,$  such that for any function~$u : \eta_*(\cu_{m+1}) \to \R,$ we have 
\begin{equation} \label{e.main_Poincare}
	\biggl( \avsum_{x \in \cu_m} |u(x) - (u)_{\cu_m}|^2 \Biggr)^{\!\!\sfrac12} \leqslant C3^m \biggl( \sum_{\rv{y,y^\prime \in \eta_*(\cu_m), y \sim y^\prime}} |u(y) - u(y^\prime)|^2 \Biggr)^{\!\!\sfrac12}\,. 
\end{equation}

\item There exists a constant $K(d,\lambda)> 0$ such that
\begin{equation} \label{e.affineenergy}
\frac{1}{|\cu_m|} \sum_{\rv{x,y \in \eta_*(\cu_m), x \sim y}} |x-y|^2 \leqslant K\,.
\end{equation}
\end{itemize}
\end{lemma}

\begin{proof}
First, by applying \cite[Proposition 2.1]{AD16} \rv{(take~$\mathcal{G} = \mathrm{WC}$)}, for any~$m\in\N$ there exists a partition~$\mathcal{P}$ of~$\Zd$ into triadic cubes \rv{of the form~$z+ \cu_n$ with~$z\in 3^n\Zd$} such that every element~$\cu\in\mathcal{P}$ satisfies~$\cu\in \mathrm{WC}$ (see Definition \ref{d.wellconnected}), and 

\begin{equation*}
	\cu,\cu'  \in \mathcal{P} \ \mbox{and} \ \dist(\cu,\cu') \leq 1 
	\implies 
	\frac{\size(\cu)}{\size(\cu')} \in \{ \sfrac13, 1,3\} 
	\,.
\end{equation*}
Moreover, by~\cite[Proposition 2.1, (iii)]{AD16} for any~$x \in \rv{\cu_m},$ if~$\cu_x \in \mathcal{P}$ denotes the unique cube in the partition that contains~$x,$ then there is a deterministic constant~$C > 0$ such that 
\begin{equation} \label{e.sizeofcubes}
\size(\cu_x) = \O_1(C)
\end{equation}
\rv{and, if~$3^m$ is larger than the element of~$\mathcal{P}$ containing the origin, then~$\cu_m$ belongs to~$\mathrm{WC}$. We may thus restrict~$\mathcal{P}$ to be a partition of~$\cu_{m+1}$. }

\smallskip 

\emph{Step 1.}
Fix~$\cu \in \mathcal{P}$. Let~$\tilde{\cu}$ be the union of~$\cu$ and its neighboring cubes in~$\mathcal{P}$. 
For each~$x,x'\in \eta_*(\cu)$, we can find a path joining~$x$ to~$x'$ such that the graph length of the path is at most the number of points of~$\eta_*(\tilde{\cu})$, which is at most 
\begin{equation*}
\frac{101 \theta(\lambda) |\tilde{\cu}| }{100}
\leq 
2\cdot 6^d \theta(\lambda) |\cu| 
\,.
\end{equation*}
Therefore, 
\begin{equation*}
\sup_{x,x' \in \eta_*(\cu) }|u(x) - u(x') |
\leq 
\sum_{\rv{y,y^\prime \in \eta_*(\cu), y \sim y^\prime}}
|u(y) - u(y') |
\leq 
2\cdot 6^d \theta(\lambda) |\cu| 
\sum_{y,y' \eta_*(\tilde\cu), y \sim y'}
|u(y) - u(y') |
\,,
\end{equation*}
so that 
\begin{equation*}
\| u - [u]_{\mathcal{P}} \|_{L^\infty \rv{(\eta_*(\cu))}}
\leq 
2\cdot 6^d \theta(\lambda) |\cu| 
\sum_{\rv{y,y^\prime \in \eta_*(\tilde\cu), y \sim y^\prime}}
|u(y) - u(y') |\,.
\end{equation*}
Here, and in what follows, 
\begin{equation} \label{e.coarsegrain}
	[u]_{\mathcal{P}} (x) := \avsum_{y \in \rv{\eta_*(\cu_x)} } u(y) \quad x \in \cu_{m+1}\,,
\end{equation}
where, as before, $\cu_x \in \mathcal{P}$ denotes the unique cube in the partition $\mathcal{P}$ containing $x. $ 
Squaring this and averaging over the cubes yields 
\begin{align} \label{e.Poincare}
\avsum_{\cu \in \mathcal{P}}
\| u - [u]_{\mathcal{P}} \|_{L^\infty \rv{(\eta_*(\cu))}}^2
\leq 
C
\avsum_{\cu \in \mathcal{P}}
|\cu|^2
\biggl (
\sum_{\rv{y,y^\prime \in \eta_*(\cu), y \sim y^\prime}}
|u(y) - u(y') |
\biggr )^{\!\!2}
\,.
\end{align}

\smallskip

\emph{Step 2. } The gradient~$\nabla [ u ]_{\mathcal{P}}$ is supported on the edges~$(x,y)$ such that~$x$ and~$y$ belong to different elements of~$\mathcal{P}$. The number of such elements is no more than
\begin{equation*}
\frac12 
\sum_{\cu\in \mathcal{P}}
|\cu|^{\frac{d-1}d}
\,.
\end{equation*}
The gradient~$\nabla [ u ]_{\mathcal{P}}(x,y)$ on such an edge~$(x,y)$  is equal to the difference between the mean of~$u$ in~$\eta_*(\cu_x)$ and the mean of~$u$ in~$\eta_*(\cu_y)$, where~$\cu_x$ is the element of~$\mathcal{P}$ containing~$x$, and similarly for~$y$. 
\rv{By property~1 of Definition~\ref{d.wellconnected}, we have that~$\eta_*(\cu_x \cup \cu_y)$ contains~$\eta_*(\cu_x) \cup \eta_*(\cu_y)$, and therefore}
this is bounded by 
\begin{equation*}
\sup_{x',y' \in \eta_*(\cu_x \cup \cu_y) }|u(x') - u(y') |\,,
\end{equation*}
which is estimated, using the same argument as above, by 
\begin{equation*}
\sup_{x',y' \in \eta_*(\cu_x \cup \cu_y) }|u(x') - u(y') |
\leq 
C|\cu_x|
\sum_{\rv{y,y^\prime \in \eta_*(28\cu_x), y \sim y^\prime}}
|u(y) - u(y') |
\,.
\end{equation*}
Notice that~$\nabla [ u ]_{\mathcal{P}}(x,y) = 0$ unless~$x$ and~$y$ are in different, adjacent elements of~$\mathcal{P}$. 
Therefore, for every~$\cu\in\mathcal{P}$,
\begin{equation*}
\sup_{x\in \cu,\, y \sim x}
\bigl| \nabla [ u ]_{\mathcal{P}}(x,y) \bigr |
\leq 
C
|\cu|
\sum_{\rv{y,y^\prime \in \eta_*(28\cu_x), y \sim y^\prime}}
|u(y) - u(y') | 
\,. 
\end{equation*}
Then, setting~$2_* := \frac{2d}{d+2}$, 
and summing over the cubes in the partition yields  
\rv{
	\begin{align}  
\label{e.gradcoarsened}
	\sum_{x\in \cu_m,\, y \sim x}
	\bigl| \nabla [ u ]_{\mathcal{P}}(x,y) \bigr |^{2_*}
	& 
	\leq
	\sum_{\cu\in\mathcal{P}}
	\sum_{x\in \cu,\, y \sim x}
	\bigl| \nabla [ u ]_{\mathcal{P}}(x,y) \bigr |^{2_*}
\notag \\ &
	\leq 
	C\sum_{\cu \in \mathcal{P}}
	|\cu|^{2_*+1 }
	\biggl (
	\sum_{y' \sim y \in \eta_*( 28\cu)}
	|u(y) - u(y') |
	\biggr )^{\!\!2_*}
\notag \\ &
	\leq 
	C
	\biggl (
	\sum_{\cu \in \mathcal{P}}
	|\cu|^{{2d+1}}
	\biggr )^{\frac{2}{d+2}}
	\biggl (
	\sum_{\cu \in \mathcal{P}}
	\sum_{y' \sim y \in \eta_*( 28\cu)}
	|u(y) - u(y') |^2
	\biggr )^{\frac{d}{d+2}}
\notag \\ &
	=
	C|\cu_m|
	\lambda(\cu_m,\mathcal{P})^{\frac{2}{d+2}}
	\biggl ( \frac{1}{|\cu_m|}
	\sum_{y' \sim y \in \eta_*( \cu_m)} \!\!
	|u(y) - u(y') |^2
	\biggr )^{\frac{2_*}{2}}
	\!\!,
\end{align}
}
where the quantity $\la(\cu_m, \mathcal{{P}})$ is defined via 
\begin{equation*}
	\la(\cu_m, \mathcal{P}) := \frac{1}{|\cu_m|} \sum_{\cu \in \mathcal{P}} |\cu|^{2d+1} \,.
\end{equation*}
Simplifying further, we obtain
\begin{equation*}
\biggl (\frac{1}{|\cu_m|}
\sum_{x\in \cu_m,\, y \sim x}
\bigl| \nabla [ u ]_{\mathcal{P}}(x,y) \bigr |^{2_*} \biggr )^{\frac{1}{2_*}}
\leq
C
\la(\cu_m, \mathcal{P})^{\frac1d}
\biggl ( \frac{1}{|\cu_m|}\sum_{\rv{y,y^\prime \in \eta_*(\cu_m), y \sim y^\prime}}
|u(y) - u(y') |^2
\biggr )^{\frac{1}{2}}\,.
\end{equation*}

\smallskip
\emph{Step 3.} 
By \cite[Proposition 2.4]{AD16}, 
there exists a random scale~$\mathcal{M}$ satisfying
\begin{equation*}
\mathcal{M}
= \O_{\frac d{2(d+1)}} (C)
\end{equation*}
such that
\begin{equation*}
3^m \geq \mathcal{M} 
\quad \implies \quad
\la(\cu_m, \mathcal{P}) 
\leq C.
\end{equation*}
Using the Sobolev-Poincar\'e inequality on~$\Zd$,
we therefore obtain, for every~$3^m\geq \mathcal{M}$,
\begin{align} 
\label{e.sobPoincare}
\biggl (
\avsum_{x\in \cu_m}
\bigl| [ u ]_{\mathcal{P}}(x)  - ([ u ]_{\mathcal{P}})_{\cu_m} \bigr |^{2}
\biggr )^{\sfrac1{2}}
&
\leq 
C3^m
\biggl (
\avsum_{x\in \cu_m,\, y \sim x}
\bigl| \nabla [ u ]_{\mathcal{P}}(x,y) \bigr |^{2_*}
\biggr )^{\frac1{2_*}}
\notag \\ &
\leq 
C3^m
\biggl ( \frac{1}{|\cu_m|}\sum_{\rv{y,y^\prime \in \eta_*(\cu_m), y \sim y^\prime}}
|u(y) - u(y') |^2
\biggr )^{\frac{1}{2}}
\,.
\end{align}
Thus, the triangle inequality, combined with \eqref{e.Poincare}, \eqref{e.gradcoarsened} and \eqref{e.sobPoincare} yields, for~$3^m\geq \mathcal{M}$,
\begin{align*}
	\biggl( \avsum_{x \in \cu_m} |u(x) - (u)_{\cu_m}|^2\Biggr)^{\!\!\sfrac12} &= \min_{k \in \R} 	\biggl( \avsum_{x \in \cu_m} |u(x) - k|^2\Biggr)^{\!\!\sfrac12} \\
	&\leqslant   \biggl(  \avsum_{x \in \cu_m}  |u(x) - [u]_{\mathcal{P}}(x)|^2\Biggr)^{\!\!\sfrac12} + \biggl (
	\avsum_{x\in \cu_m}
	\bigl| [ u ]_{\mathcal{P}}(x)  - ([ u ]_{\mathcal{P}})_{\cu_m} \bigr |^{2}
	\biggr )^{\sfrac1{2}}\\
	&\leqslant C 3^m\biggl( \sum_{y' ,y \in \eta_*( \cu_m), y\sim y'}
	|u(y) - u(y') |^2\Biggr)^{\!\!\sfrac12}
	 \,.
\end{align*}
This completes the proof of \eqref{e.main_Poincare}.
\smallskip

\emph{Step 4.} It remains to show that affine functions have uniformly bounded energy on sufficiently large scales. Towards this goal, we compute 
\begin{align*}
	\frac{1}{|\cu_m|} \sum_{x,y \in \eta_*(\cu_m): y \sim x} |x-y|^2 
	&
	= \frac{1}{|\cu_m|} \sum_{\cu \in \mathcal{P}} \sum_{x,y \in \eta_*(28\cu),x\sim y } |x-y|^2\\
	& \leqslant \frac{C}{|\cu_m|} \sum_{\cu \in \mathcal{P}} |\cu|^{2+\frac2d} \leq C \la(\cu_m,\mathcal{P})\,. 
\end{align*}
In view of Step~3, the proof of \eqref{e.affineenergy} is complete. 
\end{proof}

\begin{definition}[Good cube] \label{d.goodness}
	We say that a triadic cube~$\cu\in\mathcal{T}$ is~\emph{good} if it satisfies the conclusion of Lemma~\ref{l.Poincare}. 
\end{definition}

In particular, Lemma~\ref{l.Poincare} says that
\begin{equation}
\label{e.good}
3^m \geq \mathcal{M}
\quad \implies \quad
\cu_m \ \mbox{is good.}
\end{equation}

Applying~\cite[Proposition 2.1, (iii)]{AD16} once again, we obtain a partition~$\mathcal{Q}$ of~$\Zd$ into triadic cubes such that every element~$\cu \in \mathcal{Q}$ satisfies $\cu$ is good, in the sense of Definition~\ref{d.goodness}. In particular, for functions defined on cubes $\eta_*(\cu), \cu \in \mathcal{Q},$ the Poincar\'e inequality from Lemma~\ref{l.Poincare} holds. 

We will define our coarse-grained coefficients with respect to a given ``smallest scale'' which will in practice be a mesoscopic scale, which is actually almost as large as the largest scales we will encounter in an iteration. 
We denote this by~$l \in\N$. 
We introduce the ``thickened boundary'' of a triadic cube~$\cu\in\mathcal{T}$ with~$\size(\cu) > 3^{l}$ by
\begin{equation*}
	\partial_*^{(l)} \cu := 
	\bigcup  \bigl\{ 
	z + \cu_l \,:\, 
	z\in 3^l\Zd\cap \cu_m, \ 
	\dist(z+\cu_l,\partial \cu_m) = 0\bigr\}
	\,.
\end{equation*}
This is basically a neighborhood of the usual Euclidean boundary~$\partial \cu$ of~$\cu$ with thickness~$3^{l} $.
The ``interior'' of a cube~$\cu\in\mathcal{T}$ with~$\size(\cu) > 3^{l}$ is denoted by
\begin{equation*}
	i_l (\cu) := \cu \, \setminus \, \partial_*^{(l)}\cu
	\,.
\end{equation*}
We define~$\mathcal{G}_l$ to be the collection of triadic cubes with size larger than~$3^{l}$ with the property that every triadic subcube of size~$3^{l}$ is good:
\begin{equation*}
\mathcal{G}_l := 
\bigr\{ \cu \in\mathcal{T}\,:\, \ 
\size(\cu)>3^{l} \ \mbox{ and } \
z+\cu_{l} \in \mathrm{WC}\,,  \ 
\forall z\in 3^{l-1} \Zd  \ \mbox{such that} 
\ z+\cu_l \subseteq \cu
\bigr\}  \,.
\end{equation*}
Observe that, by~\eqref{e.Penrose} and a union bound, for every~$l < n$, 
\begin{equation}
\P \bigl[ \cu_n \not\in \mathcal{G}_l \bigr] 
\leq
\sum_{z\in 3^{l-1} \Zd \cap \cu_n} 
\P \bigl[ z+\cu_l  \not \in  \mathrm{WC} \bigr]
\leq
C 3^{d(n-l)} 
\exp( -c 3^l) 
\,.
\label{e.Gl.nope}
\end{equation}
\rv{We will frequently use the simple observation, which is immediate from the definition of~$\mathcal{G}_l$, that any~$\cu \in \mathcal{G}_l$ has the property that
\begin{equation*}
\eta_*(z+\cu_l ) \subseteq \eta_*(\cu) \,, \quad \forall z\in 3^l\Zd\cap \cu \,.
\end{equation*}
Indeed, this follows from the fact that the clusters of neighboring subcubes congruent to~$\cu_l$ must be connected due to the first and third properties of Definition~\ref{d.wellconnected} and the fact that the family~$\{ z + \cu_l \,:\, z\in 3^{l-1} \Zd\,,  z+\cu_l \subseteq \cu\}$ are overlapping. It follows that, for every~$l \leq n \leq \size(\cu)$ and~$z\in \Zd\cap \cu$ with~$z+\cu_n \subseteq\cu$, 
\begin{equation}
\eta_*(z+\cu_n ) \subseteq \eta_*(\cu) \,, \quad \forall z\in 3^l\Zd\cap \cu \,.
\label{e.connected.subcubes}
\end{equation}
By the same argument, we also observe that, for such~$z$ and~$n$, 
\begin{equation}
( \eta_*(\cu) \cap (z+\cu_n) ) \setminus \eta_*(z+\cu_n)  \subseteq \partial_*^{(l)}(z+\cu_n)\,.
\label{e.extrapoints.boundarylayer}
\end{equation}
}

\section{Subadditive quantities} \label{s.subadd}

Given a cube~$\cu\in \mathcal{G}_l$ and a function $u : \eta_* (\cu ) \to \R$, the \emph{coarsening of~$u$} is a function~$[u]_l : \Zd \cap \cu \to \R$ defined by
\begin{equation*}
[u]_l (x) :=
\avsum_{y \in \eta_*(\cu) \cap (z+\cu_l)}
u(y)\,,
\quad
z\in 3^l \Zd \cap \cu, \,
x \in \Zd \cap (z+\cu_l)\,.
\end{equation*}
In other words, the coarsening~$[u]_l$ of~$u$ is a function defined on the subset~$\Zd \cap \cu$ of the lattice~$\Zd$ which is constant on subcubes of the form~$z+\cu_l$, and the constant is equal to the average of~$u$ on that subcube. The point here is that we need to replace~$u$ with a function which approximates~$u$ but lives on Euclidean space (we use~$\Zd$ rather than~$\Rd$ for convenience, but this design choice does not really matter). This allows us to ``forget'' the complicated geometric structure of the random graph and focus on more macroscopic details. A similar but slightly different version of this coarsened function was also used in~\cite{AD16}.

\subsubsection{Coarse-grained coefficients}
For each~$p,q \in\Rd$ and~$\cu\in \mathcal{T}$ with~$\size(\cu) >3^{l}$, we define the following random variables:
\begin{equation*}
\mu^{(l)}(\cu,p) 
:= 
\inf \biggl\{ 
\frac{1}{|\cu|} 
\sum_{x,y\in \eta_*(\cu), x\sim y} 
\frac12
(u(x) - u(y))^2
\,:\, 
u : \eta_*(\cu) \to \R\,, u = \ell_p \ \mbox{in} 
\ \partial_*^{(l)} \cu
\biggr\} 
\indc_{\{  \cu \in \mathcal{G}_l  \} }
\end{equation*}
and 
\begin{align*}
\mu_*^{(l)}(\cu,q)
&:= 
\frac{1}{|\cu|} 
\sup \biggl\{ 
- \!\!\!
\sum_{x,y \in \eta_*(\cu) , x\sim y} 
\frac12
(u(x) - u(y))^2
\\ & +
\Bigl( 1 - \frac{1}{\size(\cu)} \Bigr)^{\!\!-1}\!\!\!\!\!\sum_{x,y\in \Zd\cap \cu, x\sim y}
(q\,{\cdot}\, (x{-}y) ) ( [ u ]_l(x) {-}  [ u ]_l(y))
\,\Big\vert\, 
u : \eta_*(\cu) \to \R
\biggr\} 
\indc_{\{  \cu \in \mathcal{G}_l  \} }
\,.
\end{align*}
We emphasize that these quantities are only being defined for good cubes in the sense of membership in~$\mathcal{G}_l.$ We will first check that these quantities are bounded above. 

\begin{lemma}[Upper bounds for the quantities] 
\label{l.uppbound}
\rv{There exists a deterministic constant~$K > 0$ such that} for any $p \in \Rd,$ and $\cu \in \mathcal{G}_l,$ we have 
\begin{equation} \label{e.upbd.mu}
	\mu^{(l)} (\cu,p) \leqslant K|p|^2\,, 
\end{equation}
and 
\begin{equation} \label{e.upbd.mu*}
	\mu^{(l)}_*(\cu,q) \leqslant CK|q|^2
\end{equation}

\end{lemma}
\begin{proof}
Let $\cu \in \mathcal{G}_l$, and $p \in \Rd.$ We test the definition of $\mu^{(l)}(\cu,p)$ with the function $u = \ell_p.$ By Cauchy-Schwarz this implies that 
	\begin{equation*}
		\mu^{(\ell)}(\cu,p)  \leqslant \frac{1}{2|\cu|} \sum_{x,y \in \eta_*(\cu),y\sim x}  (p \cdot (x - y))^2 \leqslant |p|^2 \frac{1	}{2|\cu|}	\sum_{x,y \in \eta_*(\cu),y\sim x}|x-y|^2 \leqslant \frac{K}2|p|^2\,. 
	\end{equation*}
This yields \eqref{e.upbd.mu}. 

\smallskip 
We turn to \eqref{e.upbd.mu*}.  \rv{As~$\size(\cu) = 3^l \geqslant 3,$ the factor~$(1 - \size(\cu)^{-1})^{-1} \in [1,\frac32].$ Therefore, to} bound the dual quantity, we estimate, for every~$u$, and \rv{by Cauchy-Schwarz,}
\begin{align*}
\lefteqn{
\biggl|\Bigl( 1 - \frac{1}{\size(\cu)} \Bigr)^{-1} \rv{\frac1{|\cu|}}\sum_{x ,y\in \Zd \cap \cu, y \sim x}  (q \cdot (x-y))([u]_l(x) - [u]_l(y))\biggr| 
} \qquad
\notag \\ &
\leq
C|q|  \Bigl\|\nabla \bigl([u]_l - \avsum_{z \in \cu} [u]_l(z)\bigr) \Bigr\|_{\underline{H}^{-1}(\cu)}
\notag \\ &
\leq 
C|q|  \Bigl\| [u]_l - \avsum_{z \in \cu} [u]_l(z)\Bigr\|_{\underline{L}^2(\cu)}
\leq 
CK^{\sfrac12}|q|  
\biggl( 
\frac{1}{|\cu|} \!
\sum_{x,y\in \eta_*(\cu), y \sim x} 
\bigl( u(x) - u(y) \bigr)^2
\biggr)^{\!\!\sfrac12}\,.
\end{align*}
We deduce therefore, using the previous display and Young's inequality, that, for every~$u$, 
\begin{align*}
&
-
\frac{1}{|\cu|} 
\sum_{x,y \in \eta_*(\cu),\,y\sim x} 
\frac12 (u(x) - u(y))^2
\notag \\ & \quad
+
\biggl|\bigl( 1 - \size(\cu)^{-1} \bigr)^{\!\!-1} \rv{\frac{1}{|\cu|}}\!\!\sum_{x,y \in \Zd \cap \cu,\,y\sim x}  (q \cdot (x-y))([u]_l(x) - [u]_l(y))\biggr| 
\\ & \qquad
\leq 
-
\frac{1}{|\cu|} 
\sum_{x,y \in \eta_*(\cu),y\sim x} 
\frac12 (u(x) - u(y))^2
+
CK^{\sfrac12} |q| 
\biggl( 
\frac{1}{|\cu|} \! 
\sum_{ x ,y\in \eta_*(\cu),x\sim y}
\bigl( u(x) - u(y) \bigr)^2
\biggr)^{\!\!\sfrac12}
\\ & \qquad
\leq 
C K |q|^2 \,.
\end{align*}
This completes the proof.
\end{proof}

\begin{lemma}[Fenchel inequality]
\label{l.fenchel}
For every~$p,q \in\Rd$ and~$\cu \in \mathcal{G}_l$, 
\begin{equation*}
\mu^{(l)} (\cu,p)
+
\mu^{(l)}_* (\cu,q)
\geq 
p\cdot q 
\,.
\end{equation*}
\end{lemma}
\begin{proof}
Letting~$\cu \in \mathcal{G}_l,$ we let~$v_l(\cdot,\cu,p)$ denote the minimizer for the variational problem defining~$\mu^{(l)}(\cu,p).$ Testing~$v_l(\cdot,\cu,p)$ as a competitor for the variational problem defining~$\mu^{(l)}_*(\cu,q),$ it follows that 
\begin{align} \label{e.fenchelpf}
	\mu^{(l)}_*(\cu,q)  &\geqslant  - \frac{1}{|\cu|} \sum_{x,y \in \eta_*(\cu)\,y\sim x}  \frac12 (v_l(x,\cu,p) - v_l(y,\cu,p))^2 \\ \notag 
	&\quad + \rv{\Bigl( 1 - \frac{1}{\size(\cu)} \Bigr)^{\!\!-1}} \rv{\frac{1}{|\cu|}}\sum_{x,y \in \Zd \cap \cu,\, y \sim x} (q\cdot(x-y)) ([v_l](x,\cu,p) - [v_l](y,\cu,p)) \\ \notag 
	&= -\mu^{(l)}(\cu,p) +\rv{\Bigl( 1 - \frac{1}{\size(\cu)} \Bigr)^{\!\!-1}} \rv{\frac{1}{|\cu|}}\sum_{x,y \in \Zd \cap \cu,\,y \sim x} (q\cdot(x-y)) ([v_l](x,\cu,p) - [v_l](y,\cu,p)) \,.
\end{align}
It remains to compute the last sum using the discrete Stokes' formula. Toward this goal we introduce the outer normal vector $\mathbf{n}(x) \in \RR^d$ to $\partial \cu$ at $x$ by
\begin{equation*}
	\mathbf{n}(x) := \sum_{i=1}^d \bigl( e_i \indc_{\{x - e_i \in \cu\}} - e_i \indc_{\{x + e_i \in \cu\}}\bigr)\,. 
\end{equation*}
Then by the discrete Stokes' formula,
\begin{align*}
\lefteqn{
\frac{1}{|\cu|}  \sum_{x,y \in \Zd \cap \cu, \, y \sim x} (q\cdot(x-y)) ([v_l](x,\cu,p) - [v_l](y,\cu,p)) 
}
\qquad & 
\\& = \frac{1}{|\cu|}  \sum_{x,y \in \Zd \cap \cu,\, y \sim x} (q\cdot(x-y)) ([v_l](x,\cu,p) - [v_l](y,\cu,p) - p\cdot(x-y)) \\
&\quad \quad \quad + \frac{1}{|\cu|}  \sum_{x,y \in \Zd \cap \cu,\,y \sim x} (q\cdot(x-y)) (p\cdot(x-y))
\\ & 
=
\frac{1}{|\cu|}  \sum_{x ,y\in \Zd \cap \cu,\,y \sim x} (q\cdot(x-y)) (p\cdot(x-y))
=
\Bigl( 1 - \frac{1}{\size(\cu)} \Bigr)p\cdot q 
\,.
\end{align*}
In the second line, we added and subtracted the vector field $(x,y) \mapsto p\cdot(x-y),$ and used that $[v_l](x,\cu,p) = p \cdot x$ at the boundary cubes so that the first term exactly vanishes; the second term has constant summand, and we are left with counting the average number of edges within the cube $\cu,$ excluding boundary edges: an elementary counting argument then leads to the last display.  

\rv{Finally, dividing the last line by~$1 - \size(\cu)^{-1}$ and plugging into~\eqref{e.fenchelpf} yields the desired inequality. }
\end{proof}

\begin{lemma}[Lower bounds for the quantities]
\label{l.lowbound}
There exists a constant $C(d) > 0$ such that for every $p,q \in \RR^d$ and $\cu \in \mathcal{G}_l,$ we have 
\begin{equation*}
	\mu^{(l)}(\cu,p) \geqslant \frac{1	}{C}|p|^2\,,
\end{equation*}
and 
\begin{equation*}
\mu_*^{(l)}(\cu,q) \geqslant \frac{1}{C}|q|^2\,. 	
	\end{equation*}
\end{lemma}
\begin{proof}
This is an immediate consequence of the two preceding lemmas. For instance, to get the lower bound for $\mu^{(l)}(\cu,p),$ we transpose the $\mu_*^{(l)}(\cu,q)$ over to the right-hand side of the Fenchel inequality in Lemma~\ref{l.fenchel}, use the quadratic upper bound from Lemma~\ref{l.uppbound}, and optimize over $q\in\RR^d.$ The other inequality is similar. 
\end{proof}

Now that we have upper and lower bounds, we define the matrix~$\a_l (\cu)$ to be the matrix for the quadratic form~$p \mapsto \mu^{(l)}(\cu,p)$ and the matrix~$\a_{*,l}(\cu)$
to be the matrix whose inverse~$\a_{*,l}^{-1}(\cu)$ is the matrix for~$q \mapsto \mu_*^{(l)}(\cu,q)$. 
This is all for good cubes. 
For the case~$\cu \not\in \mathcal{G}_l$, we just set~$\a_l (\cu)=+\infty$ and~$\a_{*,l}(\cu) := 0$.

Next we rewrite the Fenchel inequality as an ordering of these matrices. 

\begin{lemma}
	\label{l.order}
	For every~$\cu \in \mathcal{G}_l$ we have 
	\begin{equation*}
	\a_l (\cu) \geq \a_{*,l}(\cu)
		\,.
	\end{equation*}
\end{lemma}
\begin{proof}
	Plugging in $\mu^{(l)}(\cu,p) = \tfrac12 p \cdot \rv{\a_l(\cu)} p$ and $\mu^{(l)}_*(\cu,q) = \tfrac12 q \cdot \rv{\a_{*,l}^{-1}(\cu)}q$ \rv{in the left-hand side of the Fenchel inequality from Lemma~\ref{l.fenchel},} transposing the latter quadratic form to the right-hand side of the same inequality and optimizing over $q \in \RR^d$ yields
\rv{	\begin{equation*}
\frac12 p \cdot \a_l(\cu)p \geqslant \sup_{q \in \Rd} \Bigl( p \cdot q - \frac12 q \cdot \a_{*,l}^{-1}(\cu) q\Bigr) = \frac12 p \cdot \a_{*,l}(\cu) p\,. 
	\end{equation*}
This completes the proof of the lemma as~$p\in\Rd$ is arbitrary. 
}
\end{proof}

Next we study subadditivity. 

\begin{lemma}[Approximate subadditivity]
\label{l.subbadditivity}
There exists a constant~$C(d,\lambda) > 0$ such that if~$\cu_m \in \mathcal{G}_l$, then  
\begin{equation} \label{e.subadd1}
\a_l(\cu_m)  
-
\avsum_{z\in 3^n\Zd \cap \cu_m} \a_l(z+\cu_n)
\leq 
C3^{-(n-l)}
\,. 
\end{equation}
and
\begin{equation} \label{e.subadd2}
\a_{*,l}^{-1} (\cu_m)  
-
\avsum_{z\in 3^n\Zd \cap \cu_m} \a_{*,l}^{-1}(z+\cu_n)  \leq  C3^{-(n-l)/2}. 
\end{equation}

\end{lemma}
\begin{proof}
First we prove \eqref{e.subadd1}. Let $\cu_m, \{z + \cu_n\}_{z \in 3^n \Zd \cap \cu_m} \subseteq \mathcal{G}_l ,$ and $p \in \RR^d.$ It suffices to prove 
\begin{equation*}
	\mu^{(l)}(\cu_m,p) - \avsum_{z \in 3^n \Zd \cap \cu_m} \mu^{(l)}(z + \cu_n,p) \leqslant C3^{-(n-l)}|p|^2\,. 
\end{equation*}
Toward this goal, for each~$ n \in \NN,$ and~$z \in 3^n \ZZ^d \cap \cu_m,$ we denote by $v(\cdot, z  + \cu_n,p)$ the minimizer of the variational problem defining ~$\mu^{(l)}(z \!+ \cu_n,p).$ 

\smallskip 
\noindent \rv{We define the competitor 
\begin{equation*}
w(x,\cu_m,p) := 
\left\{
\begin{aligned}
& v(x,z + \cu_n,p) &&  \quad \mbox{if} \ x \in \eta_*(z + \cu_n) \ \mbox{with} \ z \in 3^n \ZZ^d \cap \cu_m \,, \\
& 
v(x,z + \cu_n,p) = \ell_p 
&&  \quad \mbox{otherwise} \,.
\end{aligned}
\right.
\end{equation*}
}
We note that, by~\eqref{e.extrapoints.boundarylayer},  if~$\cu_m\in \mathcal{G}_l$, then the only points belonging to~$\eta_*(\cu_m)$ but not belonging to any~$\eta_*(z+\cu_n)$ with~$z\in3^n\Zd \cap \cu_m$ are points in the boundary layers between two subcubes, that is,~$\partial^{(l)}_*(z+\cu_n)$. Here the function is already set equal to the affine~$\ell_p$, so we are just adding these ``missing'' points to the domain of~$w$. 
Since $w(x,\cu_m,p) = \ell_p(x)$ along the thickened boundary,~i.e.,~for all $x \in \partial^{(l)}_*(\cu_m),$ it follows that it is admissible for the variational problem defining $\mu^{(l)}(\cu_m,p).$ In the computation below we simply write $w(x)$ suppressing the dependence on $\cu_m$ and $p,$ and along the way, for convenience we introduce $\mathcal{Z}_m := \{ z \in 3^n \Zd : \rv{(z + \cu_n)} \cap \cu_m \neq \emptyset \}.$ Comparing energies we find
\begin{align*}
\mu^{(l)}(\cu_m,p) &\leqslant \frac{1}{|\cu_m|}	\sum_{x,y \in \eta_*(\cu_m),\, y \sim x} \frac12 (w(x) - w(y))^2 \\
&\leqslant \frac{1}{|\cu_m|}\sum_{z \in \mathcal{Z}_m} \biggl[ \sum_{x \in \eta_* (\cu_m) \cap (z + \cu_n), \, \rv{y \in \eta_*(\cu_m)}, \, y \sim x} \frac12 (w(x) - w(y))^2 \biggr]\\
&\leqslant \frac{1}{|\cu_m|}\! \sum_{z \in \mathcal{Z}_m} \!\biggl[ \sum_{x \in \eta_*(\cu_m) \cap \rv{(z {+} \cu_n)}} \sum_{y \sim x: y \in \rv{\eta_*(\cu_m) \cap \eta_*(z {+} \cu_n)}} \frac12 (v(x, z{+} \cu_n,p) - v(y,z {+} \cu_n,p))^2\\
& \phantom{\leqslant \frac{1}{|\cu_m|} \sum_{z \in \mathcal{Z}_m} } +  \sum_{x \in \eta_*(\cu_m) \cap \rv{(z + \cu_n)}} \sum_{y \sim x: y\in \rv{\eta_*(\cu_m) \setminus \eta_*(z + \cu_n)}} \frac12 (p\cdot x - p\cdot y)^2   \biggr]\\
&\leqslant \frac{1}{|\cu_m|} \sum_{z \in \mathcal{Z}_m} \biggl[ |\cu_n| \mu^{(l)}(z+\cu_n,p) + |p|^2 \!\!\!\!\!\!\!\!\!\! \sum_{x \in \eta_*(\cu_m) \cap \eta_*(z + \cu_n)} \sum_{y \sim x: y\in \rv{\eta_*(\cu_m) \setminus \eta_*(z + \cu_n)}} |x-y|^2\biggr]\,.
\end{align*}
The last term is controlled by the normalized total volume of the boundary layers, so that 
\begin{equation*}
	\mu^{(l)}(\cu_m,p) \leqslant \frac{1}{|\cu_m|} \sum_{z \in \mathcal{Z}_m} |\cu_n| \mu^{(l)}(z + \cu_n,p) + C|p|^2 3^{-(n-l)}\,.
\end{equation*}
This concludes the proof of \eqref{e.subadd1} upon re-writing the quantities in terms of their associated matrices.

\smallskip
Proof of \eqref{e.subadd2}. 
As before, fixing $q \in \Rd,$ we must show that 
\begin{equation*}
	\mu_*^{(l)}(\cu_m,q) - \avsum_{z \in 3^n \Zd \cap \cu_m} \mu_*^{(l)}(z + \cu_n,q) \leqslant CK^{\sfrac12} 3^{-(n-l)/2}|q|^2\,. 
\end{equation*}
To show this, we test each of the variational problems $\mu_*^{(l)}(z+\cu_n,q)$ with the restriction to the cube $z + \cu_n$, of the minimizer $u(\cdot,\cu_m,q)$ for $\mu_*^{(l)}(\cu_m,q)$. Energy comparison yields
\begin{align*}
	\mu_*^{(l)}(z+\cu_n,q) &\geqslant - \frac{1}{|\cu_n|} \!\! \sum_{x \in \eta_*(z + \cu_n)} \! \sum_{y \sim x} \frac12(u(x,\cu_m,q) - u(y,\cu_m,q))^2 \\
	& + (1 - 3^{-n})^{-1}\!\!\!\!\!\sum_{x\in \Zd\cap (z+ \cu_n)} \sum_{y\sim x} (q\cdot (x-y)) \bigl([u(\cdot,\cu_m,q) ]_l(x) - [u(\cdot, \cu_m,q)]_l(y)\bigr)\,.
\end{align*}
For brevity we let $\mathcal{Z} := \{ z \in 3^n \Zd : (z + \cu_n) \cap \cu_m \neq \emptyset\}$. Summing this over $z \in\mathcal{Z},$ and dropping nonnegative quadratic remainder terms that occur along boundary layers, we find 
\begin{align*}
&	\mu_*^{(l)}(\cu_m,q) - \avsum_{z\in 3^n \Zd \cap \cu_m} \mu_*^{(l)}(z+\cu_n,q) \\
	&\quad \leqslant -\frac{(1-3^{-n})^{-1}}{|\cu_m|}  \avsum_{z \in \mathcal{Z}}  \sum_{x \in z + \cu_n} \sum_{y \sim x, y \in z^\prime + \cu_n, z^\prime \neq z} q \cdot (x-y) \bigl( [u(\cdot,\cu_m,q)]_l(x) - [u(\cdot,\cu_m,q)]_l(y)\bigr)\,.
\end{align*}
By Cauchy-Schwarz, it follows that 
\begin{align*}
	&\biggl| \frac{(1-3^{-n})^{-1}}{|\cu_m|}  \avsum_{z \in \mathcal{Z}}  \sum_{x \in z + \cu_n} \sum_{y \sim x, y \in z^\prime + \cu_n, z^\prime \neq z} q \cdot (x-y) \bigl( [u(\cdot,\cu_m,q)]_l(x) - [u(\cdot,\cu_m,q)]_l(y)\bigr)\biggr|\\
	&\leqslant \frac{|q|}{|\cu_m|}\avsum_{z \in \mathcal{Z}} \Bigl(  \sum_{x \in z + \cu_n} \sum_{y \sim x, y \in z^\prime + \cu_n, z^\prime \neq z} |x-y|^2\Bigr)^{\sfrac12}\\
	&\qquad \qquad  \qquad \quad \times\Bigl(  \sum_{x \in z + \cu_n} \sum_{y \sim x, y \in z^\prime + \cu_n, z^\prime \neq z} \bigl( [u(\cdot,\cu_m,q)]_l(x) - [u(\cdot,\cu_m,q)]_l(y)\bigr)^2\Bigr)^{\sfrac12}\\
	&\leqslant  CK^{\sfrac12}3^{-(n-l)/2}|q| \frac{3^{l}}{|\cu_l|} \Biggl(\sum_{x\in\eta_*(z+\cu_l)} \sum_{y \sim x} (u(x,\cu_m,q) - u(y,\cu_m,q))^2 \Biggr)^{\!\!\sfrac12}\\
	&\leqslant CK^{\sfrac12} 3^{-(n-l)/2}|q|^2
	\,.
\end{align*}
In the penultimate step we used the triangle inequality and the Poincar\'e inequality at scale~$3^l,$ and a volume bound for the boundary layers, and concluded with the uniform bound on the energies. 
\end{proof}

An immediate consequence of the preceding subadditivity inequalities is the next corollary. To state it, here and in what follows we introduce the following \emph{averaged versions of the coarse-grained coefficients}. 
We define 
\begin{align*}
	\rv{\ahom_l} (\cu) := \E [\a_{\rv{l}}(\cu) \indc_{\{\cu \in \mathcal{G}_l\}}], \quad \mbox{and} \quad
	\rv{\ahom_{*,l}(\cu)} := \E[\a_{*,l}^{-1}(\cu)\indc_{\{\cu \in \mathcal{G}_l\}}]^{-1}\,.
\end{align*}

\begin{corollary}
	\label{c.order}
There exists a constant $C(d,\la) > 0$ such that for every $m > l,$  we have 
	\begin{align*}
		\ahom_l(\cu_{m+1}) \leqslant \ahom_l(\cu_m) {+} C3^{-(m-l)},  \quad \mbox{ and }  \ahom_{*,l}(\cu_m) \leqslant \ahom_{*,l}(\cu_{m+1}) {+} C 3^{-(m-l)/2} \,.
	\end{align*}
	
\end{corollary}
\begin{proof}
	Both inequalities are proven by taking expectations on the both sides of \eqref{e.subadd1} and \eqref{e.subadd2} respectively. For instance, for the first inequality, we consider \eqref{e.subadd1} with $m+1 $ in place of $m$, and $m$ in place of $n,$ and let $A_{m} $ denote the event 
	\begin{equation*}
		A_{m} := \{ \cu_{m+1} \in \mathcal{G}_l \mbox{ and } z + \cu_m \in \mathcal{G}_l \mbox{ for every } z \in 3^m \Zd \cap \cu_{m+1}\}\,.
	\end{equation*} 
We observe that, by~\eqref{e.subadd1}, 
\begin{align*}
\ahom_l(\cu_{m+1})
&
=
\E[\a_l (\cu_{m+1})\indc_{\{ \cu_{m+1}\in \mathcal{G}_l \} }] 
\notag \\ & 
= 
\E[\a_l (\cu_{m+1})\indc_{A_m}] 
+ 
\E[\a_l (\cu_{m+1})\indc_{\{ \cu_{m+1}\in \mathcal{G}_l \} \setminus A_m }] 
\notag \\ & 
\leqslant 
\E [\a_l (\cu_{m+1}) \indc_{\{ \cu_{m+1}\in \mathcal{G}_l \}} ] 
+C3^{-(m-l)} + C \P [A_m^c]
\notag 
\,.
\end{align*}	
By Proposition~\ref{p.Penrose}, we find $ \P [A_m^c] \leqslant C\exp (-c 3^{\frac{d}{2(d+1)} m} ) \leq C 3^{-100m}$, so that we compute
	\begin{equation*}
\E \bigl[ \a_l(\cu_{m+1}) \indc_{A_m}\bigr] \leqslant \avsum_{z \in 3^m \Zd \cap \cu_{m+1}} \E \bigl[ \a_l(z + \cu_m)\indc_{\{ z+\cu_{m}\in \mathcal{G}_l \}}\bigr] + C 3^{-(m-l)}\,.
	\end{equation*}
By stationarity of $\P$ (see \cite[Proposition 8.3]{LP}), it follows that for every $z \in 3^m \Zd \cap \cu_{m+1}$,  
\begin{equation*}
	\avsum_{z \in 3^m \Zd \cap \cu_{m+1}} \E \bigl[ \a_l(z + \cu_m) \indc_{\{ z+\cu_{m}\in \mathcal{G}_l \}}\bigr] = \E \bigl[ \a_l(\cu_m) \indc_{\{ \cu_{m}\in \mathcal{G}_l \}}\bigr] 
=	\ahom_l(\cu_{m})
	\,,
\end{equation*} 
so that the first inequality follows. The proof of the second inequality is similar.
\end{proof}

\smallskip 

Before proceeding, it is convenient to introduce the following master quantity that tracks both the Dirichlet and Neumann quantities~\rv{$(\cu,p) \mapsto \mu^{(l)} (\square,p)$ and~$(\cu,q)\mapsto \mu^{(l)}_*(\square,q)$, respectively, which are approximately subadditive as proven in Lemma~\ref{l.subbadditivity}}. We adapt the program in \cite{AK22}. Thus, for any $p,q \in \Rd$ we set 
\begin{equation}
	\label{e.Jdef}
	J_l(\cu,p,q) := \Bigl(\mu^{(l)}(\cu,p) + \mu_*^{(l)}(\cu,q) - p \cdot q \Bigr)\indc_{\{  \cu \in \mathcal{G}_l  \} }\,.
\end{equation}
By Lemma~\ref{l.fenchel}, $J_l(\cu,p,q) \geqslant 0$ for all $\cu \in \mathcal{G}_l, p,q \in \Rd.$ The following lemma is a representation formula for the master quantity $J_l.$ At this point it is useful to introduce the set of graph-harmonic functions. For any good cube $\cu \in  \mathcal{G}_l$ we introduce the set of solutions of the graph Laplacian \rv{on $\eta_*(\cu)$}:
\begin{align*}
	\mathcal{A}(\eta_*(\cu)) := \biggl \{ w : \eta_*(\cu) \to \RR \mbox{ such that } \sum_{y \in \eta_*(\cu), \, y \sim x} (w(x) - w(y)) = 0, \ \forall x \in \eta_*(\cu) \biggr \}\,. 
\end{align*}
Informally we refer to elements of $\mathcal{A}$ as graph-harmonic functions.

\begin{lemma}
	\label{l.J.rep.formula}
	For any $\cu \in \mathcal{G}_l,$ and $p,q \in \Rd$, there holds 
	\begin{align*}
		J_l(\cu,p,q) &= \max_{\xi \in \mathcal{A}(\eta_*(\cu))} \frac{1}{|\cu|} \sum_{x,y \in \eta_*(\cu),\,y \sim x} - \frac12(\xi(x) - \xi(y))^2 - (p \cdot (x-y)) (\xi(x) - \xi(y)) \\
		 &\quad \quad + \sum_{x,y \in \Zd \rv{\cap \cu},\, y \sim x} (q\cdot(x-y)) ([\xi]_l (x) - [\xi]_l(y))\,.
			\end{align*}
\end{lemma}

\begin{proof}
	We fix $\cu \in \mathcal{G}_l,$ and let $v_l(\cdot,\cu,p)$ denote the minimizer to $\mu^{(l)}(\cu,p).$  Then thanks to the condition ~$v_l(\cdot,\cu,p) = \ell_p(\cdot)$ on the thickened boundary $\partial_l^*\cu,$ we obtain as \rv{in the proof of Lemma~\ref{l.fenchel}} that 
	\begin{equation*}
		 p \cdot q\bigl(1 - \size(\cu)^{-1} \bigr) = \frac{1}{|\cu|} \sum_{x,y \in \Zd \cap \cu ,\, y\sim x}  (q \cdot(x-y))\bigl([v_l(\cdot,\cu,p)]_l(x) - [v_l(\cdot,\cu,p)]_l(y)\bigr)\,.
	\end{equation*}
Fix $w \in \mathcal{A}(\eta_*(\cu))$.  For brevity, we will write $[v_l]_l(x)$ in place of $[v_l(\cdot,\cu,p)]_l(x)$, and so on. We compute:
\begin{align*}
	&\rv{\mu^{(l)}}(\cu,p) - p \cdot q \\
	&+ \frac{1}{|\cu|} \Biggl[\sum_{x,y \in \eta_*(\cu),\,y \sim x} \! -\frac12 (w(x) \!- w(y))^2 \\
	&\qquad \qquad \qquad + \!\!(1 - \size(\cu)^{-1})^{-1}\!\!\!\!\!\!\sum_{x,y \in \Zd \cap \cu,\,y \sim x} (q \cdot (x \! -y))([w]_l(x)  \!- [w]_l(y))\Biggr] 
 \\
	& = \frac{1}{|\cu|} \sum_{x,y \in \eta_*(\cu),\, y \sim x} \frac12 \bigl( (v(x) - v(y))^2 - (w(x) - w(y))^2 \bigr ) \\
	&\quad \quad \quad + \frac{(1 - \size(\cu)^{-1})^{-1}}{|\cu|}\sum_{x,y \in \Zd \cap \cu,\,y \sim x} (q\cdot(x-y)) \Bigl( [w]_l(x) - [v_l]_l(x) - ([w](y) - [v_l]_l(y))\Bigr)\\
	& = \frac{1}{|\cu|} \sum_{x,y\in \eta_*(\cu),\, y \sim x} -\frac12 \bigl( v(x) - w(x) - \bigl( v(y) - w(y)\bigr) \bigr)^2 \\
	&\quad \quad \quad -  \bigl( v(x) - v(y)\bigr)\bigl( w(x) - w(y) - (v(x) - v(y)) \bigr)\\
		&\quad \quad \quad + \frac{(1 - \size(\cu)^{-1})^{-1}}{|\cu|}\sum_{x,y \in \Zd \cap \cu, \, y \sim x} (q\cdot(x-y)) \Bigl( [w]_l(x) - [v_l]_l(x) - ([w](y) - [v_l]_l(y))\Bigr)\, .
\end{align*}
As $w-v$ is graph harmonic, and $v(x) = p \cdot x$ along $\partial_*^{(l)}\cu,$ it follows that the middle term of the last line evaluates to 
\begin{equation*}
	- (p \cdot(x-y)) \bigl( w(x) - w(y) - (v(x) - v(y))\bigr)\,. 
\end{equation*}
Using this, continuing from the preceding display we obtain
\begin{align*}
	&\mu^{(l)}(\cu,p) - p \cdot q \\
	& + \frac{1}{|\cu|} \Biggl[\sum_{x,y \in \eta_*(\cu),\, y \sim x} \! {-} \frac12 (w(x) {-} w(y))^2 \\
	&\qquad \qquad \qquad \!{+} (1 {-} \size(\cu)^{-1})^{-1}\sum_{x,y \in \Zd \cap \cu,\, y \sim x} (q \cdot (x \! -y))([w]_l(x)  \!- [w]_l(y))\Biggr] \\
	& = \frac{1}{|\cu|} \sum_{x,y\in \eta_*(\cu),\, y \sim x} - \frac12 \biggl( v(x) - w(x) - \bigl( v(y) - w(y)\bigr) \biggr)^2 \\ &\quad \quad  - (p\cdot(x-y))\bigl( w(x) - w(y) - (v(x) - v(y) \bigr)\\
	&\quad \quad  + \frac{(1 - \size(\cu)^{-1})^{-1}}{|\cu|}\sum_{x,y \in \Zd \cap \cu,\, y \sim x} (q\cdot(x-y)) \Bigl( [w]_l(x) - [v_l]_l(x) - ([w](y) - [v_l]_l(y))\Bigr)\, . 
\end{align*}
Taking the supremum of the left and right hand sides over $w \in \mathcal{A}(\eta_*(\cu))$ concludes the proof of the lemma. 
\end{proof}

Our next lemma concerns the first and second variations and their consequences.

\begin{lemma}
\label{l.variations}
For any $\cu \in \mathcal{G}_l,$ and any $w \in \rv{\mathcal{A}(\eta_*(\cu))},$ we have the following identities: 
\begin{enumerate}
	\item First variation formula for the Neumann quantity:
	\begin{align} \notag
		& (1 - \size(\cu)^{-1})^{-1}\sum_{x,y \in \Zd \cap \cu,\,y \sim x} (q \cdot(x-y)) \bigl( [w]_l (x) \!- [w]_l(y)\bigr) \\ &\quad\quad \label{e.firstvarmu*}  =  \sum_{x,y \in \eta_*(\cu),\, y \sim x} (u_l(x,\cu,q) 
		\! - u_l(y,\cu,q)) (w(x) \! - w(y))\,.
	\end{align}
\item Second variation formula for the Neumann quantity: 
\begin{align}
	\label{e.secondvarmu*}
	\frac{1}{|\cu|} \sum_{x,y \in \eta_*(\cu),\, y \sim x} (w(x) \! - w(y))^2 - \bigl( u_l(x,\cu,q)\! + \! w(x) \!- \! u_l(y,\cu,q)\! - w(y)\bigr)^2\! =\! 2\mu_*^{(l)}(\cu,q)\,.
\end{align}

\item First variation formula \rv{for~$J_l$}: The quantity $J_l(\cu,p,q)$ admits a unique maximizer up to an additive constant. Denoting by $\xi_l(\cdot,\cu,p,q)$ a \rv{maximizer} (say with mean-zero over $\cu$ normalization), 
\begin{align}
	\label{e.firstvarmaster}
&(1 - \size(\cu)^{-1})^{-1}\sum_{x,y\in \Zd\cap \cu,\, y \sim x} (q \cdot (x-y))([w]_l(x) - [w]_l(y))   \\ \notag &\quad \quad  - \sum_{x,y \in \eta_*(\cu),\,y \sim x} (p\cdot(x-y)) \bigl( w(x) - w(y)\bigr)  \\ \notag &\quad \quad = \sum_{x,y \in \eta_*(\cu),\,y \sim x} (\xi_l(x,\cu,p,q) - \xi_l(y,\cu,p,q))(w(x) - w(y))\,.
\end{align}
\item Second variation of $J_l(\cu,p,q)$ at $\xi_l(\cdot,\cu,p,q)$ and quadratic response\rv{:}
\begin{equation} \label{e.secvarJ}
	\begin{aligned}
	&J_l(\cu,p,q) - \frac{1}{|\cu|} \Biggl[  \sum_{x,y \in \eta_*(\cu),\,y \sim x} - \frac12 (w(x) - w(y))^2 - (p \cdot (x-y))(w(x) - w(y)) \\ &\quad \quad \quad +(1 - \size(\cu)^{-1})^{-1}\sum_{x,y \in \Zd\cap \cu,\, y \sim x} (q \cdot (x-y))([w]_l(x) - [w]_l(y)) \Biggr] \\
	&= \frac{1}{|\cu|}\sum_{x,y \in \eta_*(\cu),\,y \sim x} \frac12 \bigl(w(x)- \xi_l(x,\cu,p,q) - (w(y) - \xi_l(y,\cu,p,q))\bigr)^2\,.
	\end{aligned}
\end{equation}
\item Characterization of $J_l(\cu,p,q)$ in terms of the energy of $\xi$\rv{:} 
\begin{align}
	\label{e.secondvarmaster}
J_l(\cu,p,q) = \frac{1}{|\cu|} \sum_{x,y \in \eta_*(\cu),\,y \sim x} 
\frac12(\xi_l(x,\cu,p,q) - \xi_l(y,\cu,p,q))^2\,. 
\end{align}

\end{enumerate}

\end{lemma}
\begin{proof}
	As in the previous proof, we let $u_l(\cdot, \cu,q)$ denote the minimizer to the variational problem for the Neumann quantity $\mu_*^{(l)}(\cu,q).$ Fixing $w \in \mathcal{A}(\eta_*(\cu)),$ for any $h \in (0,1],$ comparing the $\mu_*^{(l)}$ energy of $u_l(\cdot, \cu,q) + h w$ with that of $u_l(\cdot,\cu,q)$, and using that coarsening is a linear operation, we find 
	\begin{align*}
	&	- h \sum_{x,y \in \eta_*(\cu),\,y \sim x} (u_l(x,\cu,q) - u_l(y,\cu,q)) (w(x) - w(y))
                - \frac{h^2}{2} \sum_{x,y \in \eta_*(\cu),\,y \sim x} (w(x) - w(y))^2 \\
        &\quad \quad + h(1 - \size(\cu)^{-1})^{-1} \sum_{x,y \in \Zd \cap \cu,\,y \sim x} (q \cdot(x-y)) \bigl( [w]_l (x) - [w]_l(y)\bigr) \leqslant 0\,.         
		\end{align*}
	Dividing through by $h$ and sending $h \to 0^+$ implies 
	\begin{align*}
	& (1 - \size(\cu)^{-1})^{-1}\sum_{x,y \in \Zd \cap \cu,\,y \sim x} (q \cdot(x-y)) \bigl( [w]_l (x) - [w]_l(y)\bigr) \\ &\quad \quad 	  \leqslant  \sum_{x,y \in \eta_*(\cu),\,y \sim x} (u_l(x,\cu,q) - u_l(y,\cu,q)) (w(x) - w(y))\,.
	\end{align*}
Taking instead $h \in [-1,0),$ dividing through by $h$ in the preceding display reverses the sign; sending $h \to 0^-$ yields the opposite inequality:
\begin{align*}
	& (1 - \size(\cu)^{-1})^{-1}\sum_{x,y \in \Zd \cap \cu,\,y \sim x} (q \cdot(x-y)) \bigl( [w]_l (x) - [w]_l(y)\bigr) \\ &\quad \quad 	  \geqslant  \sum_{x,y \in \eta_*(\cu),\,y \sim x} (u_l(x,\cu,q) - u_l(y,\cu,q)) (w(x) - w(y))\,.
\end{align*}
These two displays yield \eqref{e.firstvarmu*}. 

\smallskip
Proof of \eqref{e.secondvarmu*}. 

Expanding squares in the energy of $ u_l (\cdot, \cu,q) + w$ and using the first variation yields that for any $w \in \mathcal{A}(\cu),$ 
\begin{align*}
	&- \sum_{x,y \in \eta_*(\cu),\,y \sim x} (u_l(x,\cu,q) - u_l(y,\cu,q) + w(x) - w(y))^2 \\ 
	&\quad 
	=2|\cu| \mu_*^{(l)}(\cu,q) - \sum_{x,y \in \eta_*(\cu),\,y \sim x} (w (x) - w(y))^2\,.
\end{align*}
Rearranging this results in \eqref{e.secondvarmu*}. 

\smallskip

Proof of \eqref{e.firstvarmaster} and \eqref{e.secondvarmaster}. 

We use the characterization in Lemma~\ref{l.J.rep.formula} for the variational quantity $J_l,$ and proceed as above. The details are omitted. The identity \eqref{e.secondvarmaster} is proved by inserting the first variation condition \eqref{e.firstvarmaster} into Lemma~\ref{l.J.rep.formula}. 
\end{proof}

\section{Iteration} \label{s.iteration}
Our goal in this section is to iterate up the scales and eventually prove an algebraic rate of homogenization. Our first technical ingredient is a Caccioppoli inequality; it will be used in the subsequent convex duality estimate. 

\begin{lemma}[Caccioppoli's estimate]
	\label{l.caccioppoli}
	Let $\cu \in \mathcal{G}_l$, and $u \in \mathcal{A}(\eta_*(\cu)).$ Then there exists a constant $C(d,\lambda) > 0$ such that for any $ k \in \RR,$ 
	\begin{align}
		\label{e.cacciop}
	\frac1{|\cu|}	\sum_{x,y \in \eta_*(\cu),\,y \sim x} (u(x) - u(y))^2 &\leqslant C\size(\cu)^{-2}   \avsum_{z \in 3^l \Zd \cap \cu} ([u]_l(z) -k)^2\,.
	\end{align}
\end{lemma}

\begin{proof}
	Let~$x \in \eta_*(\cu)$ so that we have that 
	\begin{equation*}
		\sum_{y \sim x; y \in  \eta_*(\cu)} (u(y) - u(x)) = 0\,.
	\end{equation*}
First we prove~\eqref{e.cacciop} with~$k = 0.$ If~$\rho \in C^1_c(\cu) $ is a nonnegative cutoff function, then as in the usual continuum proof of Caccioppoli's inequality we will test the above equation with~$v := \rho^2 u.$ With this choice, let us note that for any~$x,y \in \eta_*(\cu)$ with $x \sim y,$ one has
\begin{align*}
	v(x) - v(y) &= \rho^2(x) u(x) - \rho^2(y) u(y) \\
	&= \frac{\rho^2(x) + \rho^2(y)}{2}(u(x) - u(y)) + \frac{\rho^2(x) - \rho^2(y)}{2} (u(x) + u(y))\,.
\end{align*}
Multiplying the equation for $u$ by $v(x)  = \rho^2 (x) u(x),$ and summing over $x \in \eta_*(\cu)$ yields
\begin{align*}
	0 &=\!\!\! \sum_{x \in \eta_*(\cu)} \sum_{y \in \eta_*(\cu): y \sim x} (u(x) - u(y))(v(x) - v(y))\\
	&=\!\! \!\sum_{x \in \eta_*(\cu)} \sum_{y \in \eta_*(\cu): y \sim x} \frac{\rho^2(x) {+} \rho^2(y)}{2} (u(x) {-} u(y))^2 + \frac{\rho^2(x) {-} \rho^2(y)}{2}(u(x) {-} u(y))(u(x) {+} u(y))\\
	&\geqslant \!\!\! \sum_{x \in \eta_*(\cu)} \sum_{y \in \eta_*(\cu): y \sim x} \!\!\! \frac{\rho^2(x) {+} \rho^2(y)}{4}(u(x) {-} u(y))^2 - \!\!\!\!\sum_{x \in \eta_*(\cu)} \sum_{y \in \eta_*(\cu): y \sim x} \!\!(\rho(x) {-} \rho(y))^2  (u(x) {+} u(y))^2 ,
\end{align*}
where we used Cauchy-Schwarz, and that $\rho^2 (x) + \rho^2(y) \leqslant (\rho(x) + \rho(y))^2$ since $\rho \geqslant 0.$

\smallskip

Let~$\cu^\prime \subsetneq \cu^{\prime\prime}\subsetneq \cu$ be cubes with the same center as~$\cu$, but such that~$\size(\cu^\prime) = \frac{\size(\cu)}{9},$ and~$\size(\cu'') = \frac{\size(\cu)}{3}$, so that~$\mathrm{dist}(\cu', \partial \cu'' ) = \size(\cu'').$ Suppose now that~$\rho \geqslant 0$ is chosen to satisfy~$\indc_{\cu'} \leqslant \rho \leqslant \indc_{\cu''},$ and~$|\nabla \rho| \leqslant C \size(\cu)^{-1}\indc_{\cu'' \setminus \cu'}.$ 
Inserting this choice in the prior display and rearranging, we obtain 
\begin{align*}
	\sum_{x \in \eta_*(\cu^\prime)} \sum_{y \sim x, y \in \eta_*(\cu')} (u(x) - u(y))^2 & \leqslant \sum_{x \in \eta_*(\cu)} \sum_{y \in \eta_*(\cu): y \sim x} (\rho(x) {-} \rho(y))^2  (u(x) {+} u(y))^2\\
	&\leqslant \frac{C}{\size(\cu)^2} \sum_{x \in \eta_*(\cu)} \sum_{y \in \eta_*(\cu''\setminus \cu'): y \sim x} (u(x) + u(y))^2\\
	&\leqslant \frac{C}{\size(\cu)^2} \sum_{x \in \eta_*(\cu)} u(x)^2\,,
\end{align*}
where in the last line, we used that since $\cu \in \mathcal{G}_l,$ \rv{by the first item in Lemma 2.3,} the number of neighbors of any point is~$\leqslant C(d,\lambda).$ 
\noindent 
Finally, using the triangle inequality along with Lemma~\ref{l.Poincare},
\begin{align*}
	\avsum_{x \in \eta_*(\cu)} u(x)^2\ &\rv{\leqslant  C}\sum_{z \in 3^l\Zd \cap \cu} \frac{|\cu_l|}{|\cu|}\avsum_{x \in \eta_*(z + \cu_l)} u(x)^2 \\
	&\leqslant C\sum_{z \in 3^l \Zd \cap \cu} \frac{|\cu_l|}{|\cu|}\avsum_{x \in \eta_*(z + \cu_l)} (u(x) - [u]_l(x))^2 + ([u]_l(x))^2\\
	&\leqslant C\frac{|\cu_l|}{|\cu|} \sum_{z \in 3^l \Zd \cap \cu}3^{2l} \avsum_{x \in \eta_*(z + \cu_l)} \sum_{y \in \cu, y \sim x} (u(x) - u(y))^2 + \avsum_{z \in 3^l \Zd \cap \cu} ([u]_l(z))^2 \\
	&\leqslant C3^{2l} \avsum_{x \in \eta_*(\cu)} \sum_{y \in \eta_*(\cu), y \sim x} (u(x) - u(y))^2  +  \avsum_{z \in 3^l \Zd \cap \cu} ([u]_l(z))^2\,.
\end{align*}
\smallskip
Combining the previous two displays, and buckling completes the proof of~\eqref{e.cacciop} with the choice~$k = 0.$ To prove the inequality for~$k \neq 0$ we simply test the equation by~$v_k := \rho^2 (u-k)$ instead of~$v$ above, and repeat the subsequent estimates. 
\end{proof}

Our next main task is to prove the convex duality estimate. In its statement, we use the following quantity $\tau_m,$ which denotes the  \emph{expected additivity defect at scale $3^m$ }:
\begin{align}
	\label{e.taudef}
	\tau_m \!:=\!\Bigl( |\ahom_l(\cu_m) \indc_{\{ \cu_{m}\in \mathcal{G}_l \}} {-} \ahom_l(\cu_{m-1})\indc_{\{ \cu_{m-1}\in \mathcal{G}_l \}}| {+} |\ahom_{*,l}(\cu_m) \indc_{\{ \cu_{m}\in \mathcal{G}_l \}}{-} \ahom_{*,l}(\cu_{m-1})\indc_{\{ \cu_{m-1}\in \mathcal{G}_l \}}|\Bigr)\,,
\end{align}
where we recall the coarse-grained matrices \rv{$\ahom_l, \ahom_{*,l}$} introduced prior to Corollary \ref{c.order}. 

An elementary estimate about the quantity $\tau_m$ that we will make use of, is as follows. Rewriting the definition of $J_l(\cu,p,q),$ for any $\cu \in \mathcal{G}_l,$ and $p,q \in \Rd,$ we find 
\rv{\begin{align*}
	J_l(\cu,p,q) &= \frac12 p \cdot \a_l(\cu)p + \frac12 q \cdot \a_{*,l}^{-1} (\cu)q - p \cdot q\\
	&=
	\frac12 p\cdot (\a_l(\cu)  - \a_{*,l}(\cu))p +  \frac12 (q - \a_{*,l}(\cu)p) \cdot \a_{*,l}^{-1}(\cu)(q - \a_{*,l}(\cu)p)\,.
\end{align*}
}
Consequently, for the standard basis ~$\{e_1,\ldots, e_d\}$ of $\Rd,$ making the choices of~\rv{$p=e_i,q= \a_{*,l}(\cu),$} and for $m \in \N$ such that $\cu_m, \cu_{m+1} \in \mathcal{G}_l,$ taking expectations, and summing, we find 
\begin{align*}
	&\sum_{i=1}^d e_i \cdot\bigl(\ahom_l(\cu_m) - \ahom_{*,l} (\cu_m)  - \bigl(\ahom_l(\cu_{m+1}) - \ahom_{*,l} (\cu_{m+1})\bigr) \bigr) e_i\\
	& \  =\sum_{i=1}^d \E \bigl[ J_l(\cu_m,e_i, \ahom_{*,l}(\cu_m)e_i) \indc_{\{\cu_m \in \mathcal{G}_l\}}\bigr] - \E \bigl[ J_l(\cu_{m+1} ,e_i, \ahom_{*,l}(\cu_{m+1} e_i))\indc_{\{\cu_{m+1} \in \mathcal{G}_l\}}\bigr]  \,,
\end{align*}
Now, by Corollary \ref{c.order}, we know that when $\cu_m, \cu_{m+1} \in \mathcal{G}_l,$ then
\begin{equation*}
	\rv{\ahom_l}(\cu_m) - \rv{\ahom_l}(\cu_{m+1}) + C3^{-(m-l)}  \geqslant 0, \quad \mbox{ and } \ahom_{*,l} (\cu_{m+1}) - \ahom_{*,l}(\cu_m) + C3^{-(m-l)/2} \geqslant 0\,. 
\end{equation*}
As the trace of matrices supplies a norm that is equivalent to the usual Frobenius norm on positive symmetric matrices, it follows  from the previous display that 
	\begin{align}
	\label{e.control tau by J}
	\tau_{m+1} &\leqslant C \sum_{i=1}^d \E \bigl[ J_l(\cu_m,e_i, \ahom_{*,l}(\cu_m)e_i) \indc_{\{\cu_m \in \mathcal{G}_l\}}\bigr] - \E \bigl[ J_l(\cu_{m+1} ,e_i, \ahom_{*,l}(\cu_{m+1} e_i))\indc_{\{\cu_{m+1} \in \mathcal{G}_l\}}\bigr] \notag  \\ 
&\qquad + C 3^{-(m-l)/2}  \notag\\
&\leqslant  C \sum_{i=1}^d \E \bigl[ J_l(\cu_m,e_i, \ahom_{*,l}(\cu_m)e_i) \indc_{\{\cu_m \in \mathcal{G}_l\}}\bigr] - \E \bigl[ J_l(\cu_{m+1} ,e_i, \ahom_{*,l}(\cu_{m+1} e_i)) \indc_{\{\cu_{m+1} \in \mathcal{G}_l\}}\bigr] \,.   
	\end{align}
We next give an upper bound estimate on the expectation of~$J$ in terms of~$\tau_n$ and the variance of~$\a_{*,l}^{-1}$ at different scales.

\begin{lemma}[Convex duality estimate] \label{l.convex_duality_lemma}
	There exists a constant~$C(d,\la) < \infty$ such that, for every~$m \in \N$ with~$m > l$, with~$\cu_m \in \mathcal{G}_l$ and~$\{z + \cu_n : z \in 3^n \Zd \cap \cu_m\} \subseteq \mathcal{G}_l,$ and for every~$p,q \in B_1$, there holds 
		\begin{align}
		 \label{e.convexduality}
				J_l(\cu_{m-1},p,q) &\leqslant C\sum_{n=l}^m 3^{-(m-n)} \avsum_{z \in 3^n \Zd \cap \cu_m} (J_l(z + \cu_n,p,q) - J_l(\cu_m,p,q))\notag \\
				& \qquad+ C \sum_{n=l}^m 3^{-(m-n)} \avsum_{z \in 3^n \Zd \cap \cu_m} \bigl| \a_{*,l}^{-1} (z + \cu_n)q - p|^2 + C 3^{-2(m-l)} \,.
		\end{align}
In particular, for every~$e \in B_1,$ 
\begin{align}
	\label{e.bound J by tau and var}
	\E \bigl[ J_l(\cu_m,e, \ahom_{*,l} (\cu_m)e)\indc_{\{\cu_m \in \mathcal{G}_l\}}\bigr] &\leqslant C \sum_{n=l}^{m-1} 3^{-(m-n)} (\tau_n + \var\bigl[ \a_{*,l}^{-1}(\cu_n)\indc_{\{\cu_n \in \mathcal{G}_l\}} \bigr]) \notag \\ &\qquad \quad  +  C 3^{-(m-l)/2}  \rv{+ C (m-l) 3^{(m-l)d}\exp (-c 3^l )}  \,.
\end{align}

\end{lemma}
\begin{proof} Recall that for any~$m \in \NN, m > l$ with $\cu_m \in \mathcal{G}_l,$ the function~$\xi_l(\cdot,\cu_m ,p,q)$ denotes the unique maximizer for the quantity~$J_l(\cu_m, p, q)$ normalized to have zero mean on~$\cu_m.$ We fix~$p, q \in B_1;$ and suppress the dependence of various quantities on these vectors for the time being in an attempt to alleviate notation. Thus we write~$\xi_l(\cdot,\cu_m)$ in place of~$\xi_l(\cdot,\cu_m,p,q)$, and so on, for much of this proof. In steps 1 and 2 below, we assume that~$\cu_m \in \mathcal{G}_l$ and~$\{z + \cu_n: z \in 3^n \Zd \cap \cu_m, n = l, \ldots, m\}\subseteq \mathcal{G}_l.$ This in particular implies that~$\xi_l(\cdot, z + \cu_n, p , q)$ is well-defined for each $z \in 3^n \Zd \cap \cu_m.$ 	
	\smallskip
	
	\emph{Step 1.} We claim that, for every~$m \in \NN, m > l,$  if~$\cu_m \in \mathcal{G}_l,$ and~$p,q \in B_1,$  we have the estimate
	\begin{equation} 
	\label{e.step1}
J_l (\cu_{m-1})  \leqslant C 3^{-2m} \inf_{k \in \RR} \|[\xi_l(\cdot,\cu_m)] - k\|_{\underline{L}^2}^2  
		+ \avsum_{z \in 3^{m-1}\Zd \cap \cu_m} \!\!\! 2 \bigl( J_l(z + \cu_{m-1}) - J_l(\cu_m)\bigr) \,.
	\end{equation} 
By the variational characterization in \eqref{e.secondvarmaster} for~$J_l,$ and an elementary algebraic inequality, we find 
\begin{align*}
	J_l(\cu_{m-1}) &= \frac{1}{|\cu_{m-1}|} \sum_{x,y  \in \eta_*(\cu_{m-1}),\, y \sim x} \frac12(\xi_l(x, \cu_{m-1}) - \xi_l(y,\cu_{m-1}))^2\\
	&\leqslant  \frac{1}{|\cu_{m-1}|} \sum_{x,y \in \eta_*(\cu_{m-1}),\, y \sim x} (\xi_l(x,\cu_m) - \xi_l(y,\cu_m))^2 \\
	& \, + \frac{1}{|\cu_{m-1}|}\!\sum_{x,y \in \eta_*(\cu_{m-1}),\,y \sim x} \!\! \bigl(\xi_l(x,\cu_m) {-} \xi_l(y,\cu_m) {-} (\xi_l(x,\cu_{m-1}) {-} \xi_l(y,\cu_{m-1}))\bigr)^2\\
	&\leqslant  \frac{1}{|\cu_{m-1}|} \sum_{x,y \in \eta_*(\cu_{m-1}),\, y \sim x} (\xi_l(x,\cu_m) - \xi_l(y,\cu_m))^2 \\
	&\quad \quad \quad  + 3^d \avsum_{z \in 3^{m-1}\Zd \cap \cu_m} 2\bigl( J_l(z+\cu_{m-1}) - J_l(\cu_m)\bigr) \,.
\end{align*}
In the last step, we used the quadratic response and second variation from \eqref{e.secvarJ}. 
By Caccioppoli inequality (Lemma~\ref{l.caccioppoli}, with~$\cu^\prime = \cu_{m-1}, \cu = \cu_m$), the first term estimates as 
\begin{equation*}
	\frac{1}{|\cu_{m-1}|}\sum_{x,y \in \eta_*(\cu_{m-1}),\, y \sim x} (\xi_l(x,\cu_m) - \xi_l(y,\cu_m))^2 
	\quad \leqslant C3^{-2m} \inf_{k \in \RR} \avsum_{x \in 3^l\Zd \cap \cu_m} \bigl([\xi_l(\cdot, \cu_m)]_l - k\bigr)^2 \,.
	\end{equation*}
Taking infimum over~$k \in \RR$ completes the goal of Step 1. 
\smallskip

\emph{Step 2.}  We use the multiscale Poincar\'e inequality~\cite[Proposition A.2]{AD16} to estimate the first term on the right-hand side of~\eqref{e.step1}. Applying this inequality to the coarsened function~$[\xi_l(\cdot,\cu_m)]_l$, we find 
\begin{align} \label{e.multiscalePoincare}
&3^{-m} \inf_{k \in \RR} \Biggl(\frac{1}{|\cu_{m-1}|}\sum_{x \in 3^l \Zd \cap \cu_m} ([\xi_l(\cdot,\cu_m)]_l - k)^2\Biggr)^{\!\!\sfrac12}  \\ \notag
&\quad \quad \leqslant   C3^{-(m-l)} \|\nabla [\xi_l(\cdot,\cu_m)]_l\|_{\underline{L}^2(\cu_m)}  + C \sum_{n=l}^{m-1} 3^{n-m} \Biggl( \avsum_{ z \in 3^n \Zd \cap \cu_m} |(\nabla [\xi(\cdot,\cu_m)]_l)_{y + \cu_n}|^2\Biggr)^{\!\!\sfrac12}\,.
\end{align}
For the second term by the triangle inequality, we note 
\begin{align*}
	&\Biggl( \avsum_{z \in 3^n \Zd \cap \cu_m} |(\nabla [\xi_l(\cdot,\cu_m)]_l)_{y+\cu_n}|^2\Biggr)^{\!\!\sfrac12} \\
	&\quad  \leqslant \Biggl( \avsum_{z \in 3^n \Zd \cap \cu_m} |(\nabla [\xi_l(\cdot,\cu_m)]_l )_{y+\cu_n}- (\nabla [\xi_l(\cdot,\cu_n)]_l)_{y+\cu_n})|^2\Biggr)^{\!\!\sfrac12} \\ &\quad \quad\quad + \Biggl( \avsum_{z \in 3^n \Zd \cap \cu_m} | \a_{*,l}^{-1} (y + \cu_n) q - p|^2\Biggr)^{\!\!\sfrac12}\,,
\end{align*}
where, in the last line, we used the following computation: 
\begin{align*}
(\nabla [\xi_l(\cdot,\cu_n)]_l)_{y+\cu_n} &= (\nabla [\xi_l(\cdot,\cu_n,p,q)]_l)_{y+\cu_n} \\&  = (\nabla [\xi_l(\cdot,\cu_n,0,q)]_l)_{y+\cu_n}-(\nabla [\xi_l(\cdot,\cu_n,p,0)]_l)_{y+\cu_n}
\\ &
= \a_{*,l}^{-1}(y+\cu_n)q - p\,,
\end{align*}
which follows from \eqref{e.firstvarmaster}. For the first term, we simply use the second variation (see \eqref{e.secvarJ}): 
\begin{align*}
	&\avsum_{z \in 3^n \Zd \cap \cu_m} \bigl| (\nabla [\xi_l(\cdot,\cu_m)]_l )_{y+\cu_n}- (\nabla [\xi_l(\cdot,\cu_n)]_l)_{y+\cu_n})\bigr|^2 \\ &\quad \quad \quad \quad 
	\leqslant 2 \avsum_{z \in 3^n \Zd \cap \cu_m} \bigl( J_l(y+\cu_n,p,q) - J_l(\cu_m,p,q) \bigr)\,. 
\end{align*}
Finally, for the first term on the right-hand side of \eqref{e.multiscalePoincare}, 
\[\|\nabla [\xi_l(\cdot,\cu_m)]_l\|_{\underline{L}^2(\cu_m)} \leqslant 2 J_l(\cu_m,p,q) \leqslant C\,.\]

\noindent Combining the preceding estimates yields 
\begin{align*}
	J_l(\cu_{m-1},p,q) &\leqslant C\sum_{n=l}^m 3^{-(m-n)} \avsum_{z \in 3^n \Zd \cap \cu_m} (J_l(z + \cu_n,p,q) - J_l(\cu_m,p,q))\\
	& \qquad + C \sum_{n=l}^m 3^{n-m} \avsum_{z \in 3^n \Zd \cap \cu_m} \bigl| \a_{*,l}^{-1} (z + \cu_n)q - p|^2 + C 3^{-2(m-l)} \,.
\end{align*}
This is \eqref{e.convexduality}. 
\smallskip

\emph{Step 3.} We complete the argument continuing to closely follow \rv{\cite[Lemma 5.19]{AK22}}. 
We define $A_m$ to be the event given by
\begin{equation*}
	A_{m} := \bigl\{ z + \cu_n  \in \mathcal{G}_l \mbox{ for each } n \in\{ l+1 , \cdots, m\} \ \mbox{and} \ z \in 3^n \Zd \cap \cu_m \bigr\}\,.
\end{equation*}
\rv{ By Proposition~\ref{p.Penrose},~\eqref{e.Gl.nope} and a union bound, we find 
\begin{align*}
\P [A_m^c] 
\leq 
\sum_{n=l+1}^m \sum_{z\in 3^n\Zd\cap \cu_m} 
\P [ z + \cu_n \not \in \mathcal{G}_l ] 
& 
\leq 
\sum_{n=l+1}^m 
3^{d(m-l)}  
\exp( -c 3^l) 
\notag \\ & 
= (m-l) 3^{d(m-l)}  
\exp( -c 3^l)
\,.
\end{align*}
}
Choosing~$p = e \in S^{d-1},q = \rv{\ahom_{*,l}(\cu_m)e}$, \rv{towards analyzing the expectations on both sides of} \eqref{e.convexduality}, \rv{we note that for the middle term on the right-hand side of this estimate,} by the triangle inequality we obtain that 
\begin{align*}
\lefteqn{
\E\Biggl[ \Biggl( \sum_{n=l}^{m-1} 3^{n-m} \avsum_{z \in 3^n \Zd \cap \cu_m} | \a_{*,l}^{-1} (z + \cu_n) \ahom_{*,l}(\cu_m)e - e|^2 \Biggr)\indc_{A_m}\Biggr]
} \qquad &
\notag \\ & 
	= 	\E\Biggl[ \Biggl(\sum_{n=l}^{m-1} 3^{n-m} \avsum_{z \in 3^n \Zd \cap \cu_m} |( \a_{*,l}^{-1} (z + \cu_n) - \ahom_{*,l}^{-1}(\cu_m)) \ahom_{*,l}(\cu_m)e |^2\Biggr) \indc_{A_m}\Biggr]\\
	&\leqslant C \sum_{n=l}^{m-1} 3^{n-m}  \var \bigl[ \a_{*,l}^{-1}(\cu_n) \indc_{\{\cu_n \in \mathcal{G}_l\}}\bigr] + C \sum_{n=l}^{m-1} 3^{n-m} \bigl| \ahom_{*,l}^{-1}(\cu_n) - \ahom_{*,l}^{-1}(\cu_m) \bigr|^2\,.
\end{align*}
\rv{ For the second term of the previous display, by the triangle inequality,
\begin{align*}
\sum_{n=l}^{m-1} 
3^{n-m} \bigl| \ahom_*^{-1}(\cu_n) - \ahom_*^{-1}(\cu_m) \bigr|^2 
\leqslant 
C \! \sum_{n=l}^{m-1} 3^{n-m}
\!\!
\sum_{k=n+1}^m \tau_k 
&
=
C \sum_{k=l+1} ^ m 
\tau_k 
\sum_{n=l}^{k-1} 
3^{n-m} 
\notag \\ & 
\leqslant C \! \sum_{n=l}^{m} 3^{k-m}\tau_k\,.
\end{align*}
}
In addition, we note that 
\begin{align*}
	\E \Biggl[ \Biggl(\sum_{n=l}^{m-1} 3^{n-m} \avsum_{z \in 3^n \Zd \cap \cu_m} (J_l(z+\cu_n,p,q) - J_l(\cu_m,p,q)) \Biggr)\indc_{A_m}\Biggr]
& \leqslant \sum_{n=l}^{m-1} 3^{n-m} \!\sum_{k=n+1}^m \tau_k 
	\notag \\ & 
	\leqslant C \sum_{n=l}^{m-1} 3^{n-m}\tau_n\,.
\end{align*}
Finally, by Corollary \ref{c.order},
\begin{align*}
\lefteqn{
	\E[J_l(\cu_m,e,\ahom_*(\cu_m)e) \indc_{\{\cu_m \in \mathcal{G}_l\}}] 
	} \qquad & 
	\notag \\ &
	\leqslant \E[J_l(\cu_{m-1},e,\ahom_*(\cu_m)e)\indc_{\{\cu_{m-1} \in \mathcal{G}_l\}}] + C3^{-(m-l)/2}\\
	&= \E[J_l(\cu_{m-1},e,\ahom_*(\cu_m)e)\indc_{\{\cu_{m-1} \in \mathcal{G}_l\}\cap A_m}] \\ &\qquad +   \E[J_l(\cu_{m-1},e,\ahom_*(\cu_m)e)\indc_{\{\cu_{m-1} \in \mathcal{G}_l\}\setminus A_m}]+ C3^{-(m-l)/2}\\
	&\leqslant \E[J_l(\cu_{m-1},e,\ahom_*(\cu_m)e)\indc_{ A_m}] + C \P [A_m^c] + C3^{-(m-l)/2}\\
	&\leqslant \E[J_l(\cu_{m-1},e,\ahom_*(\cu_m)e)\indc_{ A_m}] 
+ C3^{-(m-l)/2}
\rv{ + C (m-l) 3^{(m-l)d}\exp (-c 3^l ) }
\,.
\end{align*}
Combining these estimates yields  \eqref{e.bound J by tau and var}. 
\end{proof}
In the next estimate we control the variance term in \eqref{e.bound J by tau and var} by the quantity $\tau_m,$ the expected additivity defect. 
\begin{lemma}[Decay of variance estimate] \label{l.variance} There exists~$\theta(d) \in (0,1)$, and~$C(d,\la) < \infty$ such that
\begin{equation*}
	\var \Bigl [ \a_{*,l}^{-1}(\cu_m) \indc_{\{\cu_m \in \mathcal{G}_l\}}\Bigr ]
	\leq 
C\sum_{n=l}^m
\theta^{m-n}\tau_n
+
C\theta^{2(m-l)} 
\,.
\end{equation*}
Here, the quantities $\tau_n$ are as introduced in \eqref{e.taudef}. 
\end{lemma}
\begin{proof}

Observe that by Cauchy-Schwarz, stationarity, and noting that for at least one pair of cube centers~$z \neq z^\prime \in 3^n \Zd \cap \cu_m,$ the random variables~$\a_{*,l}^{-1}(z + \cu_n) \indc_{\{z + \cu_n \in \mathcal{G}_l\}}$ and~$\a_{*,l}^{-1}(z^\prime + \cu_n)\indc_{\{z^\prime + \cu_n \in \mathcal{G}_l\}}$ are independent so that their covariance vanishes, we have
\begin{align*}
&\var \Biggl [ \sum_{z\in 3^n\Zd \cap \cu_m} \a_{*,l}^{-1}(z+\cu_n)   \indc_{\{z + \cu_n \in \mathcal{G}_l\}}\Biggr ]
\\ &
 =
\sum_{z,z' \in 3^n\Zd \cap \cu_m} 
\cov \Biggl [ \a_{*,l}^{-1}(z+\cu_n)  \indc_{\{z + \cu_n \in \mathcal{G}_l\}}\,, \,  \a_{*,l}^{-1}(z'+\cu_n)  \indc_{\{z' + \cu_n \in \mathcal{G}_l\}}\Biggr ]
\\ &  
\leq 
\bigl (3^{2d(m-n)} -1\bigr )
\var \bigl [ \a_{*,l}^{-1}(\cu_n)  \indc_{\{ \cu_n \in \mathcal{G}_l\}} \bigr ]
\,.
\end{align*}
Thus 
\begin{align}
\label{e.variance_est}
\var \Biggl [ \avsum_{z\in 3^n\Zd \cap \cu_m} \a_{*,l}^{-1}(z+\cu_n) \indc_{\{z + \cu_n \in \mathcal{G}_l\}}  \Biggr ]
\leq 
\frac{3^{2d(m-n)} -1}{3^{2d(m-n)}}
\var \bigl [ \a_{*,l}^{-1}(\cu_n)  \indc_{\{z + \cu_n \in \mathcal{G}_l\}}  \bigr ]\,.
\end{align}
Now, by the triangle inequality, for any $n < m,$ we find

\begin{align*}
&\var^{\frac12} \Bigl [ \a_{*,l}^{-1}(\cu_m)  \indc_{\{\cu_m\in \mathcal{G}_l\}} \Bigr ]
\\ & \quad \quad
\leq
\var^{\frac12} \biggl [ 
\a_{*,l}^{-1} (\cu_m)  \indc_{\{\cu_m\in \mathcal{G}_l\}} 
-
\avsum_{z\in 3^n\Zd \cap \cu_m} \a_{*,l}^{-1}(z+\cu_n)  \indc_{\{z + \cu_n \in \mathcal{G}_l\}}  
\biggr ]
\\ &\quad \quad \quad\quad +
\var^{\frac12} \biggl [ \avsum_{z\in 3^n\Zd \cap \cu_m} \a_{*,l}^{-1}(z+\cu_n)  \indc_{\{z + \cu_n \in \mathcal{G}_l\}}  \biggr ]
\\ & \quad \quad
\leq
C
\biggl|\E \biggl [  \a_{*,l}^{-1} (\cu_m)  \indc_{\{\cu_m\in \mathcal{G}_l\}} 
-
\avsum_{z\in 3^n\Zd \cap \cu_m} \a_{*,l}^{-1}(z+\cu_n)  \indc_{\{z + \cu_n \in \mathcal{G}_l\}}
\biggr ]\biggr|^{\frac12}
\\&
\quad \quad\quad\quad  +
\var^{\frac12} \biggl [ \avsum_{z\in 3^n\Zd \cap \cu_m} \a_{*,l}^{-1}(z+\cu_n) \indc_{\{z + \cu_n \in \mathcal{G}_l\}}    \biggr ] 
\\ & \quad \quad 
\leq
 \biggl( \sum_{k=n}^m 3^{-(m-k)}\tau_k\biggr)^{\sfrac12}
+
\theta_{m,n}\var ^{\frac12}\bigl [\a_{*,l}^{-1}(\cu_n) \indc_{\{ \cu_n \in \mathcal{G}_l\}}   \bigr ] 
\,,
\end{align*}
where $\theta_{m,n} := (3^{2d(m-n)}-1)/3^{2d(m-n)}<1$ and we inserted \eqref{e.bound J by tau and var}. Plugging in ~$n=m-1$ yields
\begin{equation*}
\var^{\frac12} \Bigl [ \a_{*,l}^{-1}(\cu_m) \indc_{\{\cu_m \in \mathcal{G}_l\}}   \Bigr ]
\leq
\theta \var^{\frac12} \Bigl [ \a_{*,l}^{-1}(\cu_{m-1}) \indc_{\{\cu_{m-1} \in \mathcal{G}_l\}}  \Bigr ]
+ 
C\tau_m 
\,,
\end{equation*}
with $\theta := \theta_{m,m-1}.$ 
Iterating this down to scale $n= l,$ gives 
\begin{equation*}
\var^{\frac12} \Bigl [ \a_{*,l}^{-1}(\cu_m) \indc_{\{\cu_m \in \mathcal{G}_l\}}   \Bigr ]
\leq 
C\sum_{n=l}^m
\theta^{m-n}\tau_n
+
C\theta^{m-l} 
\,.
\end{equation*}
Squaring this gives 
\begin{equation*}
\var \Bigl [ \a_{*,l}^{-1}(\cu_m) \indc_{\{\cu_m \in \mathcal{G}_l\}} \Bigr ]
\leq 
C\sum_{n=l}^m
\theta^{m-n}\tau_n
+
C\theta^{2(m-l)} 
\,.
\end{equation*}
This completes the desired estimate. 
\end{proof}

In our next proposition we combine the preceding two lemmas, and iterate towards an algebraic rate of homogenization. 
\begin{proposition}
	\label{p.iteration}
	There exists $\alpha_0(d) \in (0,\frac12]$ and $C(d) <\infty$ such that for every $m > l,$ we have 
	\begin{equation} \label{e.iterationestimate}
		\sup_{e \in B_1} \E \bigl[ J_l (\cu_m, e, \ahom_{*,l} (\cu_m) e) \indc_{\{\cu_m \in \mathcal{G}_l\}} \bigr] \leqslant C3^{-(m-l)\alpha_0}
		+C\rv{(m-l) 3^{(m-l)d} \exp(-c3^l)} 
		\,.
	\end{equation}
\end{proposition}

\begin{proof}
We closely follow the proof of \cite[Proposition 2.11]{AKM19}. As there, it suffices to estimate 
\begin{equation*}
	F_{m,l} := \sum_{i=1}^d \E \bigl[ J_l ( \cu_m,e_i, \ahom_{*,l}(\cu_m)e_i)  \indc_{\{\cu_m \in \mathcal{G}_l\}} \bigr]\,.
\end{equation*}
Observe that \eqref{e.control tau by J} rewrites as, for each $n > l,$ 
\begin{equation*}
	\tau_{n+1} \leqslant C (F_{n,l} - F_{n+1,l})  + C3^{-(n-l)/2}\,. 
\end{equation*}
We define the weighted average 
\begin{equation*}
	\tilde{F}_{m,l} := \sum_{n=l+1}^{m} 3^{-(m-n)} F_{n,l}\,. 
\end{equation*}
For this quantity, since $F_{l+1,l} \leqslant C,$ it follows that for all $m > l,$ we have 
\begin{align*}
	\tilde{F}_{m,l} - \tilde{F}_{m+1,l} &=  \sum_{n=l+1}^m 3^{-(m-n)}F_{n,l} - \sum_{n=l }^{m-1} 3^{-(m-n)} F_{n+1,l}\\
	&\geqslant \sum_{n=l+1}^{m-1} 3^{-(m-n)}(F_{n,l} - F_{n+1,l}) - C3^{-(m-l)}\\
	&\geqslant C^{-1} \sum_{n=l}^{m-2} 3^{-(m-n)}\tau_n - C 3^{-(m-l)/2}\,.
\end{align*}
On the other hand, combining Lemmata \ref{l.convex_duality_lemma} and \ref{l.variance} with monotonicity, we find 
\begin{align*}
	\tilde{F}_{m+1,l} \leqslant \tilde{F}_{m,l}&= \sum_{n=l+1}^m 3^{-(m-n)} F_{n,l} \\
	&\leqslant  
	C \!\! \sum_{n=l}^{m-1} 3^{-(m-n)} (\tau_n + \var\bigl[ \a_{*,l}^{-1}(\cu_n) \indc_{\{\cu_n \in \mathcal{G}_l\}}\bigr]) +  C 3^{-(m-l)/2} \\
	 &\quad \quad +  C 3^{(m-l)d}\exp (-c 3^{\frac{ld}{2(d+1)}}) 
	\notag \\ &
	\leqslant
	C \!\! \sum_{n=l}^{m-1} 3^{-(m-n)} \tau_n 
	+
	C\!\!  \sum_{n=l}^{m-1} 3^{-(m-n)} 
	\biggl ( 
	C\sum_{k=l}^n
	\theta^{n-k}\tau_k
	+
	C\theta^{2(n-l)}  \!
	\biggr )
	+  C 3^{-\frac12(m-l)} \\
	&\quad \quad 
	+ C \rv{(m-l) 3^{(m-l)d} \exp(-c3^l)}  
	\notag \\ &
	\leq
	C \sum_{n=l}^{m-1} \theta^{m-n} \tau_n 
	+
	C 3^{-\alpha(m-l)} +  C \rv{(m-l) 3^{(m-l)d} \exp(-c3^l)}  
	\,,
\end{align*}
upon slightly decreasing $\theta$ closer to $0,$ if needed, and for a suitable choice of $\alpha \in (0,\frac12] $ (depending on $\theta$).

\smallskip
Combining the previous two displays, we obtain 
\begin{equation*}
	\tilde{F}_{m+1,l} \leqslant C(\tilde{F}_{m,l}  - \tilde{F}_{m+1,l} ) + C3^{-\alpha (m-l)}+  C \rv{(m-l) 3^{(m-l)d} \exp(-c3^l)} \,.
\end{equation*} 
Rearranging yields 
\begin{equation*}
	\tilde{F}_{m+1,l} \leqslant \frac{C}{C+1} \tilde{F}_{m,l} + \frac{C}{C+1} \bigl ( 3^{-\alpha(m-l)} + \rv{(m-l) 3^{(m-l)d} \exp(-c3^l)} \bigr )\,.
\end{equation*}
Introducing $\mu:= \frac{C}{C+1} < 1,$ and iterating this estimate and using that $\tilde{F}_{l,l} \leqslant C,$ yields (upon slightly increasing $\mu$ closer to $1$ if needed)
\begin{equation*}
	\tilde{F}_{m+1,l} \leqslant C\mu^{m-l} 
	+ C\bigl ( 3^{-\alpha(m-l)} + \rv{(m-l) 3^{(m-l)d} \exp(-c3^l)}  \bigr )\,.
\end{equation*}
Finally, defining $\alpha_0 := \min\{ \alpha , \frac{\log 3}{|\log \mu|}\} > 0,$ so that $\mu \leq 3^{-\alpha_0},$ together with $F_{m,l} \leqslant \tilde{F}_{m,l}$ completes the desired estimate \eqref{e.iterationestimate}. 
\end{proof}

In our next lemma we study the dependence of $\ahom_l(\cu_m)$ on $l.$ For its statement, we recall the constant $K$ in the Poincar\'e inequality from the definition of a good cube. 
\begin{lemma}
	\label{l.indep_of_l}
There exists~$C(d,\lambda) <\infty$ such that for all~$N, m \in \N,$ $N > m \geqslant l$ we have 
\begin{equation*}
|\ahom_{l+1}(\cu_N) \indc_{\{\cu_N\in \mathcal{G}_{l+1}\}}- \ahom_l(\cu_m)\indc_{\{\cu_m \in \mathcal{G}_l\}}| \leqslant C3^{-\alpha_0(N-m)} + \rv{(m-l) 3^{(m-l)d} \exp(-c3^l)}  \,.
\end{equation*}
\end{lemma}
\begin{proof}
Suppose that~$\cu_m \in \mathcal{G}_{l+1}$, so that~$\cu_m \in \mathcal{G}_l;$ suppose further that~$\cu_N\in \mathcal{G}_{l+1}$ for some~$N > m.$   Then, for any~$e \in B_1,$ since the boundary layer in the variational problem defining $\a_{l+1}(\cu_m)$ is thicker than that defining~$\a_l(\cu_m),$ one inequality is immediate by using the minimizer of the former problem as a competitor for the latter, and taking expectation: for all $e \in B_1,$ 
	\begin{equation*}
		e \cdot \ahom_{l}(\cu_m)e \leqslant  e \cdot \ahom_{l+1} (\cu_m)e\,,
	\end{equation*}
so that, by Corollary \ref{c.order} we have 
\begin{align*}
\ahom_{l}(\cu_m) \leqslant \ahom_{l+1}(\cu_m) &= \ahom_{*,l+1}(\cu_m) + \ahom_{l+1}(\cu_m) -  \ahom_{*,l+1}(\cu_m)\\
&\leqslant \ahom_{*,l+1}(\cu_N) + |\ahom_{l+1}(\cu_m) -  \ahom_{*,l+1}(\cu_m)|\\
&= \ahom_{l+1}(\cu_N) + |\ahom_{*,l+1}(\cu_N) - \ahom_{l+1}(\cu_N)| +  |\ahom_{l+1}(\cu_m) -  \ahom_{*,l+1}(\cu_m)|\\
&\leqslant \ahom_{l+1}(\cu_N) +  \sup_{e \in B_1} \E [ J_{l+1} (\cu_N, e, \ahom_{*, l+1}(\cu_N)) \indc_{\{\cu_N\in \mathcal{G}_l\}}]  \\ &\quad \quad +  \sup_{e \in B_1} \E [ J_{l+1} (\cu_m, e, \ahom_{*, l+1}(\cu_m)) \indc_{\{\cu_m\in \mathcal{G}_l\}}] \\
&\leqslant \ahom_{l+1}(\cu_N)  +  C3^{-(m-l)\alpha_0}
+C3^{(m-l)d}\exp (-c 3^{\frac{ld}{2(d+1)}}) \,.
\end{align*}
\smallskip 
\noindent 
For the other direction, we let $N \in \NN$ be such that $N > m.$   We will show that 
\begin{equation} \label{e.interlace}
	\ahom_{l+1}(\cu_N)  \leqslant \ahom_l(\cu_m) + C3^{-(N-m)}\,.
\end{equation}
If this is proven, then the claim of the lemma follows immediately upon combining it with the previous display. 

\smallskip
\noindent
To prove this claim, we tile the cube $\cu_N$ with cubes of the form $z + \cu_m, z \in 3^m \Zd,$ and fix $e \in B_1, |e| = 1.$ We define the function $w: \eta_*(\cu_N) \to \R,$ which is defined to be affine in the cubes $z + \cu_m$ that share at least one face with $\partial \cu_N,$ i.e., in such cubes, we set $w(x) := \ell_e(x).$  For all other cubes such that $z + \cu_m \cap \cu_N \neq \emptyset,$ we set $w(x) := \xi_l(x, z + \cu_m, e,0)$ if $x \in z + \cu_m.$ Observe that along the boundary edges the function $w$ is affine, gluing them together. 

Now, as $w \equiv \ell_e$ along the boundary cubes of size $3^m > 3^{l+1},$ it follows that $w$ is an admissible competitor for the variational problem defining $\frac12 e \cdot \a_{l+1}(\cu_N)e.$ This implies then that 
\begin{align*}
	&\frac12 e \cdot \a_{l+1}(\cu_N) e \leqslant \avsum_{z \in 3^m \Zd \cap \cu_N} \|\nabla w\|^2_{\underline{L}^2 (z + \cu_m)} + C3^{-(N-m)} \\
	&\quad  = \frac{|\cu_m|}{|\cu_N|} \Biggl[ \sum_{z \in \partial_*^{(m)}\cu_N} \|\nabla \ell_e\|^2_{\underline{L}^2 (z + \cu_m)} +  \sum_{z \in 3^m \Zd \cap \cu_N \setminus \partial_*^{(m)}\cu_N} \frac12 e \cdot \a_l(z + \cu_m) e \Biggr] + C3^{-(N-m)}\,. 
\end{align*}
Now, the number of boundary cubes in the boundary layer of size~$3^m$ within the cube~$\cu_N$ is~$d2^{d-1} 3^{N-m} ;$ since each of these is a good cube, by definition it follows that the boundary-layer term on the right-hand side above is controlled by~$K \frac{|\cu_m|}{|\cu_N|} d2^{d-1} 3^{N-m}$. Therefore the number of interior cubes is 
\begin{equation*}
	\frac{|\cu_N|}{|\cu_m|} - d 2^{d-1}3^{N-m}\,.
\end{equation*}
Inserting this count in the preceding display, taking expectations on both sides and using stationarity, we find
\begin{align*}
	e \cdot \ahom_{l+1}(\cu_N) &\leqslant \Bigl( 1 - d2^{d-1}3^{N-m}\frac{|\cu_m|}{|\cu_N|} \Bigr) e \cdot \ahom_l(\cu_m) e+ K \frac{|\cu_m|}{|\cu_N|}d 2^{d-1}3^{N-m} + C 3^{-(N-m)}\\
	&\leqslant e \cdot \ahom_l (\cu_m) e + K d \Bigl(\frac{2}{3^{N-m}} \Bigr)^{d-1} + C 3^{-(N-m)}\,.
\end{align*}
Absorbing the error terms together by slightly increasing $C$ if necessary, this is the claimed inequality \eqref{e.interlace}. 
\end{proof}

We are ready to define the \emph{homogenized tensor} $\ahom$. 
\begin{definition}
	We define 
	\begin{equation} \label{e.ahomdef}
		\ahom := \lim_{\rv{m,l \to \infty: m > l}} \a_l(\cu_m) \indc_{\{ \cu_{m}\in \mathcal{G}_l \}}\,. 
	\end{equation}
\end{definition}
This limit clearly exists, as can be seen by combining the preceding lemma with the fact that the sequence $\{\a_l(\cu_m)\indc_{\{ \cu_{m}\in \mathcal{G}_l \}}\}_m$ is approximately monotone by Corollary \ref{c.order}.

\smallskip
 
We are now ready to prove the main theorem of this section.

\begin{theorem}
\label{t.homogenization}
Let~$\beta \in (0,1)$. For each~$m\in\N$, let~$\mathcal{E}(m)$ be the random variable defined by
\begin{equation*}
\mathcal{E}(m) := \sum_{n= \lceil \beta m\rceil}^m 3^{n-m} \sum_{i=1}^d \avsum_{z \in 3^n \Zd \cap \cu_m} J_{\lceil \beta m\rceil}(z + \cu_n, e_i, \ahom e_i) \indc_{\{z+\cu_n \in \mathcal{G}_{\lceil \beta m\rceil}\}}
\,.
\end{equation*}
There exist~$\alpha(d,\lambda),\sigma(\beta,d) \in (0,\beta]$, constants~$c(\beta,d,\lambda)>0$ and~$C(\beta,d,\lambda)<\infty$, and a nonnegative random variable~$\X$ satisfying 
\begin{equation}
\label{e.X.integrability}
\P [\X > 3^m ] \leqslant 2\exp \bigl( -c3^{m\sigma}\bigr)\,,
\end{equation}
such that, for every~$m\in \N$ with~$3^m \geqslant \X$, 
we have that 
\begin{equation}
\label{e.goodness.beta}
\cu_m \in \mathcal{G}_{\lceil \beta m\rceil }
\end{equation}
and~$\mathcal{E}(m)$ satisfies the bound
\begin{equation}
\label{e.E.convergence}
\mathcal{E}(m) \leq C 3^{-m\alpha}\,.
\end{equation}
\end{theorem}
\begin{proof}
We first find a random minimal scale~$\Y$ for which~$3^m\geq \Y$ implies that~\eqref{e.goodness.beta} holds.
We have that, by a union bound and~\eqref{e.good},
\begin{align}
\P\Bigl [ \exists z \in 3^{\lceil \beta m\rceil}\Zd \cap \cu_m\,, \ z+\cu_{\lceil \beta m\rceil} \ \mbox{is not good} \Bigr ]
&
\leq
3^{d (m -\lceil \beta m\rceil)}
\P\Bigl [ \cu_{\lceil \beta m\rceil} \ \mbox{is not good} \Bigr ]
\notag \\ &
\leq 
2\cdot 3^{d (m -\lceil \beta m\rceil)}
\exp \Bigl( -c 3^{ \frac{d}{2d+2}\lceil \beta m\rceil}\Bigr )
\notag \\ &
\leq
2\exp \Bigl( -c 3^{ \frac{d\beta }{2d+2}m }\Bigr )
\,.
\end{align}
Therefore, by taking a union bound over all~$k\geq m$ yields
\begin{align}
\label{e.X.says.its.good}
\lefteqn{
\P\Bigl [ \exists k \geq m, \  \exists z \in 3^{\lceil \beta k\rceil}\Zd \cap \cu_k\,, \ z+\cu_{\lceil \beta k\rceil} \ \mbox{is not good} \Bigr ]
} \qquad &
\notag \\ & 
\leq 
\sum_{k=m}^\infty
\P\Bigl [ \exists z \in 3^{\lceil \beta k\rceil}\Zd \cap \cu_k\,, \ z+\cu_{\lceil \beta k\rceil} \ \mbox{is not good} \Bigr ]
\notag \\ &
\leq 
\sum_{k=m}^\infty
2\exp \Bigl( -c 3^{ \frac{d\beta  }{2d+2}k }\Bigr )
\leq 
2\exp \Bigl( -c 3^{ \frac{d\beta }{2d+2}m }\Bigr )
\,.
\end{align}
Therefore, if we define~$\Y$ by 
\begin{equation*}
\Y := 
\sup \Bigl \{
3^m
\, :\,
\exists z \in 3^{\lceil \beta k\rceil}\Zd \cap \cu_k\,, \ z+\cu_{\lceil \beta k\rceil} \ \mbox{is not good}
\Bigr \}\,,
\end{equation*}
then clearly by its definition we see that 
\begin{equation*}
3^m \geq \Y 
\implies 
\cu_m \in \mathcal{G}_{\lceil \beta m\rceil }\,,
\end{equation*}
and meanwhile what we have shown in~\eqref{e.X.says.its.good} is that 
\begin{equation*}
\Y = \O_{\frac{d\beta }{2d+2}}(C)\,.
\end{equation*}

\smallskip

We next find a larger scale~$\X \geqslant \Y$ such that~$3^m\geq \X$ ensures that~\eqref{e.E.convergence} holds.
We recall that $J_{\lceil \beta m\rceil}(\cu_m,e,\ahom_{*,l}(\cu_m)e)$ is~$\mathcal{F}(3\cu_m)$-measurable, and satisfies the inequality  (by Lemma~\ref{l.subbadditivity}), for every~$m\in\N$ and~$n\in \N \cap (\lceil \beta m\rceil, m ]$,	
	\begin{equation*}
		J_{\lceil \beta m\rceil}(\cu_m,e,\ahom e) \indc_{\{ \cu_m \in \mathcal{G}_{\lceil \beta m\rceil} \}}
		\leq 
		\avsum_{z\in3^n\Zd \cap \cu_m}
		J_{\lceil \beta m\rceil}(z+\cu_n,e,\ahom e) \indc_{\{ z+\cu_n \in \mathcal{G}_{\lceil \beta m\rceil} \}}
		+  C 3^{-(n-{\lceil \beta m\rceil})/2}\,. 
	\end{equation*}
making it approximately subadditive; finally, it satisfies the uniform bound~
\[J_{\lceil \beta m\rceil}(\cu_m, e,\ahom e)\indc_{\{ \cu_m \in \mathcal{G}_{\lceil \beta m\rceil}\}} \leqslant C |e|^2.\]
 A slight modification of \cite[Lemma A.7]{AKM19} (which is stated for exactly subadditive quantities, but the proof gives the same estimates hold in our setting) yields that, for every $m \geqslant n \geqslant {\lceil \beta m\rceil},$ we have 
\begin{align*}
&J_{\lceil \beta m\rceil}(\cu_n,e,\ahom e)\indc_{\{ \cu_n \in \mathcal{G}_{\lceil \beta m\rceil} \}} \\
&\quad \quad \leqslant 
2 \E \bigl[ J_{\lceil \beta m\rceil} (\cu_k,e, \ahom e)\indc_{\{ \cu_k \in \mathcal{G}_{\lceil \beta m\rceil} \}}\bigr] 
+ 3^{-(n-{\lceil \beta m\rceil})/2}
+ \mathcal{O}_1(C3^{-d(n-k)})\,. 
\end{align*}
Combining this with Proposition~\ref{p.iteration}, yields 
\begin{align} \label{e.Jlbound}
&J_{\lceil \beta m\rceil}(\cu_n,e,\ahom e) \indc_{\{ \cu_n \in \mathcal{G}_{\lceil \beta m\rceil} \}} \\
\notag & \quad \leqslant 
C3^{-\alpha_0(n-{\lceil \beta m\rceil})} + C3^{(n-{\lceil \beta m\rceil})d}\exp (-c 3^{\frac{d}{2(d+1)}{\lceil \beta m\rceil}}) +   \mathcal{O}_1(C3^{-d(n-k)})\\
\notag & \quad \leqslant C3^{-\alpha_0(n-{\lceil \beta m\rceil})} +  \mathcal{O}_1(C3^{-d(n-k)})\,.
\end{align}
We now select~$k:= \frac12(n + \lceil \beta m\rceil )$ to obtain, after redefining~$\alpha_0$, 
\begin{equation*}
J_{\lceil \beta m\rceil}(\cu_n,e,\ahom e) \indc_{\{ \cu_n \in \mathcal{G}_{\lceil \beta m\rceil} \}}
\leqslant C3^{-\alpha_0 (n-\lceil \beta m\rceil)} +  \mathcal{O}_1(C3^{-\frac d2(n-\beta m)})
\end{equation*}
Plugging in \eqref{e.Jlbound} we obtain 
\begin{align*}
	\mathcal{E}(m) &\leqslant C\sum_{n=\lceil \beta m\rceil}^m 3^{n - m}  \bigl( C3^{-\alpha_0 (n-\lceil \beta m\rceil)} +  \mathcal{O}_1(C3^{-\frac d2(n-\beta m)}) \bigr)\\
	&\leqslant 
	C3^{-\alpha_0 (1-\beta)m} +  \mathcal{O}_1(Cm3^{-(1-\beta )m})
	\leq 
	\mathcal{O}_1 (C3^{-\alpha_0 m})\,,
\end{align*}
where in the last inequality we redefined~$\alpha_0$.
We deduce that
\begin{equation*}
\P\bigl [ \mathcal{E}(m) > C3^{-\frac12\alpha_0 m} \bigr ]
\leq 
2\exp\Bigl ( -c 3^{\frac12\alpha_0 m}  \Bigr )
\,.
\end{equation*}
Now we define the minimal scale~$\X$ by
\begin{equation*}
\X :=
\Y + 
\sup \bigl \{ 3^m \,:\,
\mathcal{E}(m) > C3^{-\frac12\alpha_0 m} \bigl \}
\,.
\end{equation*}
Cleary~$\X$ has the property that~$3^m\geq \X$ implies~\eqref{e.goodness.beta} and~\eqref{e.E.convergence}, by definition. 
To check~\eqref{e.X.integrability}, we use a union bound to obtain
\begin{align*}
\P\bigl [ \X - \Y \geq 3^m \bigr ]
\leq 
\sum_{k=m}^\infty
\P\bigl [ \mathcal{E}(k) > C3^{-\frac12\alpha_0 k} \bigr ]
\leq 
\sum_{k=m}^\infty
2\exp\Bigl ( -c 3^{\frac12\alpha_0 k}  \Bigr )
\leq 
C\exp\Bigl ( -c 3^{\frac12\alpha_0 m}  \Bigr )\,.
\end{align*}
This implies that~$\X - \Y = \O_{\frac12\alpha_0}(C)$ and thus
\begin{equation*}
\X = \Y + (\X - \Y)
\leq \O_{\frac{d\beta }{2d+2}}(C) + \O_{\frac12\alpha_0}(C) \leq \O_{\frac12\alpha_0}(C)\,.
\end{equation*}
This completes the proof.
\end{proof}

\section{Convergence of Dirichlet Problem} \label{s.homogenization}

In this section we prove quantitative estimates on the homogenization error for the Dirichlet problem. The estimate we obtain here is suboptimal, but it will be improved later. We need this preliminary version with small but positive~$\alpha$ in order to obtain the large-scale regularity result of Theorem~\ref{t.large-scale.C01.intro}. The latter result will then be used to improve the~$\alpha$, as we will see in the following sections. 

\smallskip

To this end, we let $U_0 \subseteq \Rd$ denote a bounded Lipschitz domain, and we let 
\begin{equation*}
	U_m := 3^m U_0\,. 
\end{equation*}
By rescaling suitably, we may assume that 
\begin{equation*}
	U_0 \subseteq \cu_0\,,
\end{equation*}
so that~$U_m \subseteq \cu_m$.

\begin{theorem}[Homogenization of the Dirichlet problem]
\label{t.Dirichlet} 
Suppose~$\lambda > \lambda_c$.
There exists a scalar matrix~$\ahom$, constants~$c(d,\lambda),s(d), \alpha(d,\lambda) \in (0,\frac12]$ 
and a nonnegative random variable~$\X$ satisfying 
\begin{equation*}
\P[ \X > t ] \leq 2 \exp \bigl( - c t^{s} \bigr ), \quad  \forall t \in [1,\infty)\,,
\end{equation*}
such that, for every~$m \in \N$ with~$3^m > \X$ and every function~\rv{$\uhom \in C^2(U_m) \cap H^1_0(U_m)$}, the Dirichlet problem \rv{for the function~$u: \eta_*(U_m) \to \R$ solving } 
	\begin{equation*}
	\left\{
	\begin{aligned}
		& \mathcal{L} u = \nabla \cdot \ahom \nabla \uhom & \mbox{on} & \ \rv{\eta_*( U_m) \setminus \{ x \in U_m : \dist(x,\rv{\partial U_m}) <2 \}} \,,
		\\ &
		u = 0 & \mbox{on} & \ \rv{\eta_* (U_m) \cap \{ x \in U_m : \dist(x,\rv{\partial U_m}) <2 \}} \,,
	\end{aligned}
	\right.
\end{equation*}
where the graph Laplacian~\rv{$\mathcal{L}$ is as defined in~\eqref{e.Lapdef} with~$U=U_m$}, has a unique solution~$u$ which satisfies: \rv{there exists a deterministic constant~$C<\infty$} which depends also on~$U_0$, the estimate
\begin{equation*}
\| u - \uhom \|_{{L}^2(\eta_*(U_m))} 
\leq  
C 3^{-m\alpha} 
\| \uhom \|_{{L}^2(\eta_*(U_m))} 
\,.
\end{equation*}
\end{theorem}
Setting this up at the discrete level requires some preliminary notation. For homogenization we set up the discrete problem on a very large, Lipschitz domain $U_m.$ To prescribe a Dirichlet boundary condition we prefer to allow ourselves a mesoscopic parameter, i.e., a second length-scale to describe the size of the boundary: thus, we let $l \leqslant b \leqslant m,$ \rv{where, we recall that the minimal scale~$l$ at which we coarse-grain is as introduced at the end of Section 2.} The optimal choice for $b \in \N$, of course, will depend on the regularity of the given boundary datum. Therefore, we set
\begin{equation*}
\partial_*^{(b)} U_m := \bigcup\{ z + \cu_b: z \in 3^b \Zd \cap \cu_m,\, \mathrm{dist}(z + \cu_b, \partial U_m) = 0\}\,.
\end{equation*}
Essentially, this is a thickened boundary, and the length-scale $b$ is a free parameter that we will optimize in the end.  We fix a smooth datum function $u_b: U_m \to \R$ and seek to study the minimizer for the variational problem 
\begin{equation} \label{e.dirichlet_discrete}
	\min \Bigl\{  \avsum_{x \in \eta_*(U_m)} \Bigl [ \sum_{y \in \eta_*(U_m):y \sim x} \frac12 (u(x) - u(y))^2 + f(x) u(x) \Bigr] : u (x) = u_b(x) \mbox{ if } x \in \partial_*^{(b)} U_m\Bigr\}\,. 
\end{equation}

Our goal is to provide a two-scale convergence proof of quantitative stochastic homogenization for the Dirichlet problem. Toward this goal, for any~\rv{$e \in B_1,$} we recall the maximizer~$\xi_l(\cdot, \cu_m,-e,0)$ of~$J_l(\cu_m,-e,0)$ normalized such that~$\int_{\cu_m} \xi_l(x,\cu_m,-e,0)\,dx = 0$ \rv{(see Lemmas~\ref{l.J.rep.formula} and~\ref{l.variations})}. We set 
\begin{equation*}
	\phi_{m,e} (\cdot) := \xi_l(\cdot,\cu_m,-e,0) - \ell_e(\cdot)\,.
\end{equation*}
Note that $\phi_{m,e}$ depends on $l$ but we will suppress this dependence to aid readability. We let $\uhom $ denote the unique solution to the homogenized equation 
\begin{equation}
	\label{e.uhomdef}
	\begin{aligned} 
		\nabla \cdot \ahom \nabla \uhom &= f\quad \mbox{ in } U_m\\
		\uhom &= u_b \quad \mbox{ on } \partial U_m\,.
	\end{aligned}
\end{equation}
We next present some preliminary lemmas about the estimates for the finite volume correctors $\phi_{m,e}$.

\begin{lemma}
	\label{e.phime-est} Let $\X$ denote the minimal random scale from Theorem \ref{t.homogenization}, so that \eqref{e.X.integrability} holds.
Then, there exists $\alpha \in (0,\sfrac12],$ such that for every $e \in B_1,$ and for every $m \in \N$ with $3^m > \X,$  we have 
\begin{equation*} 
3^{-m}	\|\phi_{m,e}\|_{\underline{L}^2(\rv{\eta_*(\cu_m)})} \leqslant C3^{-m\alpha}\,.
\end{equation*}
\end{lemma}
\begin{proof}
First, by the triangle inequality we have 
\begin{equation} \label{e.triangleinequality-5.2}
	\|\phi_{m,e}\|_{\rv{\underline{L}^2(\eta_*(\cu_m))}} \leqslant \|[\phi_{m,e}]_l\|_{\underline{L}^2(\cu_m)} +  \|\phi_{m,e} - [\phi_{m,e}]_l\|_{\rv{\underline{L}^2(\eta_*(\cu_m))}}\,. 
\end{equation}
For the latter term, by Poincar\'e's inequality (see Lemma~\ref{l.Poincare}),  we find 
\begin{align*}
	  3^{-m}\|\phi_{m,e} - [\phi_{m,e}]_l\|_{\rv{\underline{L}^2(\eta_*(\cu_m))}} &\leqslant K^{\sfrac12} 3^{l-m} \|\nabla \phi_{m,e}\|_{\rv{\underline{L}^2(\eta_*(\cu_m))}} \leqslant C 3^{-(m-l)}\,.
\end{align*}
For the first term, we use the multiscale Poincar\'e inequality to estimate
\begin{align*}
&	\|[\phi_{m,e}]_l\|_{\underline{L}^2(\cu_m)} \leqslant C \sum_{k=l}^m 3^k \Biggl[ \avsum_{z \in 3^k \Zd \cap \cu_m} |(\nabla [\phi_{m,e}]_l)_{z + \cu_k}|^2\Biggr]^{\sfrac12}\\
	&\qquad \qquad= C \sum_{k=l}^m 3^k \Biggl[ \avsum_{z \in 3^k \Zd \cap \cu_m} \!\!\!\! \bigl \vert (\nabla [\xi (\cdot,\cu_m,-e,0)]_l)_{z {+} \cu_k} {-}(\nabla [\xi (\cdot,z+ \cu_k,-e,0)]_l)_{z + \cu_k}  \rv{\bigr\vert}\Biggr]^{\sfrac12}\\
	&\qquad \qquad \leqslant \sum_{k=l}^m 3^{k- \alpha k} \leqslant 3^{m(1-\alpha)}\,,
\end{align*}
\rv{where in the last step we used Theorem~\ref{t.homogenization}. 
Therefore, plugging the last two inequalities in~\eqref{e.triangleinequality-5.2}, and choosing~$l,m$ so that~$m-l > \frac23 m$, we obtain 
\begin{equation*}
3^{-m}	\|\phi_{m,e}\|_{\underline{L}^2(\eta_*(\cu_m))} \leqslant C3^{-m\alpha} + C3^{-(m-l)} \leqslant C3^{-m\alpha}\,.
\end{equation*}}
\end{proof}

We defined the coarse-grained fluxes as follows; these will play a central role in quantifying the convergence of the Dirichlet problem. For any $m  \in \N$ such that $3^m > \X$, where $\X$ is as in the statement of Theorem \ref{t.homogenization}, we define

\begin{equation}
	\label{e.fluxdef}
	[\g_{ e_i,j}(\cu_m)]_{l} (z) := \frac{1}{|\cu_l|}\sum_{x \in (z + \cu_l) \cap \eta_*(\cu_m)} \sum_{y \in \eta_*(\cu_m), \,y \sim x}  (x_j -y_j)(x_i - y_i + \phi_{m,e_i}(x) - \phi_{m,e_i}(y))\,,
\end{equation}
for any $z \in 3^l \Zd \cap \cu_m.$

\begin{lemma} \label{l.fluxrates}
Let~$\X$ be as in the statement of Theorem~\ref{t.homogenization}. There exists~$\alpha \in (0,\sfrac12]$ such that for every~$e \in B_1,$ and for every $m \in \N$ such that $3^m > \X,$ we have 
	\begin{equation*}
	3^{-m}	\|[\g_{e_i,j}(\cu_m)]_{l} - \a_{l,ij}(\cu_m)\|_{\underline{H}^{-1}(\cu_m)} \leqslant 3^{-m\alpha}\,.
	\end{equation*}
\end{lemma}
\begin{proof}
	\emph{Step 1}. Recalling the definition of $\phi_{m,e},$ we use the first variation formula \eqref{e.firstvarmaster} with the choices $q = 0, p = -e_j,$ and $w(\cdot) = \xi_l(\cdot,\cu_m,-e_i,0);$ this yields that 
	\begin{align*}
		&\lefteqn{\frac1{|\cu_m|}\sum_{x,y \in \eta_*(\cu_m),\,y \sim x}  (x_j-y_j)(x_i - y_i + \phi_{m,e_i}(x) - \phi_{m,e_i}(y)) }\\
		&\quad = \frac1{|\cu_m|}\sum_{x,y \in \eta_*(\cu),\,y \sim x} (\xi_l(x,\cu_m,-e_i,0) - \xi_l(y,\cu_m, -e_i,0))(\xi_l(x,\cu_m,-e_j,0) - \xi_l(y,\cu_m, -e_j,0)) \\
		&\quad = e_j \cdot \a_l(\cu_m)e_i\,, 
	\end{align*}
where the last equality is obtained by polarizing \eqref{e.secondvarmaster}. 

Using the definition of the coarse-grained fluxes, the last identity rewrites as 
\begin{equation} \label{e.balance_of_fluxes}
 \avsum_{\zeta \in 3^l\Zd \cap \cu_m}  [\g_{e_i,j}(\cu_m)]_{l}(\zeta) = e_j \cdot \a(\cu_m)e_i\,.
\end{equation}

\smallskip
\emph{Step 2.} Therefore by the multiscale Poincar\'e inequality, and then triangle inequality for the rest of the page, we find that for any $j \in \{1,\ldots, d\},$ we have 
\begin{align*}
&{\lefteqn\| [\g_{e_i,j}(\cu_m)]_{l} - e_j \cdot \a(\cu_m) e_i\|_{\underline{H}^{-1}(\cu_m)} \leqslant C \|[\g_{m,e_i,j}]_{l} - e_j\cdot \a_l(\cu_m) e_i\|_{\underline{L}^2(\cu_m)} }
\\ &
\qquad \qquad \qquad 
+ C \sum_{n=l}^m 3^n \Biggl( \avsum_{z \in 3^n \Zd \cap \cu_m } \bigl| ([\g_{e_i,j}(\cu_m)]_{l})_{z+\cu_n} - e_j \cdot \a_l(\cu_m) e_i\bigr|^2 \Biggr)^{\!\!\sfrac12}
\\& \qquad \qquad \quad 
\leqslant 
 C \|[\g_{m,e_i,j}]_{l} - e_j\cdot \a_l(\cu_m) e_i\|_{\underline{L}^2(\cu_m)} \\
 &\qquad \qquad \qquad + C\sum_{n=l}^m 3^n \Biggl( \avsum_{z \in 3^n \Zd \cap \cu_m} \bigl| ([\g_{e_i,j}(\cu_m)]_l)_{z+\cu_n} - ([\g_{e_i,j}(z+\cu_n)]_l)_{z+\cu_n} \bigr|^2\Biggr)^{\!\!\sfrac12}\\
 &\qquad \qquad \qquad + C\sum_{n=l}^m 3^n \Biggl( \avsum_{z \in 3^n \Zd \cap \cu_m} \bigl|  ([\g_{e_i,j}(z+\cu_n)]_l)_{z+\cu_n} - e_j \cdot \a_l(\cu_m)e_i\bigr|^2\Biggr)^{\!\!\sfrac12}\,.
\end{align*}
Now, repeating the argument for \eqref{e.balance_of_fluxes}, we find 
\begin{equation*}
([\g_{e_i,j}(z+\cu_n)]_l)_{z+\cu_n} = \avsum_{\zeta \in 3^l\Zd \cap (z+\cu_n)} [\g_{e_i,j}(z+\cu_n)]_l(\zeta) = e_j \cdot \a_l(z+\cu_n)e_i\,.
\end{equation*}
Plugging this into the previous display for the third line, and bounding the middle line using quadratic response, we get 
\begin{align*}
&{\lefteqn\| [\g_{e_i,j}(\cu_m)]_{l} - e_j \cdot \a(\cu_m) e_i\|_{\underline{H}^{-1}(\cu_m)}}\\
&\quad \leqslant C + C \sum_{n=l}^m 3^n \Biggl(  \avsum_{z \in 3^n\Zd \cap \cu_m} |\a_l(z+\cu_n) - \a_l(\cu_m) \bigr|^2\Biggr)^{\!\!\sfrac12}\,. 
\end{align*}
The argument is then completed similarly to the previous lemma. 
\end{proof}
We begin with the following two-scale expansion computation, for which we must define a discrete notion of the~$H^{-1}$ norm by duality. 

Given~$f: \eta_*(\cu_m) \to \RR$ with $f (x) \equiv 0$ when~$x \in \partial_*^{(b)} \cu_m,$ we define 
\begin{align*}
	\|f\|_{\underline{H}^1(\eta_*(\cu_m))} := \Biggl(\frac1{|\cu_m|}\sum_{x,y\in \eta_*(\cu_m),\,y\sim x} (f(x) - f(y))^2 \Biggr)^{\!\!\sfrac12}\,.
\end{align*}
This norm is associated with an inner product defined in the natural way, namely, given two functions $f,g: \eta_*(\cu_m) \to \RR$ with $f (x) = g(x) = 0$ for $x \in \partial_*^{(b)}\cu_m,$ we find 
\begin{equation*}
	\langle f, g \rangle_{\underline{H}^1(\eta_*(\cu_m))} :=  \frac1{|\cu_m|}\sum_{x,y \in \eta_*(\cu_m),\,y\sim x} (f(x) - f(y))(g(x) - g(y))  \,.
\end{equation*}
The dual norm, i.e., the~$H^{-1}$ norm on the random graph~$\eta_*(\cu_m)$ of a function~$f: \eta_*(\cu_m) \to \R$ is then defined via 
\begin{equation*}
	\|f\|_{\underline{H}^{-1}(\eta_*(\cu_m))} := \sup\bigl\{ \langle f, v \rangle_{\underline{H}^1(\eta_*(\cu_m))} | v : \eta_*(\cu_m) \to \RR, v|_{\partial_*^{(b)} \cu_m} \equiv 0, \|v\|_{H^1(\eta_*(\cu_m))} \leqslant 1   \bigr\} \,.
\end{equation*}

\begin{lemma}
	\label{e.two-scale-exp}
Let~$\uhom \in C^2(\Rd),$ and let~$m \in \N$ be such that~$3^m > \X,$ where~$\X$ is the minimal random scale from Theorem \ref{t.homogenization}. Then, for any~$x \in \rv{\eta_*(\cu_m ) \setminus \partial_*^{(b)} \cu_m},$ we set
\begin{equation}
	\label{e.2scexp}
u(x) := \uhom(x) + \sum_{i=1}^d \partial_{x_i} \uhom (x)\phi_{m,e_i}(x)\,. 
\end{equation}
Then, there exists a constant~$C(d,\lambda)<\infty$ such that  
\begin{equation} \label{e.2scexp_plugin}
	\mathcal{L}u(x) = 	\sum_{i,j=1}^d \frac12 \partial_{x_i}\partial_{x_j} \uhom(x)
	\!\!	\sum_{y \in \eta_*(\cu_m) ,\,y\sim x}  
		( x_j {-} y_j )
		\bigl (  x_i {-} y_i {+} \phi_{m,e_i}(x) {-} \phi_{m,e_i}(y)\bigr )
		+ \rv{\mathcal{R}(x)}\,,
\end{equation}
with $\rv{\mathcal{R}(x)}$ satisfying 
\begin{align} \label{e.rem_bound}
\lefteqn{ 
3^{-m}\|\rv{\mathcal{R}(x)}\|_{\underline{H}^{-1}(\eta_*(\cu_m))} 
} \notag \qquad & \\ 
&
\leqslant C \! \sup_{|e|= 1} 3^{-m}\|\phi_{m,e}\|_{\underline{L}^2(\eta_*(\cu_m))}\Biggl(\sum_{x \in \eta_*(\cu_m)} \sup_{y \in B_1(x)} \!\!\Bigl( |D\uhom(y)|^2 {+} |D^2 \uhom(y)|^2\Bigr)\Biggr)^{\!\!\sfrac12} \,;
\end{align}
\rv{here, we extended~$\mathcal{R}(x) \equiv 0$ to the boundary layer~$\partial_*^{(b)} (\cu_m)$.}
\end{lemma}
\begin{proof}
Inserting the \rv{definition} \eqref{e.2scexp} into the graph Laplacian, 
and using the discrete product rule in the form 
\begin{equation*}
(fg)(x) - (fg)(y) = \frac12  ( f(x) + f(y) )  ( g(x) - g(y) ) + \frac12 ( g(x) + g(y) ) ( f(x) - f(y) )
\,,
\end{equation*}
we find at any~$\rv{x \in \eta_*(\cu_m) \setminus \partial_*^{(b)} \cu_m},$

\begin{align}
\label{e.Lu.1}
\mathcal{L} u (x) 
&
= \sum_{y \in \eta_*(\cu_m),y\sim x} \Biggl[ \uhom(x) - \uhom(y) + \sum_{i=1}^d (\partial_{x_i} \uhom \phi_{m,e_i})(x) - (\partial_{x_i} \uhom \phi_{m,e_i})(y) \Biggr] 
\notag \\&
=
\sum_{y \in \eta_*(\cu_m) ,y\sim x} \Biggl[ \uhom(x) - \uhom(y) + \sum_{i=1}^d \frac{1}{2}\bigl (\partial_{x_i}\uhom(y)+\partial_{x_i}\uhom(x)\bigr ) (\phi_{m,e_i}(x) - \phi_{m,e_i}(y))\Biggr]
\notag \\ & \qquad 
+ \sum_{y \in \eta_*(\cu_m),y\sim x} \sum_{i=1}^d \frac{1}{2}\bigl( \partial_{x_i} \uhom(x) - \partial_{x_i} \uhom (y)\bigr)
\bigl (\phi_{m,e_i}(x) +\phi_{m,e_i}(y) \bigr )
\,.
\end{align}
\rv{The second} line of~\eqref{e.Lu.1} can be rewritten as 
\begin{align}
\label{e.Lu.2}
\lefteqn{
\sum_{y \in \eta_*(\cu_m) ,y\sim x} \Biggl[ \uhom(x) - \uhom(y) + \sum_{i=1}^d \frac12\bigl (\partial_{x_i}\uhom(y)+\partial_{x_i}\uhom(x)\bigr )  (\phi_{m,e_i}(x) - \phi_{m,e_i}(y))\Biggr]
}
\qquad & 
\notag \\ &
=
\sum_{y \in \eta_*(\cu_m) ,y\sim x} 
\sum_{i=1}^d   \partial_{x_i}\uhom(x)( x_i - y_i + \phi_{m,e_i}(x) - \phi_{m,e_i}(y))
\notag \\ & \quad
+
\sum_{y \in \eta_*(\cu_m) ,y\sim x} \sum_{i,j=1}^d 
\frac12 \partial_{x_i}\partial_{x_j} \uhom(x)  ( x_j - y_j )
\bigl (  x_i - y_i + \phi_{m,e_i}(x) - \phi_{m,e_i}(y)\bigr )
\notag \\ & \quad
+
\sum_{y \in \eta_*(\cu_m) ,y\sim x} 
\biggl ( \uhom(x) - \uhom(y) - \sum_{i=1}^d 
\partial_{x_i} \uhom(x) 
( x_i - y_i ) 
\notag \\ &\qquad \qquad \qquad \qquad \qquad -\sum_{i,j=1}^d \frac12 \partial_{x_i}\partial_{x_j} \uhom(x) ( x_i - y_i ) ( x_j- y_j ) 
\biggr )
\notag \\ & \quad
+
\sum_{y \in \eta_*(\cu_m) ,y\sim x} \sum_{i=1}^d 
\frac12\Bigl  (\partial_{x_i}\uhom(y) -\partial_{x_i}\uhom(x)
\notag \\ &\qquad \qquad \qquad \qquad 
-\sum_{j=1}^d \partial_{x_i}\partial_{x_j} \uhom(x)  ( x_j - y_j )\Bigr  )
\bigl (  \phi_{m,e_i}(x) - \phi_{m,e_i}(y)\bigr )
\,.
\end{align}
The first term on the right side of~\eqref{e.Lu.2} vanishes by the corrector equation satisfied by~$\phi_{m,e_i}$.
With this in mind, we may insert~\eqref{e.Lu.2} into~\eqref{e.Lu.1} and rearrange some terms to obtain 
\begin{align}
\label{e.Lu.3}
\mathcal{L} u (x) 
&
=
\sum_{i,j=1}^d \frac12 \partial_{x_i}\partial_{x_j} \uhom(x)
\sum_{y \in \eta_*(\cu_m) ,y\sim x} 
( x_j - y_j )
\bigl (  x_i - y_i + \phi_{m,e_i}(x) - \phi_{m,e_i}(y)\bigr )
\notag \\ & \quad
+ \sum_{y \in \eta_*(\cu_m),y\sim x} \sum_{i=1}^d \frac{1}{2}\bigl( \partial_{x_i} \uhom(x) - \partial_{x_i} \uhom (y)\bigr)
\bigl (\phi_{m,e_i}(x) +\phi_{m,e_i}(y) \bigr )
\notag \\ & \quad
+
\sum_{y \in \eta_*(\cu_m) ,y\sim x}
\biggl ( \uhom(x) - \uhom(y) - \sum_{i=1}^d 
\partial_{x_i} \uhom(x) 
( x_i - y_i ) 
\notag \\ &\qquad \qquad \qquad -\sum_{i,j=1}^d \frac12 \partial_{x_i}\partial_{x_j} \uhom(x) ( x_i - y_i ) ( x_j- y_j ) 
\biggr )
\notag \\ & \quad
+
\sum_{y \in \eta_*(\cu_m) ,y\sim x} \sum_{i=1}^d 
\frac12\Bigl  (\partial_{x_i}\uhom(y) -\partial_{x_i}\uhom(x)
-\sum_{j=1}^d \partial_{x_i}\partial_{x_j} \uhom(x)  ( x_j - y_j )\Bigr  )\notag \\
&
\qquad \qquad \qquad \qquad \times \bigl (  \phi_{m,e_i}(x) - \phi_{m,e_i}(y)\bigr )
\,.
\end{align}
Of the four lines in the right side of~\eqref{e.Lu.3}, the first nearly vanishes, as it will be shown to be close to the macroscopic equation; the second is a ``divergence form'' homogenization error term, which is shown to be small using the estimates we have on the correctors; the third and fourth terms are due to the discreteness of the underlying random graph and will be shown to be small by Taylor expansion, if~$\uhom$ is sufficiently smooth.

\smallskip 
The last three terms are collected in the remainder $\mathcal{R},$ which we next proceed to estimate term-by-term. For the divergence form error, observe that for any test function $w \in \underline{H}^1(\eta_*(\cu_m))$ with $w \equiv 0$ in $\partial_*^{(b)}\cu_m,$ we find 
\begin{align*}
	&\frac1{|\cu_m|}\sum_{x,y \in \eta_*(\cu_m),\,y \sim x} \sum_{i=1}^d \frac12 (\partial_{x_i} \uhom(x) - \partial_{x_i} \uhom(y)) (\phi_{m,e_i}(x) + \phi_{m,e_i}(y)) w(x)\\
	&\quad = \frac1{2|\cu_m|} \sum_{x,y \in \eta_*(\cu_m),y\sim x} \sum_{i=1}^d \frac12 (\partial_{x_i} \uhom(x) - \partial_{x_i} \uhom(y)) (\phi_{m,e_i}(x) + \phi_{m,e_i}(y)) \bigl(w(x) - w(y)\bigr),
\end{align*}
so that, by Cauchy-Schwarz, twice,
\begin{equation}  \label{e.cookiedough}
	\begin{aligned}
	&\Biggl| \frac1{2|\cu_m|} \sum_{x,y \in \eta_*(\cu_m),y\sim x} \sum_{i=1}^d \frac12 (\partial_{x_i} \uhom(x) - \partial_{x_i} \uhom(y)) (\phi_{m,e_i}(x) + \phi_{m,e_i}(y)) \bigl(w(x) - w(y)\bigr) \Biggr|\\
&\leqslant C\sum_{i=1}^d\avsum_{x \in \eta_*(\cu_m)} \Biggl( \phi_{m,e_i}^2(x)\sum_{y \in \eta_*(\cu_m),\, y\sim x}   (\partial_{x_i}\uhom(x) - \partial_{x_i}\uhom(y))^2\Biggr)^{\!\!\sfrac12}\\
&\quad \quad \quad \quad \quad \quad \times \Biggl( \sum_{y \in \eta_*(\cu_m),\, y \sim x} (w(x) - w(y))^2 \Biggr)^{\!\!\sfrac12}\\
&\leqslant C \avsum_{x \in \eta_*(\cu_m)} |\phi_{m,e_i}(x)| \sup_{y \in B_1 (x)} |D^2\uhom(y)| \Biggl( \sum_{y \in \eta_*(\cu_m),\,y\sim x}  (w(x) - w(y))^2 \Biggr)^{\!\!\sfrac12}\\
&\leqslant C  \Biggl( \avsum_{x \in \eta_*(\cu_m)}\sum_{y \in \eta_*(\cu_m),\, y\sim x}  (w(x) - w(y))^2 \Biggr)^{\!\!\sfrac12} \avsum_{x \in \eta_*(\cu_m)} |\phi_{m,e_i}(x)| \sup_{y \in B_1 (x)} |D^2 \uhom(y)|\\
&\leqslant C \|w\|_{\underline{H}^1(\eta_*(\cu_m))}  \Biggl( \avsum_{x \in \eta_*(\cu_m)} \phi_{m,e_i}(x)^2 \Biggr)^{\!\!\sfrac12} \Biggl( \avsum_{x \in \eta_*(\cu_m)} \sup_{y \in B_1(x)} |D^2\uhom(y)|^2\Biggr)^{\!\!\sfrac12}\\
&\leqslant C \|w\|_{\underline{H}^1(\eta_*(\cu_m))}  \sup_{|e| = 1} \|\phi_{m,e}\|_{\underline{L}^2(\cu_m)} \Biggl( \avsum_{x \in \eta_*(\cu_m)} \sup_{y \in B_1(x)} |D^2\uhom(y)|^2\Biggr)^{\!\!\sfrac12} \,.
	\end{aligned}
\end{equation}

Taking the supremum over $w \in \underline{H}^1(\eta_*(\cu_m)),$ with $w \equiv 0$ on the thickened boundary $\partial_*^{(b)}(\cu_m),$ along with Lemma~\ref{e.phime-est} yield that
\begin{equation}
	\label{e.term1}
	\begin{aligned}
	&	\Biggl\| \sum_{y \in \eta_*(\cu_m),\,y\sim x}  \sum_{i=1}^d \frac12 (\partial_{x_i} \uhom(\cdot) {-} \partial_{x_i} \uhom(y)) (\phi_{m,e_i}(\cdot) {+} \phi_{m,e_i}(y))\Biggr\|_{\underline{H}^{-1} (\eta_*(\cu_m))} \!\!\!\!\! \!\!\!
	\\
	&\quad\quad \quad 	\leqslant 
		C  3^{m(1-\alpha)}  \Biggl( \avsum_{x \in \eta_*(\cu_m)} \sup_{y \in B_1(x)} |D^2\uhom(y)|^2\Biggr)^{\!\!\sfrac12}\,. 
	\end{aligned}
\end{equation}
We must estimate the remaining Taylor expansion terms in $H^{-1}.$ We consider the first of these two terms: by the prior symmetry arguments,
\begin{equation*}
	\begin{aligned}
	&	2\frac1{|\cu_m|}\sum_{x \in \eta_*(\cu_m)} \! w(x) \biggl[ \sum_{y \in \eta_*(\cu_m),y\sim x}  \uhom(x) {-} \uhom(y) {-} \sum_{i=1}^d 
		\partial_{x_i} \uhom(x) 
		( x_i {-} y_i ) 
	\\
		&\quad \quad \quad - \sum_{i,j=1}^d \frac12 \partial_{x_i}\partial_{x_j} \uhom(x) ( x_i {-} y_i ) ( x_j{-} y_j ) \biggr]
		\\
	&\quad =\frac1{|\cu_m|}\sum_{x,y \in \eta_*(\cu_m),\,y\sim x} (w(x) - w(y)) (\uhom(x) - \uhom(y))\\
	&\quad \quad - \frac1{|\cu_m|}\sum_{i=1}^d \sum_{x,y \in \eta_*(\cu_m),\,y\sim x} (w(x) - w(y))(\partial_{x_i} \uhom(x) - \partial_{x_i}\uhom(y))(x_i - y_i) \\
	&\quad \quad - \frac1{2|\cu_m|} \sum_{i,j=1}^d \sum_{x,y \in \eta_*(\cu_m),\,y\sim x} (w(x) - w(y)) (x_i - y_i)(x_j - y_j) \\
	&\qquad \qquad \qquad \qquad\times \bigl(  \partial_{x_i}\partial_{x_j} \uhom(x) - \partial_{x_i}\partial_{x_j} \uhom(y)\bigr)\,.
	\end{aligned}
\end{equation*}
The right hand side of the preceding display is a sum of three terms, the middle of which vanishes by skew-symmetry of the summand. So we must estimate the first and the third terms. Therefore, we must estimate the first and the third terms. The first one is easy: by Cauchy-Schwarz, we get
\begin{equation} \label{e.term2a}
	\Biggl| \frac1{|\cu_m|}\sum_{x,y \in \eta_*(\cu_m),\,y\sim x} (w(x) - w(y)) (\uhom(x) - \uhom(y))\Biggr| \leqslant \|\nabla w\|_{\underline{L}^2(\eta_*(\cu_m))} \|\nabla \uhom\|_{\underline{L}^2(\eta_*(\cu_m))}\,.
\end{equation}
For the third one, proceeding as in the chain of inequalities in \eqref{e.cookiedough}, we obtain 
\begin{equation} \label{e.term2}
	\begin{aligned}
		&\Biggl| \frac1{2|\cu_m|} \sum_{i,j=1}^d \sum_{x,y \in \eta_*(\cu_m),\,y\sim x} (w(x) - w(y)) (x_i - y_i)(x_j - y_j) \bigl(  \partial_{x_i}\partial_{x_j} \uhom(x) - \partial_{x_i}\partial_{x_j} \uhom(y)\bigr)\Biggr| \\
		&\quad \leqslant \frac{1}{2|\cu_m|} \sum_{i,j=1}^d \sum_{x \in \eta_*(\cu_m)} \bigl|\sup_{y \in B_1(x)} \partial_{x_i}\partial_{x_j} \uhom(y)\bigr| \sum_{y \in \eta_*(\cu_m),\,y\sim x}  |w(x) - w(y)|\\
		&\quad \leqslant \frac{1}{2} \Biggl(\frac1{|\cu_m|} \sum_{x \in \eta_*(\cu_m)}  \sup_{y \in B_1 (x)} |D^2 \uhom(y)|^2\Biggr)^{\!\!\sfrac12}  \|\nabla w\|_{\underline{L}^2(\eta_*(\cu_m))}\,.
	\end{aligned}
\end{equation}
It remains to estimate the last term in \eqref{e.Lu.3}, once again, in $\underline{H}^{-1}.$ Recalling that $w$ vanishes on the thickened boundary, and symmetrizing in $x$ and $y$ as before, we find  
\begin{align*}
&	\Biggl| 
	\frac1{|\cu_m|}\sum_{x \in \eta_*(\cu_m)} \! \! w(x) \!\! \sum_{y \in \eta_*(\cu_m),y\sim x } 
	 \sum_{i=1}^d 
	\bigl  (\partial_{x_i}\uhom(y) {-}\partial_{x_i}\uhom(x)\\
	&\quad \quad \quad 
	{-}\sum_{j=1}^d \partial_{x_i}\partial_{x_j} \uhom(x)  ( x_j {-} y_j )\bigr  ) \bigl (  \phi_{m,e_i}(x) {-} \phi_{m,e_i}(y)\bigr )
	\Biggr|\\
& = \frac1{2|\cu_m|} \Biggl|\sum_{x,y \in \eta_*(\cu_m),\,y\sim x} (w(x) - w(y)) \sum_{i=1}^d\bigl( \partial_{x_i}\uhom(y)- \partial_{x_i} \uhom(x) \bigr)\bigl(\phi_{m,e_i}(x) - \phi_{m,e_i}(y)\bigr)	\Biggr|\\
& + \frac1{2|\cu_m|}\Biggl| \sum_{x,y \in \eta_*(\cu_m),\,y \sim x}(w(x) - w(y)) \times \\
&\quad \quad \quad  \sum_{i,j=1}^d (x_j-y_j) \Bigl(\partial_{x_i}\partial_{x_j} \uhom(x) - \partial_{x_i} \partial_{x_j} \uhom(y) \Bigr)(\phi_{m,e_i}(x) - \phi_{m,e_i}(y))\Biggr|\,,
	\end{align*}
is a sum of two terms. Once again, the first of these is zero by skew symmetry of the summand with respect to $x,y,$ while for the second, we estimate as in the string of inequalities \eqref{e.cookiedough}, to obtain that this term is no more than 
\begin{equation} \label{e.term3}
	\leqslant C \Biggl( \avsum_{x \in \eta_*(\cu_m)} \sup_{y \in B_1(x)} |D^2\uhom(y)|^2\Biggr)^{\!\!\sfrac12} \|w\|_{\underline{H}^1(\eta_*(\cu_m))} \sup_{|e| = 1} \|\phi_{m,e}\|_{\underline{L}^2(\eta_*(\cu_m))}\,.
 \end{equation}  
Combining \eqref{e.term1}, \eqref{e.term2a}, \eqref{e.term2} and \eqref{e.term3} and plugging these in \eqref{e.Lu.3} completes the proof. 
\end{proof}

Our next lemma provides control on the first term on the right-hand side of \eqref{e.2scexp_plugin}, when $\uhom$ is a solution of the homogenized equation. Combining the two preceding lemmas will essentially complete the proof of our main result of this section: convergence of the solutions to the Dirichlet problem, with rates\rv{. The} only remaining ingredient will then be the control of the boundary layer terms. 

\begin{lemma} \label{l.firstterm_of_2scexp}
Let $m \in \N$ be such that $3^m > \X$, where $\X$ is the minimal random scale from Theorem \ref{t.homogenization}.	 Then, there exists a constant $C(d,\la)<\infty$ such that for each~$i,j \in \{ 1,\ldots, d\},$ we have
	\begin{equation}
		\label{e.H-1bound_uhom}
	3^{-m}	\Biggl\| 
		\sum_{y \in \eta_*(\cu_m) ,\, y \sim \cdot}  
		( (\cdot)_j {-} y_j )
		\bigl (  (\cdot)_i {-} y_i {+} \phi_{m,e_i}(\cdot) {-} \phi_{m,e_i}(y)\bigr ) - \ahom_{ij}
		\Biggr\|_{\underline{H}^{-1}(\eta_*(\cu_m))} \leqslant C3^{-m\alpha}\,,
	\end{equation}
where $\alpha \in (0,\frac12]$ is as in Lemma~\ref{e.phime-est}. 
\end{lemma}
\begin{proof}
The indices $i,j$ as in the statement of the lemma are fixed throughout. In addition, we fix $w \in \underline{H}^1(\eta_*(\cu_m))$ with $w|_{\partial_*^{(b)} \cu_m} \equiv 0,$ and write
	\begin{align*}
		A &:= 
		\frac1{|\cu_m|}\sum_{x \in \eta_*(\cu_m)} w(x) \Biggl[ \sum_{y \in \eta_*(\cu_m) ,y \sim x} 
		( x_j {-} y_j )
		\bigl (  x_i {-} y_i {+} \phi_{m,e_i}(x) {-} \phi_{m,e_i}(y)\bigr ) - \ahom_{ij}\Biggr]
		\\
		&=: B + C\,,
	\end{align*}
with 
\begin{align*}
	B := \frac1{|\cu_m|}\sum_{x \in \eta_*(\cu_m) }
	(w(x) {-} [w]_l(x))
\Biggl[	\sum_{y \in \eta_*(\cu_m),y \sim x }  
	( x_j {-} y_j )
	\bigl (  x_i {-} y_i {+} \phi_{m,e_i}(x) {-} \phi_{m,e_i}(y)\bigr ) -\ahom_{ij} \Biggr]\,,
\end{align*}
and 
\begin{align*}
	C := 	\frac1{|\cu_m|}\sum_{x \in \eta_*(\cu_m) }
	[w]_l(x)
	\Biggl[	\sum_{y \in \eta_*(\cu_m), \,y \sim x }  
	( x_j {-} y_j )
	\bigl (  x_i {-} y_i {+} \phi_{m,e_i}(x) {-} \phi_{m,e_i}(y)\bigr ) - \ahom_{ij} \Biggr]\,. 
\end{align*}
Toward the goal of estimating $C,$ first, we rewrite it as follows: 
\begin{align*}
&C_{ij} =	 \frac1{|\cu_m|}\sum_{x \in \eta_*(\cu_m) }
		[w]_l(x)
	\Biggl[	\sum_{y \in \eta_*(\cu_m) ,\, y \sim x}  
		( x_j {-} y_j )
		\bigl (  x_i {-} y_i {+} \phi_{m,e_i}(x) {-} \phi_{m,e_i}(y)\bigr )-\ahom_{ij}\Biggr]
	 \qquad &
	\\ & 
	=
	\avsum_{z\in 3^l\Zd \cap \cu_m}
	[w]_l(z) 
	\notag \\ & \qquad \quad \times 
	\avsum_{x \in (z+ \cu_l )\cap \eta_*(\cu_m) }
	\!\Biggl[
	\sum_{y \in \eta_*(\cu_m),\,y \sim x }  
	( x_j {-} y_j )
	\bigl (  x_i {-} y_i {+} \phi_{m,e_i}(x) {-} \phi_{m,e_i}(y)\bigr ) {-} \ahom_{ij} \Biggr]
	\\ & 
	=
	\avsum_{z\in 3^l\Zd \cap \cu_m}
	[w]_l(z)
	\bigl( [\g_{e_i,j}(\cu_m) ]_l(z) - \ahom_{ij} \bigr) 
\,,
\end{align*}
By Lemma~\ref{l.fluxrates} we find 
\begin{equation*}
	\begin{aligned}
		\Biggl| 	\avsum_{z\in 3^l\Zd \cap \cu_m}
		[w]_l(z)
		\bigl( [\g_{e_i,j}(\cu_m) ]_l(z) - \ahom_{ij} \bigr) \Biggr| 
		&\leqslant 
		\|[\g_{e_i,j}]_l - \ahom_{l,ij}\|_{\underline{H}^{-1}(\cu_m)} \|[w]_l  \|_{\underline{H}^1(\cu_m)}\\
		&\leqslant 
			\|[\g_{e_i,j}]_l - \ahom_{l,ij}\|_{\underline{H}^{-1}(\cu_m)} \|w \|_{\underline{H}^1(\eta_*(\cu_m))}\,.
	\end{aligned}
\end{equation*}

\smallskip

\noindent Turning next to the term $B$, by Cauchy-Schwarz, and the energy bound, we have 
\begin{align*}
|B| 
&= \Biggl| \avsum_{z \in 3^l \Zd \cap \cu_m} \avsum_{x \in z + \cu_l \cap \eta_*(\cu_m) }
(w(x) {-} [w]_l(x))\\
&\quad \quad \quad \times
\Biggl[	\sum_{y \in \eta_*(\cu_m), y \sim x }  
( x_j {-} y_j )
\bigl (  x_i {-} y_i {+} \phi_{m,e_i}(x) {-} \phi_{m,e_i}(y)\bigr ) -\ahom_{ij} \Biggr]\Biggr| \\
&\leqslant C\avsum_{z \in 3^l \Zd \cap \cu_m} \Biggl( \avsum_{x \in \eta_*(\cu_m) \cap z + \cu_l} |w(x) - [w_l](z)|^2\Biggr)^{\!\!\sfrac12}\\
&\leqslant C3^l \|w\|_{\underline{H}^1(\eta_*(\cu_m))}\,.
\end{align*}
Therefore, by the triangle inequality, and by supremizing among ~$w$ with ~$\|w\|_{\underline{H}^1(\eta_*(\cu_m))} \leqslant 1,$ we complete the required estimate for $A.$ 
\end{proof}

The main ingredient in the proof of Theorem \ref{t.Dirichlet} is the following quantitative two-scale expansion. We recall that by scaling, if necessary, we are concerned with the Dirichlet problem in a large domain $U_m := 3^m U_0,$ where $U_0$ is a bounded Lipschitz domain, with $U_0 \subseteq \cu_0.$ In its statement we prefer to give ourselves the parameter $b > l$ which dictates the length-scale of the boundary layer for $U_m$, depending on the regularity of the boundary data (see \eqref{e.dirichlet_discrete}). 

Accordingly, we  let~$\xi \in C^\infty(\cu_m)$ be such that ~$\xi (x) \equiv 1$ if~$x \in \partial_*^{(b)} U_m$, and~$\xi(x) \equiv 0$ if~$x \in \cu_m \setminus \partial_*^{(2b)}(U_m),$ and is such that~$|\nabla \xi| \leqslant C3^{-b}.$

\begin{proposition}
	\label{t.DirichletProblem}
	Let~$m \in \N$ be such that~$3^m > \X,$ with~$\X$ as in Theorem \ref{t.homogenization}. 
	Let~$v$ denote the solution to the discrete variational problem, \eqref{e.dirichlet_discrete}, let $\uhom$ the solution to the homogenized problem~\eqref{e.uhomdef}, and define~$u$ to be the two-scale function with respect to~$\uhom$, as defined in~\eqref{e.2scexp}. Then, there exists a constant~$C(d,\lambda,U_0)<\infty$ such that we have the two-scale estimate  
		\begin{align}\label{e.2scexp_statement}
	\|v {-} u {-} \xi(u_b {-} u) \|_{{H}^1(\eta_*(U_m))} & \!\leqslant  C3^{m(1-\alpha)-b}  \Biggl( \sum_{x \in \eta_*(U_m\setminus \partial_*^{(2b)} U_m)} \sup_{y \in B_1(x)} |D\uhom(y)|^2 + |D^2\uhom(y)|^2\Biggr)^{\!\!\sfrac12}
		 \notag \\ & \qquad + C3^{b-\frac{m}{2}}  \|\nabla u_b\|_{{L}^2(\eta_*(\partial_*^{(2b)} U_m))} \,,
		\end{align}
	where $\alpha \in (0,\frac12]$ is as in Lemma~\ref{e.phime-est}.
\end{proposition}
\begin{proof}
\emph{Step 1.} 	We let~$v$ denote the unique solution to the variational problem \eqref{e.dirichlet_discrete}.	
	Let~$u$ be the function in~\eqref{e.2scexp}, and set 
	\begin{equation*}
			w(x) := \xi(x) u_b(x) + (1-\xi(x)) u(x)\,. 
		\end{equation*}
It follows that~$v(x) \equiv w(x)$ when~$x \in \partial_*^{(b)} U_m,$ the thickened Dirichlet boundary of $U_m$. We test the identity \eqref{e.2scexp_plugin} with $w - v$, i.e., we multiply both sides of \eqref{e.2scexp_plugin} by $(w(x) - v(x))$ and sum over~$x \in \eta_*(U_m).$ Recall we have~$\mathcal{L}v = f,$ so that 
\begin{equation} \label{e.testing}
	\sum_{x \in \eta_*(U_m)} (w(x) - v(x))\mathcal{L}(w - v)(x) = \sum_{x \in \eta_*(U_m)} (w(x) - v(x))\mathcal{L} w(x) - \sum_{x\in \eta_*(U_m)} (w(x) - v(x)) f(x)\,.
\end{equation}
We will first focus on the first term. Notice that  
\begin{align*}
	\mathcal{L}w(x) &= \sum_{y \in \eta_*(U_m),\, y \sim x} \bigl(w(y) - w(x) \bigr)\\
	&= \sum_{y \in \eta_*(U_m),\,y\sim x}  \bigl( (\xi u_b )(y) - (\xi u_b)(x) + ((1-\xi)u)(y) - ((1-\xi)u)(x)\bigr)\\
	&= \frac12\sum_{y \in \eta_*(U_m),\,y\sim x}  \biggl[ (\xi(x) + \xi(y))(u_b(x) - u_b(y)) + (u_b(x) + u_b(y))(\xi(x) - \xi(y))  \\
	&\quad \quad \quad \quad + (2 - (\xi(x) + \xi(y)))(u(x)- u(y)) + (\xi(y) - \xi(x))(u(x) + u(y))\biggr]\\
	&=: \sum_{y \in \eta_*(U_m),\,y\sim x}  \biggl(1- \frac{\xi(x)+\xi(y)}{2} \biggr)(u(x) - u(y)) + \mathcal{B}(x) \,,
\end{align*}
where the boundary layer terms $\mathcal{B}$ are given by 
\begin{equation}
	\label{e.bdrylayer}
	\begin{aligned}
	\mathcal{B}(x) &:= \!\!\! \!\!\sum_{y \in \eta_*(U_m),\,y\sim x} \!\!\!\! \Bigl[ (\xi(x) + \xi(y))(u_b(x) - u_b(y)) + (u_b(x) + u_b(y))(\xi(x) - \xi(y)) \\ &\quad \quad \quad  + (\xi(y) - \xi(x))(u(x) + u(y))\Bigr] \,. 
	\end{aligned}
\end{equation}

\smallskip
\noindent 
From these observations we compute that 
\begin{align*}
&	\frac1{|U_m|}\sum_{x \in \eta_*(U_m)} (w(x) - v(x))\mathcal{L} w(x) \\
	&\quad = \frac1{|U_m|}\sum_{x \in \eta_*(U_m)} (w(x) - v(x))  \Bigl[ \sum_{y \in \eta_*(U_m),\,y\sim x}  \biggl(1- \frac{\xi(x)+\xi(y)}{2} \biggr)(u(x) - u(y))  - f(x) \Bigr]\\
	&\quad \quad + \frac1{|\cu_m|}\sum_{x \in \eta_*(U_m)} (w(x) - v(x))\mathcal{B}(x)\,. 
\end{align*}
As $\mathcal{B}$ is in divergence form, we can rewrite the last term by summing by parts to  through the derivatives upon $w-v.$ Rearranging, the left hand side is the squared $L^2-$ norm of the discrete gradient of $w-v,$ and since $w -v \equiv 0$ on the boundary, it follows that 
\begin{equation} \label{e.endofstep1}
	\begin{aligned}
		&\frac1{|U_m|}\sum_{x,y \in \eta_*(U_m),\,y\sim x}  (w(x) - v(x) - (w(y) - v(y)))^2\\
		&\quad \quad   \leqslant \|\nabla ((1{-}\xi)(w{-}v)\|_{\underline{L}^2(\eta_*(U_m))} \|\mathcal{L}u {-} f\|_{\underline{H}^{-1}(\eta_*(U_m))} + \|w {-} v\|_{\underline{H}^1(\eta_*(U_m))} \|\mathcal{B}\|_{\underline{H}^{-1}(\eta_*(U_m))}\,.
	\end{aligned}
\end{equation}

\smallskip

\emph{Step 2.} We estimate the size of the boundary layer terms. Toward this goal, if~$\rho : \eta_*(U_m) \to \RR$ is any test function with~$\rho|_{\partial_*^{(b)} U_m} \equiv 0,$ then testing, and recalling that~$\xi \equiv 0$ in $\cu_m \setminus \partial_*^{(2b)}U_m,$ yields 
\begin{equation*}
	\begin{aligned}
		&\frac1{|U_m|}\sum_{x \in \eta_*(U_m)} \mathcal{B}(x) \rho(x) \\
		&\quad = \frac{1}{|U_m|}\sum_{x \in \eta_*(\partial_*^{(2b)} U_m) \setminus \partial_*^{(b)}U_m} \rho(x) \sum_{y \in \eta_*(U_m),\,y\sim x}  \Bigl[ (\xi(x) + \xi(y))(u_b(x) - u_b(y))  \\ &\quad \quad \quad  + (u_b(x) + u_b(y))(\xi(x) - \xi(y)) + (\xi(y) - \xi(x))(u(x) + u(y))\Bigr]\\
		&\quad = \frac{1}{|U_m|}\sum_{x \in \eta_*(\partial_*^{(2b)} U_m) \setminus \partial_*^{(b)}U_m} \sum_{y \in \eta_*(U_m),\,y\sim x}  (\rho(x) - \rho(y)) \Biggl[ (\xi(x) + \xi(y)) u_b(x) \\
		&\quad \quad \quad  + (u_b(x) + u_b(y))\xi(x) + (u(x) + u(y))\xi(x) \Biggr]\\
		&\quad= : \mathrm{I} +\mathrm{II} + \mathrm{III}\,.
		\end{aligned}
\end{equation*}
For the first term $\mathrm{I}$, symmetrizing in $x,y,$ then using Cauchy-Schwarz,  and recalling that $|\xi|\leq 1,$ we find
\begin{equation}
	\begin{aligned}
	|\mathrm{I}|& = 	\frac{1	}{|U_m|}\Biggl|\sum_{x \in \eta_*(\partial_*^{(2b)} U_m )\setminus \partial_*^{(b)}U_m}  \sum_{y \in \eta_*(U_m),\,y\sim x} (\rho(x) - \rho(y)) (\xi(x) + \xi(y))u_b(x) \Biggr|\\
	&\leqslant C\Biggl(\frac{|\partial_*^{(2b)}U_m|}{|U_m|}\Biggr)^{\!\!\sfrac12}\|\rho\|_{\underline{H}^1(\eta_*(U_m))}  \| \nabla u_b\|_{\underline{L}^2(\eta_*(\partial_*^{(2b)}U_m))}\,.
	\end{aligned}
\end{equation}

\noindent Next, for $\mathrm{II},$ by Poincar\'e,
\begin{equation*}
	\begin{aligned}
		|\mathrm{II}| &\leqslant \frac{1}{|U_m|} \Biggl| \sum_{x \in \eta_*(\partial_*^{(2b)} U_m) \setminus \partial_*^{(b)}U_m}  \sum_{y \in \eta_*(U_m),y\sim x} |\rho(x) - \rho(y)| |\nabla \xi||u_b(x) + u_b(y)|\Biggr| \\
		&\leqslant C 3^{-b} \Biggl(\frac{|\partial_*^{(2b)}U_m|}{|U_m|}\Biggr)^{\!\!\sfrac12}\|\rho\|_{\underline{H}^1(\eta_*(U_m))} \|u_b\|_{\underline{L}^2(\eta_*(\partial_*^{(3 b)} U_m )}\,. 
	\end{aligned}
\end{equation*}

\noindent Finally, for $\mathrm{III},$ we similarly estimate 
\begin{equation*}
	\begin{aligned}
		|\mathrm{III}| &\leqslant 3^{-b} |U_m|^{-\sfrac12}\|\rho\|_{\underline{H}^1(\eta_*(U_m))} \|u\|_{H^1(\eta_*(\partial_*^{(2b)}U_m))}\,\\
		&\leqslant C3^{-b} \Biggl(\frac{|\partial_*^{(2b)}U_m|}{|U_m|}\Biggr)^{\!\!\sfrac12}\|\rho\|_{\underline{H}^1(\eta_*(U_m))} \|\uhom\|_{\underline{H}^1(\eta_*(\partial_*^{(2b)}U_m))}\,. 
	\end{aligned}
\end{equation*}

Using the triangle inequality after adding the previous three displays, and supremizing over $\rho$ with $\|\rho\|_{\underline{H}^1(\eta_*(U_m))} \leqslant 1$ yields the estimate 
\begin{align}
	\label{e.B_bound}
	\lefteqn{
	\|\mathcal{B}\|_{\underline{H}^{-1}(U_m)} 
	} \ \  &
	\notag \\ & \leqslant C\biggl(\frac{|\partial_*^{(2b)}U_m|}{|U_m|}\biggr)^{\!\!\sfrac12} \Bigl( \|u_b\|_{\underline{H}^1(\eta_*(\partial_*^{(2b)} U_m))} + \la 3^{-b} \|u_b\|_{\underline{L}^2(\eta_*(\partial_*^{(3b)} U_m ))} + 3^{-b} \|\uhom\|_{\underline{H}^1(\eta_*(\partial_*^{(2b)}U_m))} \Bigr)\,. 
\end{align}

\smallskip
\emph{Step 3.} We are ready to conclude the theorem; this involves plugging in the estimates for the boundary layer error from Step 2 in \eqref{e.endofstep1}. We notice that for the first term on the right-hand side, by the discrete product rule, 
\begin{equation*}
	\begin{aligned}
		\|\nabla ((1-\xi)(w-v))\|_{\underline{L}^2(\eta_*(U_m))} \leqslant 3^{-b}\|w-v\|_{\underline{L}^2(\eta_*(U_m))}  + \|w-v\|_{\underline{H}^1(\eta_*(U_m))}\,. 
	\end{aligned}
\end{equation*}
Combining and simplifying, using Poincar\'e's inequality,
	\begin{align*}
	\|w-v\|_{\underline{H}^1(\eta_*(U_m))} 
	& \leqslant 
	C 3^{m-b} \|\mathcal{L}u - f\|_{\underline{H}^{-1}(\eta_*(U_m) \setminus \partial_*^{(2b)} U_m)}  \\ &\qquad + C\biggl (\frac{|\partial_*^{(2b)}U_m|}{|U_m|}\biggr )^{\!\!\sfrac12} \Bigl  ( \|u_b\|_{\underline{H}^1(\eta_*(\partial_*^{(2b)} U_m))}+ 3^{-b} \|\uhom\|_{\underline{H}^1(\eta_*(\partial_*^{(2b)}U_m))} \Bigr )\,.
	\end{align*}
Now we must insert the results of the preceding two lemmas, and conclude. Observe that by the triangle inequality, the definition of~$\eta_*(U_m)$, Lemmas \ref{e.two-scale-exp}, and \ref{l.firstterm_of_2scexp}, and that $\uhom$ solves the homogenized equation, we find 
\begin{align*}
\|w-v\|_{\underline{H}^1(\eta_*(U_m))}&\leqslant C3^{m(1-\alpha)-b}  \Biggl( \sum_{x \in \eta_*(U_m )\setminus \partial_*^{(2b)} U_m} \sup_{y \in B_1(x)} |D\uhom(y)|^2 + |D^2\uhom(y)|^2\Biggr)^{\!\!\sfrac12} \\ &\quad \quad + C\biggl(\frac{|\partial_*^{(2b)}U_m|}{|U_m|}\biggr)^{\!\!\sfrac12} \Biggl( \|u_b\|_{\underline{H}^1(\eta_*(\partial_*^{(2b)} U_m))}+ 3^{-b} \|\uhom\|_{\underline{H}^1(\eta_*(\partial_*^{(2b)}U_m))} \Biggr)\,.
\end{align*}
Finally, we observe that 
\begin{equation*}
\biggl(\frac{|\partial_*^{(2b)}U_m|}{|U_m|}\biggr)^{\!\!\sfrac12} \leqslant C(d,U_0) \sqrt{\frac{3^{m(d-1)}3^b}{3^{md}} } =  C(d,U_0)3^{\frac{b-m}2} \,,
\end{equation*}
completing the argument. 
	\end{proof}

\begin{proof}[Proof of Theorem~\ref{t.Dirichlet}]
	The proof of Theorem \ref{t.Dirichlet} is an immediate consequence of Proposition~\ref{t.DirichletProblem}. Indeed, choosing $b$ so that $m(1-\alpha) - b = \frac{b-m}{2}$ so as to balance the two error contributions, and noting by the interior estimates for harmonic functions, and Caccioppoli, that 
	\begin{align*}
		\Biggl( \avsum_{x \in \eta_*(U_m )\setminus \partial_*^{(2b)} U_m} \sup_{y \in B_1(x)} |D\uhom(y)|^2 + |D^2\uhom(y)|^2\Biggr)^{\!\!\sfrac12} & \leqslant C	\Biggl( \avsum_{x \in \eta_*(U_m) \setminus \partial_*^{(2b)} U_m} \|\uhom \|_{\underline{L}^2(B_2(x))}^2\Biggr)^{\!\!\sfrac12} \\
	&	\leqslant C\|\uhom\|_{\underline{L}^2(U_m)}\,,
	\end{align*}
completes the proof.

\end{proof}

\section{Large Scale Regularity} \label{s.regularity}
Our main goal in this section is to obtain large-scale regularity for graph-harmonic functions. We begin with the following lemma that provides a quantitative harmonic approximation. With this lemma at our disposal, an immediate application of the abstract argument in \cite[Lemma 6.3]{AK22} yields large-scale $C^{0,1}$ regularity. 

In its statement, once again we work with coarsened functions. Let $R > \X$, where $\X$ is the minimal scale from Theorem \ref{t.homogenization}. We let $\mathcal{Q}$ denote the partition of $\Rd$ into \emph{good cubes}, see the discussion after \eqref{e.good}. Then, as before, for any function $u : \eta_*(B_R) \to \R$ we define its coarse-grained function $[u]_{\mathcal{Q}}$ with respect to this partition $\mathcal{Q},$ via 

\begin{equation} \label{e.Qcoarse}
	[u]_{\mathcal{Q}}(x) := \avsum_{y \in \eta_*(\cu_x)} u(y)\,,
\end{equation}
where $\cu_x \in \mathcal{Q}$ is the unique cube in the partition $\mathcal{Q}$ that contains $x.$ For brevity, oftentimes we will drop the subscript $\mathcal{Q}.$ 

\begin{lemma}
	\label{l.verifhyp}
	Let $\X$ denote the random scale from Theorem \ref{t.homogenization}, and $\mathcal{M}$ that from Lemma~\ref{l.Poincare} and suppose that $\frac12R > \max(\X,\mathcal{M}).$  
	Let \rv{$u:\eta_*(B_R)\to \R$} satisfy 
	\begin{align} \label{e.Lu = f}
		\sum_{y \in \eta_*(B_R),\,y\sim x} (u(y) - u(x)) &= f(x) \quad \mbox{ for each } x\in \rv{\eta_*(B_R) \cap i_l(B_R)} \,.
	\end{align}
Then, there exists $R_1 \in (\frac12 R,R)$, a constant $C(d,\la)<\infty,$ and there exists a harmonic function $\uhom$, i.e., a function satisfying $-\nabla \cdot \ahom \nabla \uhom = 0$ in $B_{R_1}$, such that 
\begin{equation} \label{e.harmonicapprox}
\| [u] - \uhom\|_{L^2(B_{R_1})} \leqslant CR^{-\alpha}\|[u] - ([u])_{B_R}\|_{\underline{L}^2(B_R)} + C R \|f\|_{\underline{H}^{-1}(\eta_*(B_R))}\,.
\end{equation}
\end{lemma}
{\rv{\begin{remark}
We remind the reader that the interior~$i_l(B_R)$ by analogy with~$i_l(\cu)$ at the end of Section 2. 
\end{remark}}}
\begin{proof} 
\emph{Step 0.}
Let $V$ denote the solution to \eqref{e.Lu = f} with zero boundary conditions, i.e., with $V \equiv 0$ on $\eta_*(\partial_*^{(b)} B_R).$ Then, thanks to the energy estimate and Poincar\'e's inequality (Lemma~\ref{l.Poincare}), we find that 
\begin{equation*}
	\frac1{R} \|u - V\|_{\underline{L}^2(\eta_*(B_R))} + \|\nabla u - \nabla V\|_{\underline{L}^2(\eta_*(B_R))} \leqslant C \|f\|_{\underline{H}^{-1}(\eta_*(B_R))}\,.
\end{equation*} 
Thus, in the remaining of the argument, by working with $u - V$ instead of $u,$ we may assume that $f = 0.$ 
\smallskip 

\emph{Step 1.}	
Let~$t,b,l\in\N$ with $t \geq b > l$. These represent the size of boundary layers and will be selected later in the argument. 
By the pigeonhole principle, there exists~$R_1\in [\frac12R , R-3^t]$  with the property that the energy density of~$u$ in the thin annular region~$B_{R_1+3^t}\setminus B_{R_1}$ is not larger than its average in the whole ball~$B_R$:
\begin{align} \label{e.goodboundarylayer}
\| \nabla u  \|_{\underline{L}^2(\eta_*(B_{R_1+3^t}\setminus B_{R_1} ))} 
\leqslant 
C\|\nabla u \|_{\underline{L}^2(\eta_*(B_R))}\,.
\end{align}
We mollify~$u$ on scale~$3^l,$ with $l < m,$ by setting 
\begin{equation*}
\tilde{u} := 
[u] \ast \rho_{3^l}
\,.
\end{equation*}
We view the function $\tilde{u}$ as defined in the continuum. 

\smallskip
\emph{Step 2.} 
We let~$\uhom$ denote the solution of the Dirichlet problem
\begin{align*}
\left\{
\begin{aligned}
& -\nabla \cdot \ahom \nabla \uhom = 0
& \mbox{in} & \ B_{R_1} \,
\\ &
\uhom = \tilde{u}
& \mbox{on} & \  \partial B_{R_1} \,,
\end{aligned}
\right.
\end{align*}
and, at the discrete level, we let $v$ denote the solution to 
\begin{align*}
\left\{
\begin{aligned}
&	\sum_{y\in \eta_*(B_{R_1}),\, y \sim x}  (v(x) - v(y)) = 0 \quad \mbox{ in } \mathrm{int}(\eta_*(B_{R_1}))\\ &
	v(x) = \tilde{u}(x) \quad \mbox{ in } \partial_*^{(b)} (B_{R_1})\,. 
\end{aligned}
\right.
\end{align*}
The harmonic approximation that we are after will turn out to be $\uhom,$ and to quantify this we must estimate $\|u - \uhom\|_{\underline{L}^2(\eta_*(B_{R_1}))}\,.$ By the triangle inequality, 

\begin{equation} \label{e.triangleineq}
	\|u - \uhom\|_{\underline{L}^2(\eta_*(B_{R_1}))} \leqslant \|\uhom - v\|_{\underline{L}^2(\eta_*(B_{R_1}))} + \|v - u\|_{\underline{L}^2(\eta_*(B_{R_1}))}\,.
\end{equation}
We must argue that each of the two terms on the right-hand side are small; the first will be so by homogenization, and the second by energy comparison. We detail each of these estimates in turn. 

\smallskip

\emph{Step 3.} Starting with the second term, we show that $v$ is close to $u$. To see this, since $v-u$ is a discrete harmonic function in $B_{R_1},$ it is the unique minimizer of the discrete Dirichlet energy in $B_{R_1}$ subject to its own boundary conditions. As a competitor, consider the function which is identically equal to $v-u = \tilde{u} - u$ in $\partial_*^{(b)} (B_{R_1})$, and vanishes identically in $\mathrm{int}(\eta_*(B_{R_1})).$ Comparing energies, we obtain 
\begin{align*}
	\| v - u\|^2_{\underline{H}^1(\eta_*(B_{R_1}))} &\leqslant \frac1{|B_{R_1}|}\sum_{x,y \in \eta_*(B_{R_1}),\, y\sim x}   \bigl( v(x) - v(y) - (u(x) - u(y)) \bigr)^2\\
	&\leqslant  \frac{C}{|B_{R_1}|}\sum_{x \in \partial_*^{(b)} (B_{R_1})} \sum_{y \in \eta_*(B_{R_1}),\,y\sim x}  \bigl( \tilde{u}(x) - \tilde{u}(y) - (u(x) - u(y)) \bigr)^2\,,
	\end{align*}
so that, once again by the triangle inequality, using Poincar\'e with respect to the decomposition given by $\mathcal{Q}$ (see Lemma~\ref{l.Poincare} and recall that $R_1 > \frac12 R > \mathcal{M}$), and finally, the choice of a good boundary layer from \eqref{e.goodboundarylayer}, we find 
\begin{equation}
	\label{e.pinotnoir}
	\begin{aligned}
&\lefteqn{\frac{1}{|B_{R_1}|}\sum_{x \in \partial_*^{(b)} (B_{R_1})} \sum_{y \in \eta_*(B_{R_1}),\,y\sim x}  \bigl( \tilde{u}(x) - \tilde{u}(y) - (u(x) - u(y)) \bigr)^2
}   \\
&\quad \leqslant \frac{C}{|B_{R_1}|} \sum_{x \in \partial_*^{(b)} (B_{R_1})} \sum_{y \in \eta_*(B_{R_1}),\,y\sim x}  \bigl( \tilde{u}(x) - \tilde{u}(y) - ([u](x) - [u](y)) \bigr)^2\\
&\quad + \frac{C}{|B_{R_1}|} \sum_{x \in \partial_*^{(b)} (B_{R_1})} \sum_{y \in \eta_*(B_{R_1}),\,y\sim x}  \bigl( u(x) - u(y) - ([u](x) - [u](y)) \bigr)^2\\
&\quad \leqslant  \frac{|\partial_*^{(b)} (B_{R_1})|}{|B_{R_1}|} (C 3^{-2l} + C)
\|\nabla u \|_{\underline{L}^2(\eta_*(B_R))}^2
\\ & \quad \leqslant  C
 \frac{|\partial_*^{(b)} (B_{R_1})|}{|B_{R_1}|} 
\frac1{R^2} 
\|[u] - ([u])_{B_R} \|_{\underline{L}^2(B_R)}^2
\,,
	\end{aligned}
\end{equation}
by Caccioppoli in the last line. 
\smallskip

\emph{Step 4. }  Having controlled the second term from the right-hand side of \eqref{e.triangleineq}, it remains to control the first one using homogenization. For this, appealing to Proposition~\ref{t.DirichletProblem} yields (with $m \in \NN$ the smallest natural number such that $B_{R_1} \subseteq \cu_m$)
\begin{equation} \label{e.merlot}
	\begin{aligned}
		\|v - \uhom\|_{\underline{L}^2(\eta_*(B_{R_1}))}& \leqslant C3^{m(1-\al)-b} \Biggl(\sum_{x \in \eta_*(B_{R_1})} \sup_{y \in B_1(x)} |D\uhom(y)|^2 + |D^2 \uhom(y)|^2 \Biggr)^{\!\!\sfrac12}  \\ & \qquad  + C\Biggl(\frac{|\partial_*^{(2b)}(B_{R_1})|}{|B_{R_1}|}\Biggr)^{\!\!\sfrac12}  \|D \tilde{u}\|_{\underline{L}^2(\eta_*(\partial_*^{(2b)} (B_{R_1})))} \,.
	\end{aligned}
\end{equation}

The size of $\tilde{u}$ in the boundary layer is easy to estimate, and it is where we make use of the choice of the good boundary layer. Indeed, using the definition of $\tilde{u},$ and \eqref{e.goodboundarylayer}, and that $\partial_*^{(2b)} B_{R_1} \subseteq B_{R_1 + 3^t} \setminus B_{R_1},$ we find 
\begin{align*}
\|\nabla \tilde{u}\|_{\underline{L}^2(\eta_*(\partial_*^{(2b)} B_{R_1}))}
=	\|\nabla [u]* \rho_l\|_{\underline{L}^2(\eta_*(\partial_*^{(2b)} B_{R_1}))} 
	&  \leqslant C \|\nabla u\|_{\underline{L}^2(\eta_*(B_R))} \\ & \leqslant \frac{C}{R} \| [u] - ([u])_{B_R}\|_{\underline{L}^2(B_R)}\,, 
\end{align*}
where the last inequality follows from Caccioppoli's inequality (Lemma~\ref{l.caccioppoli}). By the interior estimates for harmonic functions, since $R_1 \leqslant \frac12 R < \frac34 R,$ we find
\begin{equation*}
	\|D^2\uhom\|_{L^\infty(B_{R_1})}\leqslant \frac{C}{R}\| D\uhom\|_{L^\infty(B_{\frac34 R})} \leqslant \frac{C}{R}\|\nabla \uhom\|_{\underline{L}^2 (B_R)} \,.
\end{equation*}
Using $\tilde{u}$ as a competitor, and then using Caccioppoli once again, we find that 
\begin{equation*}
\|\nabla \uhom\|_{\underline{L}^2 (B_R)} \leqslant \|\nabla u\|_{\underline{L}^2(\eta_*(B_R))} \leqslant \frac{C}{R} \|[u] - ([u])_{B_R}\|_{\underline{L}^2(B_R)} \,.
\end{equation*}
Inserting the foregoing displays in \eqref{e.merlot}, selecting $b$ so that $m-b - m \alpha = - \frac{m\alpha}{2},$ and then choosing $l = b,$ and finally, absorbing the second term into the first upon slightly increasing the prefactor, if needed, completes the proof. 
\end{proof}

The main result of this section is 
\begin{theorem}[Large-scale~$C^{0,1}$ regularity] \label{t.large-scale.C01} Let~$\X$ be as in Theorem~\ref{t.homogenization}. There exists a constant~$C(d,\la)<\infty$ such that for every~$R \geqslant \X$,~$f: \eta_*(B_R) \to \R$, and solution~$u$ of~\eqref{e.Lu = f} satisfies, for every~$r \in [\X, \frac12 R],$ 
\begin{equation} \label{e.gradbound.lsr}
	\sup_{t \in [r, \frac12 R]} \|\nabla u\|_{\underline{L}^2(\eta_*(B_t))} \leqslant \frac{C}{R}\|u - (u)_{B_R}\|_{\underline{L}^2(\eta_*(B_R))} + C \int_r^R \|f \|_{\underline{H}^{-1}(\eta_*(B_t))} \frac{\,dt}{t}\,.
\end{equation}
In particular, in the case~$f \equiv 0,$ we have
\begin{equation} \label{e.oscillations}
\sup_{r\in[\X,\frac12R]}	|(u)_{B_r} - (u)_{B_R}| \leqslant C \|u - (u)_{B_R}\|_{\underline{L}^2(\eta_*(B_R))} \,.
\end{equation}
\end{theorem}
\begin{proof}
	The proof of Theorem \ref{t.large-scale.C01} relies on using \cite[Lemma 6.3]{AK22} in conjunction with the harmonic approximation \rv{Lemma  \ref{l.verifhyp}}. Indeed, given $R > \X$, we apply   \cite[Lemma 6.3]{AK22} with the choice $X = \X,$ and the function $[u]$ extended to be piecewise constant on the respective cubes. Naturally, $[u] \in L^2(B_R),$ and in Lemma~\ref{l.verifhyp} we showed that there is a harmonic function $\uhom$ such that $[u]$ is well-approximated by $\uhom$ in the sense of \eqref{e.harmonicapprox}. 
	
	\smallskip 
We apply the Caccioppoli inequality to obtain, for every $r \in (0,R),$  
	\begin{equation*}
		\|\nabla u\|_{\underline{L}^2(\eta_*(B_{\frac{r}{2}}))} \leqslant \frac{C}{r}\|[u] - ([u])_{B_r}\|_{\underline{L}^2(B_r)} + C \|f\|_{\underline{H}^{-1}(\eta_*(B_r))}\,.
	\end{equation*}
Next, by harmonic approximation (Lemma~\ref{l.verifhyp}) and~\cite[Lemma 6.3]{AK22}, for each $r \geqslant \X,$
\begin{equation*}
	\sup_{t \in [r, R]} \frac{1}{t} \| [u] - ([u])_{B_t}\|_{\underline{L}^2(B_t)} \leqslant \frac{C}{R} \|[u] - ([u])_{B_R}\|_{\underline{L}^2(B_R)} + C \int_r^R \| f \|_{\underline{H}^{-1}(\eta_*(B_t))} \frac{\,dt}{t}\,.
\end{equation*}
Combining these, it then follows that for each $r \in [\X, \frac{1}{2}R], $ we obtain 
\begin{equation*}
	\sup_{t \in [r, \frac12 R]} \|\nabla u\|_{\underline{L}^2 (B_t)} \leqslant  \frac{C}{R} \|[u] - ([u])_{B_R}\|_{\underline{L}^2(B_R)} + C \int_r^R \| f \|_{\underline{H}^{-1}(\eta_*(B_t))} \frac{\,dt}{t}\,;
\end{equation*}
this is the desired estimate \eqref{e.gradbound.lsr}.

\smallskip 
We turn next, to the proof of \eqref{e.oscillations}. Let $r \in [\X, \frac12R],$ and  let $n \in \N$ denote the maximum number such that $R2^{-n}  \geqslant r,$ and  we find that by the triangle inequality, and running down the scales dyadically up to the scale $r,$ we find using \eqref{e.gradbound.lsr} in the case~$f\equiv 0$, 
\begin{align*}
	|(u)_{B_r} - (u)_{B_R}| &\leqslant |(u)_{B_r} - (u)_{B_{R2^{-n}}}| +  \sum_{m=0}^n |(u)_{B_{R2^{-m+1}}} - (u)_{B_{R2^{-m}}}| \\
	&\leqslant   C \|u - (u)_{B_R}\|_{\underline{L}^2(\eta_*(B_R))} \sum_{m=0}^n \frac{1}{2^m} \leqslant C \|u - (u)_{B_R}\|_{\underline{L}^2(\eta_*(B_R))}\,.
\end{align*}
Taking a supremum concludes the proof of \eqref{e.oscillations}. 
\end{proof}

\section{Estimates on first-order correctors} 
\label{s.correctors}
The main result of this section is the existence of first-order correctors that are sublinear at infinity, and have stationary gradients, along with optimal bounds on their sizes. As in prior sections, we continue to work with a coarse-graining, and to this end we recall the partition $\mathcal{Q}$ of $\Rd$ into good cubes \rv{(see \eqref{e.good} and subsequent definition)}, and the associated coarsening operator defined via \eqref{e.Qcoarse}. The goal is to show the following bound on spatial averages of~$\nabla [\phi_e]$. Given a nonnegative smooth function~$\Psi$ with~$\int_{\Rd} \Psi = 1$ and~$r>0$, we denote
\begin{equation*}
	\Psi_r (x) :=
	r^{-d} \Psi (x/r)
	\,.
\end{equation*}

\smallskip

Thanks to the regularity theory, the existence of correctors $\phi_e$ (uniquely defined up to additive constants) is straightforward, and much of our efforts will lie in proving optimal bounds on their sizes. Granting their existence, for now, 
the optimal bounds for their sizes~\rv{(see Theorem \ref{t.corrector.bounds})} are based on applying a concentration inequality to the random vector 
\begin{equation*}
Z_e(\Psi_r)
:=
\int_{\Rd} 
\nabla [\phi_e] (x) 
\Psi_r(x)\,dx
\,.
\end{equation*}
This kind of argument, in which a concentration inequality is combined with improved regularity (or integrability of gradient) on large scales to yield optimal estimates on the first-order correctors, was introduced in the context of homogenization by Naddaf and Spencer~\cite{NS}. It later became the basis of the quantitative homogenization theory of Gloria and Otto and collaborators (see for instance~\cite{GO,GNO,DGO} and~\cite{AK22} for more references).

\smallskip

The concentration inequality we will apply can be found in~\cite[Lemma 3.3]{AK22} (see also \cite[Proposition 2.2]{AL17}). This inequality permits one to control the variance of $Z_e(\Psi_r)$ by the so-called \emph{vertical derivative} of the random variable $Z_e(\Psi_r),$  along with exponential moments, and reflects a sensitivity analysis of the point cloud on the underlying random environment. This random variable $Z_e(\Psi_r)$ depends on the point cloud~$\eta$, obviously, and so we may also write~$Z_e(\Psi,\eta)$ to indicate this dependence. 
The vertical derivative we wish to study is the difference
\begin{equation*}
	Z_e(\Psi,\eta) - Z_e(\Psi,\eta')
\end{equation*}
where~$\eta'$ is the same as~$\eta$ except all the points in a fixed unit cube~$\zeta+\cu_0$ have been deleted and resampled. 
We let~$\eta_*$ and~$\eta_*'$ denote the percolation clusters corresponding to the two point clouds~$\eta~$ and~$\eta'$, respectively. 

\smallskip

The difference~$ \phi_e -  \phi_e'$ solves the equation
with a right-hand side that looks essentially like~$\nabla \cdot \indc_{B_\X(\zeta)}$. 
We need to be more precise about what this means, since the graph has been obviously changed in the resampled unit cube, and so the two functions do not even share the same domain. However, the graph is unchanged in~$\Rd \setminus \cu$ where~$\cu\in\mathcal{P}$ is the cube in the partition~$\mathcal{P}$ which contains~$\zeta+\cu_0$. We will locally redefine $\phi_e'$ on the original graph by means of the coarsened function (which will in fact, be defined on the continuum), and then measure the sensitivity of the random variable $Z_e(\Psi_r,\eta)$ on the environment.

\begin{proof}[Proof of Theorem~\ref{t.corrector.bounds}]
The proof is structured in several steps. 

\smallskip 
\emph{Step 0.} In this step we obtain the existence of the correctors $\{\phi_e\}_{e \in \Rd}.$ Fix $e \in \Rd, e \neq 0.$ We observe that the functions $\{\phi_{m,e} - \phi_{m-1,e}\}$ are graph-harmonic in balls $B_r$ of radius $r,$ for any $r$ such that~$\X \leq r \leq 3^{m-2}$ (with $\X$ is in Theorem \ref{t.large-scale.C01}). Consequently, by \eqref{e.gradbound.lsr}, we find 
\begin{equation*}
	\| \nabla (\phi_{m,e} - \phi_{m-1,e} ) \|_{\underline{L}^2(\eta_*(B_r))}
	\leq
	C 3^{-m\alpha}
\end{equation*}
Therefore,~$\{ \nabla \phi_{m,e} \}_{m\in\N}$ is Cauchy in~$L^2(\eta_*(B_r))$ for every~$r\geq 1$. 
We can therefore define~$\nabla \phi_{e}$ as the limit, as~$m \to \infty$, of~$\nabla \phi_{m,e}$ and we have the estimate
\begin{equation*}
	\sup_{r \in [\X, 3^{m-1}]}
	\| \nabla \phi_{e} - \nabla\phi_{m,e}  \|_{\underline{L}^2(\eta_*(B_r))}
	\leq
	C 3^{-m\alpha}
	\,.
\end{equation*}
It is clear that~$\phi_{e}$, which is defined up to additive constants, is a solution of the corrector equation. 
This is stationary because we apply the regularity theory to the difference of a~$\phi_{m,e}$ and a similar function defined with respect to a cube shifted by a fixed translation, and repeat the above argument once again. 

\smallskip 
\emph{Step 1.} In this step we redefine the corrector $\phi_e'$ on the original graph, and estimate the difference. Toward this goal, we define the function~$\hat{\phi}'_e$ via
\begin{equation*}
	\hat{\phi}'_e(x)
	=
	\left\{
	\begin{aligned}
		& \phi_e'(x) & \mbox{if} & \ x \in \Rd \setminus \cu \,, \\ 
		& [\phi_e'](x) & \mbox{if} & \ x \in \cu \,.
	\end{aligned}
	\right.
\end{equation*}
We have that~$w_e := \phi_{e} - \hat{\phi}'_e$ solves the equation on the original cluster~$\eta_*$, away from the cube~$\cu$. In the cube~$\cu$ (plus neighboring vertices) it solves the equation with a right-hand side in divergence form which is bounded. 
To be precise, letting $\rho : \eta_* \to \R$ be any compactly supported test function, we find that the weak formulation of the equation satisfied by $w_e$ is
\begin{align*}
	&\sum_{x,y\in \eta_*,\, y\sim x} (w(y) - w(x)) \rho(x) \\
	&\quad =\frac12 \sum_{x,y\in \eta_*,\, y\sim x} (w(y) - w(x)) (\rho(x) - \rho(y))\\
	&\quad = \frac12\sum_{x \in\eta_*(\cu), \mathrm{dist}(x,\partial \cu) \leqslant 1} \sum_{y \not\in \eta_*(\cu),\,y\sim x} \bigl(e \cdot (x-y) - \phi_e'(y) + [\phi_e'](x)  \bigr)  (\rho(x) - \rho(y))\\
		&\quad + \frac12\sum_{x \not\in\eta_*(\cu), \mathrm{dist}(x,\partial \cu) \leqslant 1} \sum_{y \in \eta_*(\cu),\,y\sim x}  \bigl(e \cdot (x-y) - \phi_e'(x) + [\phi_e'](y)  \bigr)  (\rho(x) - \rho(y))\\
	&\quad =: \sum_{x \in \eta_*(\cu),y\in \eta_*,\, y\sim x}  \f(x,y) (\rho(x) - \rho(y))\,,
\end{align*}
where the vector field $\f$ satisfies $\f(x,y) = - \f(y,x),$ and moreover,  
\begin{equation} \label{e.fbound}
	\| \f \|_{{L}^2(\eta_*(\Rd))} \leq C\,.
\end{equation}
Using~$\hat\phi_e'$, we can write 
\begin{align} \label{e.diffofZ}
	Z_e(\Psi_r,\eta) - Z_e(\Psi_r,\eta')
	=
	\int_{\Rd} 
	\nabla\bigl ( [\phi_e] (x) - [\phi_e'] (x) \bigr )
	\Psi_r(x)\,dx
	=
	\int_{\Rd} 
	\nabla [w]  (x) 
	\Psi_r(x)\,dx
\end{align}
We can write a formula for $w$ in terms of the Green's function:
\begin{equation} \label{e.formula}
	w_e(x) = \sum_{z \in \eta_*} G(x,z) \sum_{y \in \eta_*,\, y\sim z} \f(z,y) \,,
\end{equation}
Since $\f$ is antisymmetric in its arguments, we find 
\begin{equation} \label{e.repformula}
	w(x) = \sum_{z \in \eta_*} \sum_{y \in \eta_*,\,y\sim z }  \f(z,y) \bigl( G(x,z) - G(x,y)\bigr)\,. 
\end{equation}
Towards estimating the integral on the right-hand side of \eqref{e.diffofZ}, we recall that $\zeta$ is the center of the cube $\cu,$ and let $\X(\zeta)$ denote the minimal scale from Proposition~\ref{p.greensfunction} below, and estimate our integral by splitting it in $B(\zeta, 3\X(\zeta))$, and $\Rd \setminus B(\zeta, 3\X(\zeta)).$ 

\smallskip 
If $x \in \Rd \setminus B(\zeta, 3\X(\zeta)),$ we let $R(x) := \frac12 |x - \zeta|,$ and let $\cu_{x} \in \mathcal{P}$ denote the unique cube in the partion $\mathcal{P}$ which contains $x.$ Then, by the large-scale regularity \rv{T}heorem \ref{t.large-scale.C01} we estimate 
\begin{align*}
	|\nabla [w](x)|&\leqslant |\cu_x|^{\sfrac12} \|\nabla [w]\|_{\underline{L}^2(\cu_x)} \\
	&\leqslant  \frac{C}{R(x)}|\cu_x|^{\sfrac12} \|[w] - ([w])_{B(x,R(x))}\|_{\underline{L}^2 (B(x, R(x)))}\\
	&\leqslant  \frac{C}{|x- \zeta|}|\cu_x|^{\sfrac12}\bigl \|[w] - ([w])_{A(\zeta, \frac12 R(x), \frac32 R(x))}\bigr \|_{\underline{L}^2 (A(\zeta, \frac12 R(x), \frac32 R(x)))}\,,
\end{align*}
where $A(\zeta, \frac12 R(x), \frac32 R(x))$ is the annulus with inner and outer radii $\frac12 R(x)$ and $\frac32 R(x),$ respectively, centered at $\zeta.$ Combining the representation formula \eqref{e.repformula} with the bound \eqref{e.fbound}, using the large-scale regularity statement in Theorem~\ref{t.large-scale.C01} on the Greens function, and then using the decay from Proposition~\ref{p.greensfunction} with the center $\zeta$ in place of $x,$ we obtain that 
\begin{equation*}
	|\nabla [w](x)| \leqslant \frac{C}{|x - \zeta|} |\cu_x|^{\sfrac12} \frac{\X(\zeta)^{\sfrac{d}2}}{|x - \zeta|^{d-1}} = \frac{C |\cu_x|^{\sfrac{1}2} \X(\zeta)^{\sfrac{d}2}}{|x - \zeta|^d}\,.
\end{equation*}
Thus, we find 
\begin{equation*}
	\Biggl|	\int_{\Rd \setminus B(\zeta, 3\X(\zeta))} \nabla [w](x) \Psi_r(x)\,dx \Biggr| \leqslant \int_{|x-\zeta| \geqslant 3 \X(\zeta)} C|\cu_x|^{\sfrac12} \frac{\X(\zeta)^{\sfrac{d}2}}{|x-\zeta|^d}\Psi_r(x)\,dx\,.
\end{equation*}
On the other hand, for the ``near field'', we observe that if $x \in B(\zeta,3\X(\zeta)),$ then, thanks to \eqref{e.sizeofcubes}, we estimate 
\begin{align*}
	|\nabla [w](x)| &\leqslant |\cu_x|^{\sfrac12} \|\nabla [w]\|_{\underline{L^2}(\cu_x)} \\
	& \leqslant \O_1(C) \bigl(\|\nabla [\phi_e]\|_{\underline{L}^2(B(\zeta,3\X(\zeta)))} + \|\nabla [\phi_e']\|_{\underline{L}^2(B(\zeta,3\X(\zeta)))} \bigr)	\leqslant \O_1(C)\,. 
\end{align*}
Combining the last two displays, and upon increasing the constant $C$ if needed, by \eqref{e.diffofZ}
\begin{equation} \label{e.vertder}
	|\partial_\zeta Z_e(\Psi_r)| = \Biggl|	\int_{\Rd} \nabla[w] (x) \Psi_r(x)\,dx \Biggr| \leqslant \int_{\Rd} \frac{C|\cu_x|^{\sfrac12} \X(\zeta)^{\sfrac{d}2}}{1 + |x-\zeta|^d}\Psi_r(x)\,dx\,. 
\end{equation}

\smallskip 

\emph{Step 2.} We are now ready to apply a nonlinear concentration inequality.  Applying \cite[Proposition 2.2]{AL17}, there exists a constant $C(d)<\infty$ such that, for every~$E>0$,  
\begin{equation}
	\label{e.exp.spectral.gap}
	\E
	\Biggl[ 
	\exp\biggl  (\biggl  (\frac {\sum_{z\in \Zd}|\partial_z Z_e(\Psi)|^2 }{E^2}\biggr )^{\!\!\frac {s}{2-s}} \biggr )
	\Biggr] \leq 2
	\quad \implies \quad
	\E
	\biggl[ 
	\exp\biggl  (\biggl  (\frac {|Z_e(\Psi)|}{CE}\biggr )^{s} \biggr )
	\biggr]
	\leq 
	2\,.
\end{equation}
Squaring  \eqref{e.vertder}, we estimate
\begin{align} \label{e.thingy}
	\sum_{\zeta\in\Zd}
	|\partial_\zeta Z_e(\Psi_r)|^2
	&
	\leq
	C
	\sum_{\zeta\in\Zd}
	\X(\zeta)^{2d}
	\biggl (\int_{\Rd} 
	\frac{|\cu_x|}{1 + |x - \zeta|^d}
 \Psi_r(x)\,dx\biggr )^{\!2}\,.
\end{align}
Thanks to \eqref{e.minscalesize}, we note that 
\begin{equation*}
	\int_{\Rd} 
 \frac{1}{1 + |x-\zeta|^d}
	|\cu_x| \Psi_r(x)\,dx
	= \O_{\frac{d}{2(d+1)}}\biggl  (C\int_{\Rd} 
 \frac{1}{1 + |x-\zeta|^d}
	\Psi_r(x)\,dx\biggr ) \,.
\end{equation*}
Since $\X = \O_1(C),$ it follows  by \cite[Lemma A.3]{AKM19} that 
\begin{equation*}
	\X(\zeta)^{2d}
	\biggl (\int_{\Rd} 
 \frac{1}{1 + |x-\zeta|^d}
	|\cu_x| \Psi_r(x)\,dx\biggr )^{\!2}
	= \O_{\frac{d}{3d+2}}\biggl  (C \biggl (\int_{\Rd} 
 \frac{1}{1 + |x-\zeta|^d}
	\Psi_r(x)\,dx\biggr )^{\!2}\,\biggr ) \,.
\end{equation*}
Summing this estimate over $\zeta \in \Zd$ yields 
\begin{equation*}
	\sum_{\zeta\in\Zd}\X(\zeta)^{2d}
	\biggl (\int_{\Rd} 
 \frac{1}{1 + |x-\zeta|^d}
|\cu_x|\Psi_r(x)\,dx\biggr )^{\!2}
	=  \O_{\frac{d}{3d+2}}\biggl  (C \sum_{\zeta \in \Zd} \biggl (\int_{\Rd} 
	\frac{1}{1 + |x-\zeta|^d}
	\Psi_r(x)\,dx\biggr )^{\!2}\,\biggr ) \,.
\end{equation*}
By Young's inequality for convolutions, we finally compute that
\begin{equation*}
	\sum_{\zeta \in \Zd} \biggl (\int_{\Rd} 
	\frac{1}{1 + |x-\zeta|^d}
	\Psi_r(x)\,dx\biggr )^{\!2} \leqslant C \int_{\Rd} \biggl(\int_{\Rd} 	\frac{1}{1 + |x-\zeta|^d}
	\Psi_r(x)\,dx \biggr)^{\!2} \,d\zeta
	\leqslant C r^{-d}\,,
\end{equation*}
so that, from \eqref{e.thingy} we have 
\begin{equation*}
	\sum_{\zeta\in\Zd}
	|\partial_\zeta Z_e(\Psi_r)|^2 =
	\O_{\frac{d}{3d + 2}} ( Cr^{-d})\,.
\end{equation*}
Applying~\eqref{e.exp.spectral.gap} lets us conclude
\begin{equation*}
Z_e(\Psi_r)  =\O_{\frac{2d}{4d+2}} ( Cr^{-\sfrac{d}2})\,.
\end{equation*}

\smallskip

\emph{Step 3.} We are ready to conclude the proof of the theorem, by proving the optimal sizes of the correctors, namely \eqref{e.correctorsize.opt}.  For this, we make special choices of the test function $\Psi$ in terms of the parabolic heat kernel in $\Rd$, and then appeal to well-known characterizations of Sobolev spaces in terms of convolution against the heat kernel.  To be precise, fixing $y \in \Rd$, we use~$\Psi_r(x) := (4\pi r^2)^{-\sfrac{d}2} \exp \Bigl( -\frac{|x-y|^2}{4r^2}\Bigr)\,, $ in the preceding steps.  Following \cite[Proof of Theorem 4.24]{AKM19}, we fix $\ep \in (0,\frac12],$ and $r \geqslant \ep,$ and for any $x \in \Rd,$ the prior steps imply that 
\begin{equation*}
	|(\g_\ep * \Psi_1(r,\cdot))(x)|\leqslant \O_s \Bigl( C (\ep r^{-1})^{\frac{d}2} \Bigr)\,. 
\end{equation*}
Proceeding, then, identically as in the proof of \cite[Proof of Theorem 4.24]{AKM19} completes the argument. 
\end{proof}

 We conclude with the proof of Theorem~\ref{t.homogenization.intro}.
 
\begin{proof}[Proof of Theorem~\ref{t.homogenization.intro}] The proof of Theorem~\ref{t.Dirichlet} essentially contains a significant bulk of the proof of Theorem~\ref{t.homogenization.intro}: in that argument we used the suboptimal algebraic rate on the sizes of correctors that we needed, in order to develop the large-scale regularity theory. The difference here is that we use the optimal bounds on the correctors, which are provided by Theorem~\ref{t.corrector.bounds}. Note that~\eqref{e.correctorsize.opt} implies that 
\begin{equation} \label{e.correctorsize.opt.nograd}
		\Bigl\|[\phi_e]\Bigl(\frac{\cdot}{\varepsilon}\Bigr)\Bigr\|_{L^{2}(B_1)} \leqslant 
		\begin{cases}
			\O_{s} (C)  & \mbox{if} \ d>2\,.
			\\
			\O_s\bigl( C |\log \ep|^{\frac12}\bigr) &  \mbox{if} \ d=2\,.
		\end{cases}
	\end{equation}
We can then use \rv{these} bounds in the two-scale expansion in Proposition~\ref{t.DirichletProblem}. Finally, in order to complete the proof of Theorem~\ref{t.Dirichlet}, it remains to make a careful boundary layer analysis, and we refer the reader to~\cite[Section 6.4]{AKM19} for details on how to execute this. It is this step which requires that $U_0$ is $C^{1,1}$ or convex, and is directly applicable here, since it only makes use of a delicate boundary layer estimate for the \emph{homogenized} equation (i.e., Poisson's equation). 
\end{proof}

\appendix

\section{Estimates on the Green's function} 

Let $G$ denote the Green's function associated to the graph Laplacian. Explicitly, for any $x \in \eta_*, $ let $G_x(\cdot) : \eta_* \to (0,\infty)$ be the unique solution to 
\begin{align*}
	\sum_{y \in \eta_*,\,y\sim z}  \bigl( G_x(y) - G_x(z) \bigr) =\left\{
	\begin{array}{cc}
		1 & z = x\\
		0 & z \neq x\,.
	\end{array}
	\right.
\end{align*}
It is convenient to write $G(x,z) := G_x(z).$ It can be shown that $G(x,z) = G(z,x)$ (and this is proven exactly as in the continuum).  The key property of the Green's function is the following. If $f : \eta_* \to \RR$ is a compactly supported function, then the unique $H^1(\eta_*)$ solution to 
\begin{equation*}
	\sum_{y \in \eta_*,\,y\sim x} (u(y) - u(x)) = f(x) \quad \mbox{ in } x \in \eta_*\,,
\end{equation*}
is given by the representation formula 
\begin{equation*}
	u(x) = \sum_{z \in \eta_*} G(x,z) f(z)\,. 
\end{equation*}
One can check that this is true by direct computation. 

Thus the formula above gives a solution to the preceding equation: since this equation has a variational formulation that is strictly convex in the gradient, it has a unique minimizer. So the associated equation on the graph has a unique solution, as given by the representation formula. 

Observe that the Green's function $G$ is a random object, though we do not indicate the dependence on $\eta$ in the notation. We will need the following bounds on the Green's function; in the statement of the following proposition, given $x \in \eta_*, $ and $0 < r_1 < r_2 \in \R,$  
\begin{equation*}
	A(x,r_1,r_2) := \{y \in \eta_* : r_1 < |y - x| < r_2\}\,. 
\end{equation*}

\begin{proposition} \label{p.greensfunction}
	 Let $\X$ denote the minimal scale from Theorem \ref{t.homogenization}. Then there exists a constant $C$  depending only on $d$ (and not on the particular instance $\eta_*$), such that 
	\begin{align*}
	\Biggl(\avsum_{z \in \eta_*(A(x,\frac{r}3, \frac{2r}3))} \Biggl( G(y,z)  -  \avsum_{w \in \eta_*(A(x,\frac{r}3, \frac{2r}{3}))} G(y,w) \Biggr)^{\!\!2}\,\Biggr)^{\!\!\sfrac12} \leqslant \frac{C\X^{\sfrac{d}2}}{|x-y|^{d-2}} \,. 
	\end{align*}
	for all $x,y \in \eta_*$ which are such that $r := |x-y| > 3\X.$
\end{proposition} 
\begin{proof}
The proof follows very closely the associated argument in the continuum from \cite[Theorem 13]{AL}. 
Let $ x\neq y \in \eta_*,$ and set $r := |x-y|.$ We let $f : \eta_*(B_r(y)) \to \RR$ be such that $f (z) \equiv 0$ if $|z - y| \geqslant \frac{r}{3}.$   From the prior discussion, $u$ defined via 
\begin{equation*}
	u(z) := \sum_{w \in \eta_*} G(z,w) f(w)\,,
\end{equation*}
is the unique solution to the problem 
\begin{align*}
	\sum_{w \in \eta_*,\, w\in z}  (u(w) - u(z)) &= f(z)\quad z \mbox{ in } \eta(B_r(y))\,,
\end{align*}
which vanishes at infinity. 
As $f$ vanishes in $\eta_*(A(x,\frac{r}3,\frac{2r}{3}),$ large-scale regularity theory of $u$ is applicable. Thus, applying Theorem \ref{t.large-scale.C01} (see Eq. \eqref{e.oscillations}) in the ball $A(x,\frac{r}3,\frac{2r}{3})$, yields
\begin{align*}
|u(x) - (u)_{A(x,\frac{r}3,\frac{2r}{3})} | 
&
\leq 
\X^{\sfrac d2}
\|u - (u)_{A(x,\frac{r}3,\frac{2r}{3})}\|_{\underline{L}^2(\eta_*(B(x,\X)))}
\\
&\leqslant 
C\X^{\sfrac d2}
\|u - (u)_{A(x,\frac{r}3,\frac{2r}{3})}\|_{\underline{L}^2(\eta_*(A(x,\frac{r}3,\frac{2r}{3})))}
\,.
\end{align*}
It follows by the triangle inequality first, then by H\"{o}lder's inequality, and then Sobolev's imbedding theorem (applied to the coaresened function) and Poincar\'e's inequality, that  
\begin{align*}
&\Biggl|\avsum_{z \in \eta_*(A(x,\frac{r}3,\frac{2r}{3}))} \Biggl( G(y,z)  -  \avsum_{w \in \eta_*(A(x,\frac{r}3,\frac{2r}{3}))} G(y,w) \Biggr) f(z) \Biggr| \\ &
\quad \quad  = \frac{1}{r^d}|u(x) - (u)_{A(x,\frac{r}3,\frac{2r}{3})} | 
\leqslant C\frac{\X^{\sfrac d2}}{r^d}
\|u - (u)_{A(x,\frac{r}3,\frac{2r}{3})}\|_{\underline{L}^2(\eta_*(A(x,\frac{r}3,\frac{2r}{3})))}
\,.
\end{align*}
Now, we estimate by the triangle inequality that 
\begin{equation*}
	\|u - (u)_{A(x,\frac{r}3,\frac{2r}{3})}\|_{\underline{L}^2(\eta_*(A(x,\frac{r}3,\frac{2r}{3})))} \leqslant \|[u] - (u)_{A(x,\frac{r}3,\frac{2r}{3})}\|_{\underline{L}^2(A(x,\frac{r}3,\frac{2r}{3}))} +  \|u - [u]\|_{\underline{L}^2(\eta_*(A(x,\frac{r}3,\frac{2r}{3})))}  \,.
\end{equation*}
For the first term on the righthand side, by H\"{o}lder inequality, and then the Poincar\'e, we find that if $\frac{r}{3} > \X,$ then 
\begin{align*}
 \|[u] - (u)_{A(x,\frac{r}3,\frac{2r}{3})}\|_{\underline{L}^2(A(x,\frac{r}3,\frac{2r}{3}))}  
 &\leqslant \frac{C}{r^{\frac{d-2}{2}}} \|\nabla [u] \|_{{L}^{2}(A(x,\frac{r}3,\frac{2r}{3}))} \\
 &\leqslant \frac{C}{r^{\frac{d}{2}-1}}\|\nabla u\|_{{L}^{2}(\eta_*(A(x,\frac{r}3,\frac{2r}{3})))} \leqslant \frac{C}{r^{\frac{d}{2}-2}}\|f\|_{{L}^{2}(\eta_*(A(x,\frac{r}3,\frac{2r}{3})))}\,.
\end{align*}
For the second term in the triangle inequality we find 
\begin{align*}
	\|u - [u]\|_{\underline{L}^2(\eta_*(A(x,\frac{r}3,\frac{2r}{3})))} &\leqslant \frac{C}{|B_r|^{\sfrac12}} \sum_{\cu \in \mathcal{P}; \cu \cap A(x,\frac{r}3,\frac{2r}{3})) \neq \emptyset} \|u - [u]\|_{L^2(\eta_*(\cu))} \\
	&\leqslant \frac{C}{|B_r|^{\sfrac12}} \sum_{\cu \in \mathcal{P}; \cu \cap A(x,\frac{r}3,\frac{2r}{3}) \neq \emptyset} \size(\cu) \|\nabla u\|_{L^2(\eta_*(\cu))}\\
	&\leqslant \frac{C\O_1(C)}{|B_r|^{\sfrac12}} \sum_{\cu \in \mathcal{P}; \cu \cap A(x,\frac{r}3,\frac{2r}{3}) \neq \emptyset} \size(\cu) \|\nabla u\|_{L^2(\eta_*(\cu))}\\
	&= \O_1(C) \|\nabla u\|_{\underline{L}^{2}(\eta_*(A(x,\frac{r}3,\frac{2r}{3})))} \leqslant \O_1(C)r \|f\|_{\underline{L}^{2}(\eta_*(A(x,\frac{r}3,\frac{2r}{3})))}
	\,.
\end{align*}
Combining the last three displays, we find that 
\begin{equation*}
	\Biggl|\avsum_{z \in \eta_*(A(x,\frac{r}3,\frac{2r}{3}))} \Biggl( G(y,z)  -  \avsum_{w \in \eta_*(A(x,\frac{r}3,\frac{2r}{3}))} G(y,w) \Biggr) f(z) \Biggr| \leqslant C \X^{\sfrac d2} \frac{1}{r^{d-2}} \|f\|_{\underline{L}^{2}(\eta_*(A(x,\frac{r}3,\frac{2r}{3})))}\,.
\end{equation*}
Taking the supremum among functions $f$ such that $\|f\|_{\underline{L}^{2}(\eta_*(A(x,\frac{r}3,\frac{2r}{3})))} \leqslant 1$ completes the argument.
\end{proof}

{\small
\bibliographystyle{abbrv}
\bibliography{homogenization}
}

\end{document}